\documentclass[11pt]{article}
\usepackage{amstext,amssymb,amsmath,amsbsy}
\usepackage{hyperref}
\usepackage{amscd} 
\usepackage{amsfonts,esint}
\usepackage{indentfirst}
\usepackage{verbatim}
\usepackage{amsmath}
\usepackage{amsthm}
\usepackage{enumerate}
\usepackage{graphicx}
\usepackage{subfig}
\usepackage{color}
\usepackage[OT1]{fontenc}
\usepackage[latin1]{inputenc}
\usepackage[english]{babel}
\usepackage{amssymb}

\makeatletter
\DeclareRobustCommand\widecheck[1]{
	{\mathpalette\@widecheck{#1}}}
\def\@widecheck#1#2{
		\setbox\z@\hbox{\m@th$#1#2$}
	\setbox\tw@\hbox{\m@th$#1%
		\widehat{%
			\vrule\@width\z@\@height\ht\z@
			\vrule\@height\z@\@width\wd\z@}$}%
	\dp\tw@-\ht\z@
	\@tempdima\ht\z@ \advance\@tempdima2\ht\tw@ \divide\@tempdima\thr@@
	\setbox\tw@\hbox{%
		\raise\@tempdima\hbox{\scalebox{1}[-1]{\lower\@tempdima\box
				\tw@}}}%
	{\ooalign{\box\tw@ \cr \box\z@}}}
\makeatother
\newtheorem{theorem}{Theorem}
\newtheorem{lemma}{Lemma}

\newtheorem{proposition}{Proposition}
\newtheorem{corollary}{Corollary}
\newtheorem{definition}{Definition}
\theoremstyle{remark}
\newtheorem{remark}{Remark}

\numberwithin{equation}{section}

\usepackage{dsfont}

\newcommand{\supp}{\operatorname{supp}}

\newcommand{\loc}{_{loc}}

\title{ Quantitative estimates for regular Lagrangian flows with $BV$ vector fields}
\author{
	\textbf{Quoc-Hung Nguyen} \thanks{Address: quochung.nguyen@sns.it,qn2@nyu.edu,   Scuola Normale Superiore,  Piazza dei Cavalieri 7, I-56100 Pisa, Italy.
Current address: qhnguyen@shanghaitech.edu.cn, ShanghaiTech University,393 Middle Huaxia Road, Pudong, Shanghai, 201210, China.}}
\date{}

\begin{document}
\maketitle 
\begin{abstract}
 This paper is devoted to the study of flows associated to non-smooth vector fields. We prove the well-posedness of regular Lagrangian flows associated to vector fields $\mathbf{B}=(\mathbf{B}^1,...,\mathbf{B}^d)\in L^1(\mathbb{R}_+;L^1(\mathbb{R}^d)+L^\infty(\mathbb{R}^d))$ satisfying $
\mathbf{B}^i=\sum_{j=1}^{m}\mathbf{K}_j^i*b_j,$ $b_j\in L^1(\mathbb{R}_+,BV(\mathbb{R}^d))$
and $\operatorname{div}(\mathbf{B})\in L^1(\mathbb{R}_+;L^\infty(\mathbb{R}^d))$ for $d,m\geq 2$,  where   $(\mathbf{K}_j^i)_{i,j}$ are singular kernels in $\mathbb{R}^d$. Moreover, we also show that there exist  an autonomous vector-field $\mathbf{B}\in L^1(\mathbb{R}^2)+L^\infty(\mathbb{R}^2)$ and  singular kernels $(\mathbf{K}_j^i)_{i,j}$, singular Radon measures  $\mu_{ijk}$ in $\mathbb{R}^2$  satisfying $\partial_{x_k}
\mathbf{B}^i=\sum_{j=1}^{m}\mathbf{K}_j^i\star\mu_{ijk}$ in distributional sense   
for some $m\geq 2$  and  for $k,i=1,2$ such that regular Lagrangian flows  associated to vector field $\mathbf{B}$  are not unique.  \medskip

\end{abstract}

\tableofcontents

\section{Introduction}	
In this paper we study the well-posedness of flows of ordinary differential equations
\begin{eqnarray}\label{odeeq1}
\left\{ \begin{array}{rcl}
\frac{d X(t,x)}{dt}&=&\mathbf{B}(t,X(t,x)),~~~\forall~t\in [0,T]\\
X(0,x)&=& x,~~\forall~x\in\mathbb{R}^d,
\end{array}\right.
\end{eqnarray} 	
where $\mathbf{B}(t,x)=\mathbf{B}_t(x)\in \mathbb{R}^d$ is a function in $[0,T]\times\mathbb{R}^d$, $d\geq 2$. It is well known that by Peano's Theorem, there exists at least one solution to the problem \eqref{odeeq1} provided that $\mathbf{B}$ is continuous. Moreover, by the usual Cauchy-Lipschitz Theorem, one has also uniqueness if $\mathbf{B}$ is a bounded smooth vector field. 

The  ordinary differential equation \eqref{odeeq1} is related to the continuity equation 
\begin{eqnarray}\label{contieq}
\left\{ \begin{array}{rcl}
\partial_t u(t,x)+\operatorname{div}\left(\mathbf{B}(t,x)u(t,x)\right)&=& G(t,x)u(t,x)+F(t,x), \\
u(0,x)&=& u_0(x),
\end{array}\right.
\end{eqnarray} 	
for any $(t,x)\in [0,T]\times\mathbb{R}^d$. 
Indeed, assume that $u_0,\mathbf{B},G$ and $F$ are smooth and compactly supported. Let $X:[0,T]\times\mathbb{R}^d\to \mathbb{R}^d$ be the unique solution of \eqref{odeeq1}. It is called the flow of vector field $\mathbf{B}$. We have 
$$
\det(\nabla_x X(t,x))= \exp\left(\int_{0}^{t}\operatorname{div}(\mathbf{B})(s,X(s,x))ds\right)> 0.
$$
 In particular, the map $X(t,.)$ is a diffeomorphism from $\mathbb{R}^d$ to itself and we denote by $X^{-1}(t,.)$ its inverse. A solution of \eqref{contieq} is given in term of the flow $X$ by the following formula
\begin{align}\nonumber
u(t,x)&=u_0(\overline{x})\exp\left(-\int_{0}^{t}\left(\operatorname{div}(\mathbf{B})-G\right)(s,X(s,\overline{x}))ds\right)\\&+\int_{0}^{t}F(\tau,X(\tau,\overline{x})) \exp\left(-\int_{\tau}^{t}\left(\operatorname{div}(\mathbf{B})-G\right)(s,X(s,\overline{x}))ds\right)d\tau,\label{formu}
\end{align}
with $\overline{x}=X^{-1}(t,.)(x)$, its
proof is elementary. Therefore, we can say that \textit{ the well posedness of \eqref{odeeq1}  is equivalent to the well-posedness of \eqref{contieq}}.

The continuity equations (often with non-smooth vector fields) are important for describing various quantities in mathematical physics such as mass, energy, momentum, electric charge. Especially, they are essential to study  transport equations as such the convection-diffusion, Boltzmann,  Vlasov-Poisson, Euler and Navier-Stokes equations.

 Let us start by the seminal work of Diperna and Lions \cite{Diperlions}. They established the existence, uniqueness and stability of distributional solutions of   \eqref{contieq}  for   vector fields $\mathbf{B}$ in $L^1_tW^{1,1}_x$ satisfying $\operatorname{div}(\mathbf{B})\in L^1_tL^\infty_x$ and a growth condition $\mathbf{B}/(1+|x|)\in L^1_tL^1_x+L^1_tL^\infty_x$. Later some progress was achieved in several papers \cite{lion98,Bo01,BoJa,Hau,ColoLerner1,ColoLerner2}, finally  it was fully extended by Ambrosio \cite{luigiinvent} to BV vector fields. The  approach by Diperna, Lions and Ambrosio relies on the theory of renormalized solutions of \eqref{contieq}. Roughly speaking renormalized solutions are distributional solutions such that  the chain rule holds for $u$ and $\mathbf{B}$ i.e 
$$
\operatorname{div}(\mathbf{B}h(u))=\left(h(u)-u h'(u)\right) \operatorname{div}(\mathbf{B})+h'(u)\operatorname{div}(\mathbf{B}u)$$
for any $h\in C^1(\mathbb{R})$.\\
A main point of this approach is  proving the strong convergence of the commutator 
$$
r_\delta:=\rho_\delta\star(\operatorname{div}(\mathbf{B}u))-(\operatorname{div}(\mathbf{B} \rho_\delta\star u))
$$
to $0$ in $L^1_{\text{loc}}$ for  some regularizing kernel $(\rho_\delta)_{\delta>0}$  in $\mathbb{R}^d$. In the Sobolev case, in  \cite{Diperlions}, Diperna and Lions showed this convergence for any regularizing kernel $(\rho_\delta)_{\delta>0}$. The same problem in the $BV$ case is much more complicated. In \cite{luigiinvent} Ambrosio took  a special  kernel $\rho$ strictly depending on the structure of $\mathbf{B}$ to obtain the convergence. More precisely, first he proved that 
$$|r_\delta|\rightharpoonup \sigma,~~~\text{and}~~~\sigma(x)\lesssim \int\left|\langle M(x) z, \nabla\rho(z)\rangle\right|dz |D^s \mathbf{B}|(x),$$
with $M(x)=\frac{dD^{\text{s}}\mathbf{B}}{d|D^{\text{s}}\mathbf{B}|}(x)$ for any smooth kernel $\rho$, $\int\rho(z) dz=1$ for any $x\in \mathbb{R}^d$, where $D^{\text{s}}\mathbf{B}$ is  singular part of $D\mathbf{B}$ with respect to the Lebesgue measure. Then, he took  $\rho$ such that $$
\int\left|\langle M(x) z, \nabla\rho (z)\rangle\right|dz\lesssim |trace M(x)|.$$  Using the fact that  $\operatorname{div}(\mathbf{B})<<\mathcal{L}^d$ ~$\Longleftrightarrow$~ $|trace M(x)| |D^s \mathbf{B}|(x) =0$, then he got the "defect" measure $\sigma=0$.  \medskip\\
 Moreover, Diperna and Lions constructed distributional solutions to the continuity equation \eqref{contieq}  with $\mathbf{B}\in W^{\alpha,1}$ ($\alpha<1$) and $\operatorname{div}(\mathbf{B})=0$ that are not renormalized. A counterexample for non-BV is provided by Depauw \cite{Depauw}. Further results can be found in \cite{AmDelelissMa,Delellis,AMF1,AMF2,bianBoni1,bianBoni2,HauLebris,comcrip,AmLeMa,CCS1,CCS2,CripLiga,BCL,BN18a,BN18b,BN18c,BN18d}. For a recent review on the well-posedness theories for the continuity equations \eqref{contieq} and ODE \eqref{odeeq1}, we refer the reader the lecture notes \cite{AmbCrip} (and \cite{Am-Rev17}).\medskip\\
 In \cite{CripLellis}, C. De Lellis and G. Crippa gave an independent  proof of the existence and uniqueness of the solutions of \eqref{odeeq1} with Sobolev vector fields,  that is without exploiting the connection with the continuity equations  \eqref{contieq}. The basic idea of \cite{CripLellis} is to consider  the following time dependent quantity $$
 \Phi_\delta(t)=\int_{B_R}\log\left(1+\frac{|X_1(t,x)-X_2(t,x)|}{\delta}\right) dx~~\forall~\delta\in(0,1/2),$$
 where $X_1,X_2$ are regular Lagrangian flows associated to the same vector field $\mathbf{B}$ and $B_R:=B_R(0),R>0$. We have 
 \begin{align}\label{inesPhi1}
 \Phi_\delta(t)\geq \mathcal{L}^d\left(\left\{x\in B_R :|X_1(t,x)-X_2(t,x)|>\delta^{1/2}\right\}\right)\log\left(1+\delta^{-1/2}\right).
 \end{align}
 However, differentiating in time, one has
 \begin{align}\nonumber
 \Phi_\delta(t)=\int_{0}^{t}\Phi_\delta'(s)ds&\leq \int_{0}^{t}\int_{B_R}\frac{|\mathbf{B}_s(X_1(s,x))-\mathbf{B}_s(X_2(s,x))|}{\delta+|X_1(s,x)-X_2(s,x)|}dx ds\\& \leq \int_{0}^{t}\int_{B_R}\min\left\{\frac{2||\mathbf{B}_s||_{L^\infty}}{\delta},\frac{|\mathbf{B}_s(X_1(s,x))-\mathbf{B}_s(X_2(s,x))|}{|X_1(s,x)-X_2(s,x)|}\right\} dx. \label{differeintime}
 \end{align}
 By the standard estimate of the Hardy-Littlewood  function 
$\mathbf{M}$ and 
 changing variable along the flows,  we obtain 
 \begin{align}\label{inesPhi2}
 \Phi_\delta(t)  \lesssim\int_{0}^{T}\int_{B_{R_1}}\min\left\{\delta^{-1},\mathbf{M}(|\nabla \mathbf{B}_s|)(x)\right\} dx ds.
 \end{align}
 for some $R_1>R$. 
 Using the boundedness of $\mathbf{M}$ from $L^p$ to itself for $p>1$ together with  \eqref{inesPhi1} and \eqref{inesPhi2}, we deduce  that
$$
  \mathcal{L}^d\left(\left\{x\in B_R :|X_1(t,x)-X_2(t,x)|>\delta^{1/2}\right\}\right)\lesssim|\log(\delta)|^{-1}~~\forall\delta\in (0,1/2)$$
 provided $\mathbf{B}\in L^1(W^{1,p}),p>1$. At this point, sending $\delta\to 0$, we get $X_1=X_2$.
 
 Later, in \cite{Jabin} P.E. Jabin successfully improved  this to $\mathbf{B}\in L^1_t(W_x^{1,1})$. Besides, also  in \cite{Jabin} he extends this to $\mathbf{B}\in L^1(SBV)$ in any dimension, and in two-dimension to $L^1(BV)$ with local assumptions in the direction of flows. Furthermore, in \cite{BHS} we  showed that 
 \begin{align*}
 \int_{0}^{T}|D^s\mathbf{B}_t|(B_{R_1})dt\lesssim\limsup_{\delta\to 0}\frac{1}{|\log(\delta)|}\int_{0}^{T}\int_{B_{R_1}}\min\left\{\delta^{-1},\mathbf{M}(|\nabla \mathbf{B}_s(.)|)\right\}\lesssim \int_{0}^{T}|D^s\mathbf{B}_t|(\overline{B}_{R_1})dt.
 \end{align*}
 Therefore, this is reason that De Lellis and Crippa's  approach is not able to deal with vector fields $\mathbf{B}\in L^1(BV\backslash W^{1,1})$. Moreover,
   in  \cite{BoCrip}  F. Bouchut and G. Crippa proved the existence and uniqueness of flows for vector fields with gradients given by singular integrals  of $L^1$ functions i.e $D\mathbf{B}=\mathbf{K}\star g$, $g\in L^1$, where $\mathbf{K}$ is a singular kernel of fundamental type in $\mathbb{R}^d$. Notice that this class is very natural in the study of  nonlinear PDEs,  such as the Euler equation and the classical Vlasov-Poisson equation, this class is not contained in $BV$ and neither contains it. To do this, they have used  the following  maximal singular integral operator: 
$$\mathbf{T}(\mu)(x)=\sup_{\varepsilon>0}|(\rho_{\varepsilon}\star \mathbf{K}\star\mu )(x)|~~\forall~x\in\mathbb{R}^d,$$
 where $\rho_\varepsilon(.)=\varepsilon^{-d}\rho(./\varepsilon)$, for some  $\rho\in C^1_c$ such that $\int_{\mathbb{R}^d}\rho dx=1$. Then, $
 \Phi_\delta(t) =\circ(|\log(\delta)|)$  is obtained from using the  boundedness of such operator from $L^1$ to weak-$L^1$ and the fact that 
 \begin{align}\label{weak11}
\lambda\mathcal{L}^d\left(\left\{ \mathbf{T}(\mu)>\lambda\right\}\right)\to 0 ~~\text{as}~~\lambda\to \infty,
 \end{align}
 for any $\mu \in L^1(\mathbb{R}^d)$, see the proof of Lemma \ref{le2}. Notice that \eqref{weak11} is not true for $\mu\in \mathcal{M}_b(\mathbb{R}^d)$;  indeed, it is easy to check that if $\mu=\delta_{0}$  then $\lambda\mathcal{L}^d\left(\left\{ \mathbf{T}(\mu)>\lambda\right\}\right)\gtrsim 1,\forall\lambda>0$ for some $\rho$ and $\mathbf{K}$. \medskip\\
 However, later in \cite{BoBoCrip1} they extended the analysis to the case where
 \begin{align*}
 D\mathbf{B}=\left( {\begin{array}{cc}
 	D_{x_1}\mathbf{B}_1 & D_{x_1} \mathbf{B}_2 \\
 D_{x_2} \mathbf{B}_1 & D_{x_2} \mathbf{B}_2\\
 	\end{array} } \right)=\left( {\begin{array}{cc}
 	\mathbf{K}_1\star f_1 & \mathbf{K}_2 \star f_2  \\
 	\mathbf{K}_0 \star\mu & \mathbf{K}_3\star f_3 \\
 	\end{array} } \right)~~~x=(x_1,x_2),~~\mathbf{B}=(\mathbf{B}_1,\mathbf{B}_2),
 \end{align*}
 where $\mathbf{K}_0,\mathbf{K}_1,\mathbf{K}_2,\mathbf{K}_3$ are singular kernels of fundamental type. 
 This is motivated from the Classical Vlasov-Poisson system associated to $B(x_1,x_2)=(x_2, \mathbf{P}\star\mu (x_1))$, $(x_1,x_2)\in \mathbb{R}^m\times\mathbb{R}^m$, $d=2m$ and $\mathbf{P}(x_1)=c\frac{x_1}{|x_1|^{m}}$, $\mu\in \mathcal{M}_b$. In addition, Jabin in \cite{ChamJabin}   has proven the well-posedness of this system with  $\mathbf{P}\star\mu\in H^{3/4}$ (or $\mu\in H^{-1/4}$). 
 In \cite{Seis,CripNobiSeisSpi} Seis  has provided  a quantitative  theory for continuity equation with $W^{1,1}$ vector fields via logarithmic Kantorovich-Rubinstrain distances. Recently, in \cite{BN18a,BN18b,BN18c,BN18d}, we have proved  sharp regularity estimates for solutions of continuity equations with $W^{1,p} (p>1)$ vector fields. 
 
 To our knowledge, these results in \cite{BoBoCrip1,Jabin} are  the best results for the quantitative ODE estimates at this moment.  In this paper, we give quantitative  estimates for $\mathbf{K}\star BV$ vector fields with bounded divergence. Namely, we prove the following theorem: \medskip\\\\
  Given a vector field $\mathbf{B}=(\mathbf{B}^1,...,\mathbf{B}^d)\in L^1([0,T];L^1_{\loc}(\mathbb{R}^d,\mathbb{R}^d))$, we assume that  for any $R>0$, there exist functions $b_{jR}\in L^1([0,T],BV(\mathbb{R}^d))$ for $j=1,...,m$; and degree-zero homogeneous functions
  $(\Omega_{jR}^i)_{i,j}\in L^1_{\loc}(\mathbb{R}^d)$ $(i=1,..,d,j=1,...,m)$ satisfying 
   $\int_{S^{d-1}}\Omega_{jR}^i=0$  and $\Omega_{jR}^i\in BV(S^{d-1})$  such that 
  \begin{align}\label{vectorfiB''}	\mathbf{B}^i=\sum_{j=1}^{m}\left(\frac{\Omega_{jR}^i(\cdot)}{|\cdot|^{d}}\right)\star b_{jR}~~\text{in}~B_{2R}.
  \end{align}
 \textbf{Main Theorem}. \textit{Let $\mathbf{B}_1,\mathbf{B}_2\in L^1([0,T];L^1_{\loc}(\mathbb{R}^d,\mathbb{R}^d))$ satisfy $$
 	||(\frac{|\mathbf{B}_1|}{|x|+1},\frac{|\mathbf{B}_2|}{|x|+1})||_{L^1((0,T);(L^1+L^\infty)(\mathbb{R}^d))}\leq C_0,$$ and let $X_1,X_2$ be regular Lagrangian flows associated to $\mathbf{B}_1,\mathbf{B}_2$ resp. Assume that  $\operatorname{div}(\mathbf{B})\in L^1((0,T),L^1(\mathbb{R}^d))$. Then, for any $\kappa\in (0,1), r>1$ there exist $R_0=R_0(d,T,r,C_0,\kappa)>1$ and $\delta_0=\delta_0(d,T,r,C_0,c_{R_0},b_{R_0},\kappa)\in (0,1)$ such that 
 	\begin{equation}
 	\label{maines}
 	\sup_{t\in [0,T]}\mathcal{L}^d\left(\left\{x\in B_r:|X_{1t}(x)-X_{2t}(x)|>\delta^{1/2}\right\}\right)\lesssim \delta^{-1}||\left(\mathbf{B}_1-\mathbf{B},\mathbf{B}_2-\mathbf{B}\right)||_{L^1([0,T]\times B_{R_0})}+\kappa,
 	\end{equation}
for any $\delta\in (0,\delta_0)$. }
\medskip\\
Note that if $\mathbf{B}_1,\mathbf{B}_2\in L^\infty_{t,x}$, we can take   $R_0$ independent of $\kappa$. Moreover, if $\mathbf{B}\in L^1((0,T);BV_{\loc}(\mathbb{R}^d))$, we can write for any $R>0$, 
$\mathbf{B}^i=\sum_{j=1}^{d}\mathcal{R}_j^2(\chi_{R}\mathbf{B}^i)~~\text{in}~B_R(0)$, where $\chi_{R}\in C_c^\infty(\mathbb{R}^d)$ satisfies $\chi_R=1$ in $B_{2R}(0)$ and $\chi_{R}=0$ in $B_{4R}(0)^c$,  $\mathcal{R}_1,...,\mathcal{R}_d$ are the Riesz transforms in $\mathbb{R}^d$. Thus, the class of  $\mathbf{B}$ in above theorem contains the class of $BV-$ vector fields and hence a main open problem posed by Luigi Ambrosio (see \cite{AmbCrip}) is solved. \medskip 

This Theorem is as a consequence of Theorem \ref{mainthm2} and Corollary \ref{maincorollary} in Section 4.  In Section 5, we will use this to deduce the well-posedness of regular Lagrangian flows and Transport, continuity equations.  The following result gives an existence and uniqueness result of regular Lagrangian flows.\vspace{0.3cm}\\
\textbf{Proposition}. Let $\mathbf{B}$ be as above. Assume that $
||\frac{|\mathbf{B}|}{|x|+1}||_{L^1((0,T);(L^1+L^\infty)(\mathbb{R}^d))}\leq C_0$ and $\operatorname{div}(\mathbf{B})\in L^1((0,T),L^1(\mathbb{R}^d))$.  Then, there exists a unique regular Lagrangian flows associated to vector field $\mathbf{B}$.\vspace{0.3cm}\\
 Let us describe our idea to prove \eqref{maines}. For simplicity, assume that  $\mathbf{B}_1(t,x)=\mathbf{B}_2(t,x)= \mathbf{B}(t,x)\equiv \mathbf{B}(x)\in\left( BV\cap L^\infty\right)(\mathbb{R}^d,\mathbb{R}^d)$. Thanks to Alberti's rank one Theorem (see section 2), there  exist unit vectors $\xi(x)\in \mathbb{R}^d$ and $\eta(x)\in \mathbb{R}^d$ such that $ D^{s}\mathbf{B}(x)=\xi(x)\otimes\eta(x)|D^{s}\mathbf{B}|(x)$ i.e  $D_{x_i}^{s}\mathbf{B}_j(x)=\xi_j(x)\eta_i(x)|D^{s}\mathbf{B}|(x)$ for any $i,j=1,...,d.$ Thus, one gets from $\operatorname{div}(\mathbf{B})\in L^1([0,T]\times\mathbb{R}^d)$ that $|\langle\xi,\eta\rangle| =0$ for $|D^{s}\mathbf{B}|-$a.e in $\mathbb{R}^d$.
 We first have the following basic inequality: for any $x_1\not= x_2\in \mathbb{R}^d$ and $\nu\in S^{d-1}$,
$$
\left| \langle \nu,\mathbf{B}(x_1)-\mathbf{B}(x_2)\rangle\right|\lesssim \sum_{l=1,2}\int \frac{\mathbf{1}_{|x_l-z|\leq r}}{|x_l-z|^{d-1}} |\langle \nu,\xi(z)\rangle|d\mu(z)+\int \frac{\mathbf{1}_{|x_l-z|\leq r}}{|x_l-z|^{d-1}}d|D^a\mathbf{B}|(z),
$$
see Proposition \ref{Jabin-lem}, where $\mu=|D^{s}\mathbf{B}|$ and $r=|x_1-x_2|$, where $D^{a}\mathbf{B}$ is  regular part of $D\mathbf{B}$ with respect to the Lebesgue measure. We now assume that $\xi$ and $\eta$ are smooth functions in $\mathbb{R}^d$. Then, choosing $\nu=\eta(x_1)$ and thanks to $|\langle\xi,\eta\rangle| =0$ for $|\mu|-$a.e in $\mathbb{R}^d$ yields 
$$
|\langle \nu,\xi(z)\rangle|\leq ||\nabla \eta||_{L^\infty}\left(|x_1-x_2|+|x_l-z|\right)~\text{for}~ |\mu|-\text{a.e }~z~\text{in}~\mathbb{R}^d,l=1,2,$$
This implies
\begin{align}\label{es-1}
\frac{\left| \langle \eta(x_1),\mathbf{B}(x_1)-\mathbf{B}(x_2)\rangle\right|}{|x_1-x_2|}\lesssim \sum_{l=1,2} ||\nabla \eta||_{L^\infty}\mathbf{I}_1( \mu)(x_l) +\mathbf{M}(|D^a\mathbf{B}|)(x_l),
\end{align}
where $\mathbf{I}_1$ is the Riesz potential with the first order in $\mathbb{R}^d$.\vspace{.2cm}\\
Let $X_1,X_2$ be  Regular Lagrangian flows associated to the same vector field $\mathbf{B}$ and $r>0$. Thus, we derive from \eqref{es-1} that 
\begin{align}\label{zero}
\limsup_{\delta\to 0}\frac{1}{|\log(\delta)|}\int_{0}^{T}\int_{B_r}\frac{\left| \langle \eta(X_{1t}),\mathbf{B}(X_{1t})-\mathbf{B}(X_{2t})\rangle\right|}{\delta+|X_{1t}-X_{2t}|}dxdt=0.
\end{align}
This suggests to consider the following new quantity: for $\delta\in (0,1),\gamma>1$
\begin{align}
\Phi_\delta^{\gamma}(t)=\frac{1}{2}\int_{B_r}\log\left(1+\frac{|X_{1t}-X_{2t}|^2+\gamma \langle \eta(X_{1t}),X_{1t}-X_{2t}|\rangle ^2}{\delta^2}\right)dx.
\end{align}
We have, 
\begin{align*}
\sup_{t\in [0,T]}\Phi_\delta^{\gamma}(t)&= \sup_{t_1\in [0,T]}\int_{0}^{t_1}\frac{d \Phi_\delta^{\gamma}(t)}{dt}dt\leq \int_{0}^{T}\int_{B_r}\frac{\gamma^{1/2}\left| \langle \eta(X_{1t}),\mathbf{B}(X_{1t})-\mathbf{B}(X_{2t})\rangle\right|}{\delta+|X_{1t}-X_{2t}|}dxdt\\&+
\int_{0}^{T}\int_{B_r}\frac{|\mathbf{B}(X_{1t})-\mathbf{B}(X_{2t})|}{\delta+|X_{1t}-X_{2t}|+\gamma^{1/2} |\langle \eta(X_{1t}),X_{1t}-X_{2t}\rangle|}dxdt\\&+ \int_{0}^{T}\int_{B_r}\frac{\gamma^{1/2}\left| \langle \nabla\eta(X_{1t})\mathbf{B}(X_{1t}),X_{1t}-X_{2t}\rangle\right|}{|X_{1t}-X_{2t}|}dxdt.
\end{align*}
Combining this and \eqref{zero}, we get 
\begin{align*}
&\sup_{t\in [0,T]}\mathcal{L}^d\left(\left\{x\in B_r:|X_{1t}-X_{2t}|>0\right\}\right)=\limsup_{\delta\to 0}\frac{1}{|\log(\delta)|}\sup_{t\in [0,T]}\Phi_\delta^{\gamma}(t)\\&\quad\quad\quad\leq 
\limsup_{\delta\to 0}\frac{1}{|\log(\delta)|}\int_{0}^{T}\int_{B_r}\frac{|\mathbf{B}(X_{1t})-\mathbf{B}(X_{2t})|~dxdt}{\delta+|X_{1t}-X_{2t}|+\gamma^{1/2} |\langle \eta(X_{1t}),X_{1t}-X_{2t}\rangle|}\\&\quad\quad\quad:=\limsup_{\delta\to 0} A(\delta).
\end{align*}
Hence, in order to get $X_{1t}=X_{2t}$ for a.e $(x,t)\in B_r\times [0,T]$, we need to show that 
$$\limsup_{\delta\to 0} A(\delta)=\circ(1)~\text{as}~\gamma\to\infty.$$
In fact, 
 we use the following estimate for $\mathbf{B}(x_1)-\mathbf{B}(x_2)$:  
\begin{align*}
|\mathbf{B}(x_1)-\mathbf{B}(x_2)|&\lesssim \varepsilon^{-d+1}|x_1-x_2|\left(\mathbf{M}(|D^a \mathbf{B}|)(x_1)+\mathbf{M}(|D^a \mathbf{B}|)(x_2)\right)\\ & +\varepsilon^{-d+1}\sum_{l=1,2}\int \frac{ \mathbf{1}_{|x_l-z|\leq r}\mathbf{1}_{\left|\frac{x_l-z}{|x_l-z|}-e_l\right|\leq\varepsilon}}{|x_l-z|^{d-1}} \frac{|\langle \eta(z),x_1-x_2\rangle|}{|x_1-x_2|} d|\mu|(z)\\&+ \varepsilon^{-d+2}\sum_{l=1,2}\int \frac{ \mathbf{1}_{|x_l-z|\leq r}\mathbf{1}_{\left|\frac{x_l-z}{|x_l-z|}-e_l\right|\leq\varepsilon}}{|x_l-z|^{d-1}}  d|\mu|(z),
\end{align*}
for any $\varepsilon>0$ where  $\mu=|D^{s}\mathbf{B}|$, $e_1=-e_2=\frac{x_1-x_2}{|x_1-x_2|},r=|x_1-x_2|$ for $ l=1,2$
(see Proposition \ref{Jabin-lem} and Lemma \ref{coro-Jab-Hu}).
Then, using the fact that $|\langle \eta(z),x_1-x_2\rangle|\leq |\langle \eta(x_1),x_1-x_2\rangle|+2||\nabla\eta||_{L^\infty} r^2$ for $|z-x_1|\leq r$ or $|z-x_2|\leq r$ and changing variable along the flows  we can estimate 
\begin{align*}
&A(\delta)\lesssim \frac{\gamma^{-1/2}\varepsilon^{-d+1}}{|\log(\delta)|}\int_{B_{r'}} \min\left\{\frac{\mathbf{I}_1(\mu)}{\delta},\mathbf{M}(\mu) \right\} dx+||\nabla\eta||_{L^\infty} \frac{\varepsilon^{-d+1}}{|\log(\delta)|}\int_{B_{r'}} \mathbf{I}_1(\mu)dx\\&~+\frac{\varepsilon}{|\log(\delta)|}\int_{B_{r'}} \min\left\{\frac{\mathbf{I}_1(\mu)}{\varepsilon^{d-1}\delta},\mathbf{M}^{\varepsilon}(\mu) \right\} dx+\frac{\varepsilon^{-d+1}}{|\log(\delta)|}\int_{B_{r'}} \min\left\{\frac{\mathbf{I}_1(|D^aB|)}{\delta},\mathbf{M}(|D^aB|) \right\} dx,
\end{align*}
 for some $r'>r$,  where  $\mathbf{M}^\varepsilon$ is the Kakeya maximal function in $\mathbb{R}^d$ i.e $$
\mathbf{M}^\varepsilon(\mu)(x)=\sup_{\rho\in (0,2r'),e\in S^{d-1}}\fint_{B_\rho(x)}\varepsilon^{-d+1}\mathbf{1}_{|\frac{x-z}{|x-z|}-e|\leq \varepsilon}d|\mu|(z).$$
We then will deduce that 
\begin{align*}
\limsup_{\delta\to 0}A(\delta)&\lesssim \gamma^{-1/2}\varepsilon^{-d+1}|\mu|(\mathbb{R}^d) +\varepsilon \limsup_{\lambda\to \infty}\lambda\mathcal{L}^d\left(\{ \mathbf{M}^{\varepsilon}(\mu)>\lambda\}\cap B_{r'}\right)\\ &= \varepsilon \limsup_{\lambda\to \infty}\lambda\mathcal{L}^d\left(\{ \mathbf{M}^{\varepsilon}(\mu)>\lambda\}\cap B_{r'}\right)~\text{as}~\gamma\to\infty.
\end{align*}
So it remains to show that 
\begin{align}\label{keyes}
I(\varepsilon):=\varepsilon \limsup_{\lambda\to \infty}\lambda\mathcal{L}^d\left(\{ \mathbf{M}^{\varepsilon}(\mu)>\lambda\}\cap B_{r'}\right)=\circ(1)~~\text{as}~~\varepsilon\to 0.
\end{align}
This estimate is very delicate and  its proof is very complicated, hence we  will spend Section 3 to establish it. In order to see the key idea for proving the estimate \eqref{keyes}, we only consider $\mu(x)=|D^sf|(x)\equiv|Df|(x)$ with $f\in BV(\mathbb{R}^d,\mathbb{R})$ such that $\nu=\frac{dD^sf}{d|D^sf|}(x)$ is a constant function in $B_{8r'}$. 
Set $H_{\nu}:=\left\{x\in \mathbb{R}^d: \langle \nu,x\rangle=0\right\}$ and
 $\tilde{H}_{\nu}:=\left\{t\nu\in \mathbb{R}^d: \forall t\in \mathbb{R}\right\}$. We also denote $f^{\nu}_{y_2}:\tilde{H}_{\nu}\ni y_1\mapsto f(y_2+y_1)$ for any $y_2\in H_{\nu}$. By assumption one has $d\mu(y)=d|Df^{\nu}_{z}|(y_1)d\mathcal{H}^{d-1}(y_2)$ for any   $y_1=\langle y,\nu\rangle\nu,y_2=y-\langle y,\nu\rangle\nu$, $y\in B_{8r'}$ and  $z\in H_{\nu}$.  We  can prove that 
 \begin{align}\label{kaye}
 \mathbf{M}^{\varepsilon}(\mu)(x)\leq C\mathbf{M}^1(|Df^{\nu}_{x_\nu}|,\tilde{H}_{\nu})(\langle x,\nu\rangle \nu), \text{with}~x_\nu:=x-\langle x,\nu\rangle \nu,
 \end{align}
 where $\mathbf{M}^1(|Df^{\nu}_{x_\nu}|,\tilde{H}_{\nu})$ is the Hardy-Littlewood maximal function of $|Df^{\nu}_{x_\nu}|$ on $\tilde{H}_{\nu}$.
Indeed, by a standard approximation argument, we only prove  for case $|Df^{\nu}_{x_\nu}|\in L^1(\tilde{H}_{\nu}, d\mathcal{H}^1)$.  
By changing variables, we have  for any $\rho\in (0,2r'),e\in S^{d-1}$, $x\in B_{r'}$
\begin{align*}
 &\fint_{B_\rho(x)}\varepsilon^{-d+1}\mathbf{1}_{|\frac{x-y}{|x-y|}-e|\leq \varepsilon}d|\mu|(y)\\&= \rho^{-d}\varepsilon^{-d+1}\fint_{S^{d-1}}\int_{0}^\rho\mathbf{1}_{|\theta-e|\leq \varepsilon} |Df^{\nu}_{x_\nu}|(\langle x,\nu\rangle \nu-\langle \theta,\nu\rangle \nu s )s^{d-1}dsd\mathcal{H}^{d-1}(\theta)
 \\& \leq \varepsilon^{-d+1}\fint_{S^{d-1}}\mathbf{1}_{|\theta-e|\leq \varepsilon} 4\mathbf{M}^1(|Df^{\nu}_{x_\nu}|,\tilde{H}_{\nu})(\langle x,\nu\rangle \nu)d\mathcal{H}^{d-1}(\theta) 
 \\&\lesssim\mathbf{M}^1(|Df^{\nu}_{x_\nu}|,\tilde{H}_{\nu})(\langle x,\nu\rangle \nu),
\end{align*}which implies \eqref{kaye}.  
Therefore, we get from \eqref{kaye} and weak type (1,1) bound of $\mathbf{M}^1(|Df^{\nu}_{x_\nu}|,\tilde{H}_{\nu})$ that 
\begin{align*}
\lambda\mathcal{L}^d\left(\{ \mathbf{M}^{\varepsilon}(\mu)>\lambda\}\cap B_{r'}\right)&\leq \lambda\int_{H_\nu} \mathcal{H}^1\left(\left\{x_1\in \tilde{H}_\nu :\mathbf{M}^1(|Df^{\nu}_{x_2}|,\tilde{H}_{\nu})(x_1)\gtrsim\lambda\right\}\right)d\mathcal{H}^{d-1}(x_2)\\&\lesssim \int_{H_\nu} \int_{\tilde{H}_\nu} d|Df^{\nu}_{x_2}|(x_1)d\mathcal{H}^{d-1}(x_2)=|\mu|(\mathbb{R}^d).
\end{align*}
This gives \eqref{keyes}.  In order to prove \eqref{keyes} in the general case, we use that $\mu=|D^{s}\mathbf{B}|$ and  the slicing theory of $BV$ functions. Notice \eqref{keyes} is not true for any Radon measure $\mu$, indeed  if $\mu=\delta_{0}$, then $\mathbf{M}^\varepsilon(\mu)(x)=\varepsilon^{-d+1}|x|^{-d}$ and so  $I(\varepsilon)\sim \varepsilon^{-d+2}$.\medskip\\
To conclude this section, let us give an important remark on our result.  We deduce from \eqref{vectorfiB''} that 
	\begin{align}\label{vectorfiB'''}	\partial_l\mathbf{B}^i=\sum_{j=1}^{m}\left(\frac{\Omega_{jR}^i(\cdot)}{|\cdot|^{d}}\right)\star \mu_{jR}^l~\text{in}~\mathcal{D}'(B_{2R}),
\end{align}
where $\mu_{jR}^l=\partial_{l}b_{jR}$,  $l,i=1...,d,j=1,...,m$ are  bounded Radon measures in $\mathbb{R}^d$. Thus,\\

  \textit{A natural question is  whether the above Proposition holds for  a class of vector fields $\mathbf{B}$  satisfying \eqref{vectorfiB'''}  with   arbitrary Radon measures $\mu_{jR}^l$  in $\mathbb{R}^d$. }\\

The following proposition gives a negative answer to this question.
\begin{proposition} \label{exampleDipernaLions} There exist a vector field $B:\mathbb{R}^2\to\mathbb{R}^2$ and degree-zero homogeneous  functions $\Omega_1^i,...,\Omega_m^i\in \left(L^\infty\cap BV\right)(S^1) $, $i=1,2$ with $\frac{|B(x)|}{|x|+1}\in L^1(\mathbb{R}^2)+L^\infty(\mathbb{R}^2)$,  $\operatorname{div}(B)=0$,  $\int_{S^1}\Omega_l=0$ such that for any $R>1$ we have $$
	\partial_l \mathbf{B}^i=\sum_{j=1}^{m}\left(\frac{\Omega_j^i(\cdot)}{|\cdot|^2}\right)\star\mu_{jR}^{l}~~\text{in}~~\mathcal{D}'(B_R),$$
for some $\mu_{jR}^{l}\in \mathcal{M}_b(\mathbb{R}^2)$ $i,l=1,2$ and $j=1,...,m$ and the problem \eqref{odeeq1} is ill-posed with this vector field, i.e  there exist two different regular Lagrangian flows $X_1,X_2$  associated to $\mathbf{B}$.
\end{proposition}
We will prove proposition \ref{exampleDipernaLions} in the  Appendix.\medskip\\\\
\textbf{Acknowledgments}~ The author is particularly grateful to Professor Luigi Ambrosio who introduced this project to him and patiently guided, supported, encouraged him during this work. The author is also very thankful to Elia Bru\`e, Fran\c{c}ois Bouchut and Thomas Alazard for their several helpful comments and corrections. The author would like to thank the anonymous referees for their careful reading of our manuscript and their many insightful comments and  suggestions to improve the presentation of the paper. This  research was supported by the Centro De Giorgi, Scuola Normale Superiore, Pisa, Italy.
\section{Main Notation and preliminary results}
We begin with some notations which will be used in  this  paper.\medskip\\
\noindent $\bullet$ $x\cdot y$, $\langle x,y\rangle$ denote the usual scalar product of $x,y\in \mathbb{R}^d$;\\
\noindent $\bullet$ $a\wedge b$ denotes $\min\{a,b\}$;\\
\noindent $\bullet$ $S^{d-1}$ denotes the $(d-1)-$dimensional unit sphere in $\mathbb{R}^d$;\\
	\noindent $\bullet$ $\mathbf{1}_E$ is the characteristic function of the set $E$, defined  as $\mathbf{1}_E(x)=1$ if $x\in E$ and $\mathbf{1}_E(x)=0$ otherwise;\\
	\noindent $\bullet$ $B_r(x)$ is the open ball in $\mathbb{R}^d$ with radius $r$ and center $x$; $B_r$ is the open ball in $\mathbb{R}^d$ with radius and center $0$; if  $X$ is a vector subspace of $\mathbb{R}^d$, for any $x\in X$, $B_r(x,X)$ is the open ball in $X$ with radius $r$ and center $x$ i.e $B_r(x,X)=B_r(x)\cap X.$\\
	\noindent$\bullet$ $\mathcal{M}_b(X)$ is a set of bounded Radon measure in a metric space $X$;  $\mathcal{M}_b^+(X)$ is a set of positive bounded Radon measure in $X$;\\
	\noindent $\bullet$ $|\mu|$ is the total variation of a measure $\mu$; $\mu^{s}, \mu^{a}$ are the singular component and regular component of $\mu$ with respect to the Lebesgue measure, respectively;\\
		\noindent $\bullet$ $\mathcal{L}^d$ is the Lebesgue measure on $\mathbb{R} ^d$ and $\mathcal{H}^k$ is the $k-$dimensional Hausdorff measure;\\
		\noindent $\bullet$ $BV(\mathbb{R}^d,\mathbb{R}^m)$ is a set of $\mathbb{R}^m-$valued functions with bounded variation in $\mathbb{R}^d$;\\
		\noindent $\bullet$ $f\star g$ is the convolution of  $f$ and $g$, in particular if $f,g\in \mathbb{R}^l$, then $f\star g:=\sum_{j=1}^{l}f_j\star g_j$;  if $f\in \mathbb{R}^l,g\in \mathbb{R}$, then $f\star g=g\star f:=\left(f_1\star g,f_2\star g,...,f_l\star g\right)$;\\
		\noindent $\bullet$ $f_{\#}\mu$ is the push-forward of $\mu$ via a Borel map $f$, more specifically, a Borel map $f:\mathbb{R}^{l}\to\mathbb{R}^m$, and a measure $\mu$ in $\mathbb{R}^{l}$ then $f_{\#}\mu$ is a measure in  $\mathbb{R}^{m}$ given by $f_{\#}\mu(B)=\mu(f^{-1}(B))$ for any Borel set $B\subset \mathbb{R}^m$; this is equivalent to $\int_{\mathbb{R}^m}\phi df_{\#}\mu=\int_{\mathbb{R}^l}\phi\circ fd\mu$ for any $\phi:\mathbb{R}^m\to [0,+\infty]$ Borel.\\
		\noindent $\bullet$ $\fint_Efd\omega$ denotes the average of the function $f$ over the set $E$ with respect to the positive measure $\omega$ i.e $\fint_Efd\omega:=\frac{1}{\omega(E)}\int_E fd\omega$;\\
		\noindent $\bullet$ $\left\{f>\lambda\right\}$, $\left\{f<\lambda\right\}$ stand for $\left\{x:f(x)>\lambda\right\}$, $\left\{x:f(x)<\lambda\right\}$ respectively;\\
	\noindent $\bullet$	$\varrho_n$  is a standard sequence of mollifiers in $\mathbb{R}^d$;\\
		\noindent $\bullet$ $E^c$ is the complement of set $E$;\\
		\noindent $\bullet$  $A\lesssim B$  denotes the estimate $A\leq CB$ for some constant $C>0$ depending only on fixed  quantities; and $A\sim B$ denotes the estimate $A\lesssim B\lesssim A$;\\
		\noindent $\bullet$ $C(n,\varepsilon,\kappa..)$ is a common constant that satisfies parameters $n,\varepsilon,\kappa,...$. \medskip\\
\textbf{2.1}\textit{ $BV$ functions}.
Given  $b\in BV(\mathbb{R}^d,\mathbb{R}^m)$, we have the canonical decomposition of $Db$ as $D^ab+D^{s}b$, with $|D^ab|<<\mathcal{L}^d$ and $|D^{s}b|\bot \mathcal{L}^d$. The following deep result of Alberti will be used in the proof of the main Theorem \ref{mainthm2}. Its proof can be found in \cite{Alber} and \cite{PhiRin}. 
\begin{proposition}[Alberti's rank one Theorem] \label{albertithm}There  exist unit vectors $\xi(x)\in \mathbb{R}^m,\eta(x)\in \mathbb{R}^d$ such that $D^sb(x)=\xi(x)\otimes\eta(x)|D^{s}b|(x)$ i.e  $D_{x_i}^{s}b_j(x)=\xi_j(x)\eta_i(x)|D^{s}b|(x)$ for any $i=1,...,d, j=1,...,m.$
\end{proposition}
Notice that the pair of unit vector $(\xi,\eta)$ is uniquely determined $|D^{s}b|-$a.e up to a change of sign.  Case $m=d$, we can write the distributional divergence $\operatorname{div}(b)$ as $
\operatorname{div}(b)=\operatorname{trace}(D^a b)\mathcal{L}^d+\langle \xi,
\eta\rangle |D^{s}b|, $
thus, $
\operatorname{div}(b)<<\mathcal{L}^d~~~\text{if and only if}~~\xi\bot\eta~~|D^{s}b|-\text{a.e.~in}~~\mathbb{R}^d.$\medskip \\
For $e\in S^{d-1}$, let us introduce the hyperplane orthogonal to $e$:  $$H_{e}:=\left\{x\in \mathbb{R}^d: \langle e,x\rangle=0\right\},$$
and the line of $e$:
$$\tilde{H}_{e}:=\left\{te\in \mathbb{R}^d: \forall t\in \mathbb{R}\right\}.$$
 Given a Borel function $f$ in $\mathbb{R}^d$, we denote 
$f^{e}_{y_1}:\tilde{H}_{e}\ni z_1\mapsto f(y_1+z_1)$ for $y_1\in H_{e}$. The following characterization of $BV$ by hyperplanes will be used in proof of main  Theorem \ref{mainthm1}. 
\begin{proposition}(\cite[Section 3.11]{luigiFuPa}) \label{ML1-BV} Let $f\in BV(\mathbb{R}^d)$ and $e\in S^{d-1}$. Then, $f^{e}_{y_1}\in BV(\tilde{H}_{e})$,~~$\mathcal{H}^{d-1}$-a.e $y_1$ in $H_{e}$ and $\int_{H_{e}}||f^{e}_{y_1}||_{BV(\tilde{H}_{e})}d\mathcal{H}^{d-1}(y_1)<\infty$. Moreover, for any bounded Borel function $\phi:\mathbb{R}^d\to \mathbb{R}^+$, there holds
	\begin{equation}
	\label{equa-eta}
	\int_{H_{e}} \int_{\tilde{H}_{e}}\phi(t+y_1) 	d |D^{s}f^e_{y_1}|(t)d\mathcal{H}^{d-1}(y_1)=\int_{\mathbb{R}^d}\phi(x)\left|\langle e,\eta(x)\rangle\right|d|D^sf|(x),
	\end{equation} 
	where $\eta(x)=\frac{d D^{s}f(x)}{d |D^{s}f|(x)}$.
\end{proposition}

We next have an extension of \cite[Proposition 4.2]{ChamJabin}. It is one of main tools to be used in proof of main theorem \ref{mainthm2}. 
\begin{proposition} \label{Jabin-lem} Let $\varepsilon\in (0,1/100)$, $f\in BV_{loc}(\mathbb{R}^d)$. Then, for every $x,y\in \mathbb{R}^d$, $x\not= y$,
	\begin{align}\nonumber
	&	f(x)-f(y)\\&=\int_{\mathbb{R}^d}\frac{\varepsilon^{-d+1}}{|x-z|^{d-1}}	\Theta^{\varepsilon,\mathbf{e}_1}_1\left(\frac{x-z}{|x-y|}\right)\mathbf{e}_1\cdot dDf(z)
	+\int_{\mathbb{R}^d}\frac{\varepsilon^{-d+2}}{|x-z|^{d-1}}\Theta^{\varepsilon,\mathbf{e}_1}_2\left(\frac{x-z}{|x-y|}\right)\cdot dDf(z)\nonumber \\&~~-\int_{\mathbb{R}^d}\frac{\varepsilon^{-d+1}}{|y-z|^{d-1}}\Theta^{\varepsilon,\mathbf{e}_2}_1\left(\frac{y-z}{|x-y|}\right)\mathbf{e}_2\cdot dDf(z)-\int_{\mathbb{R}^d}\frac{\varepsilon^{-d+2}}{|y-z|^{d-1}}\Theta^{\varepsilon,\mathbf{e}_2}_2\left(\frac{y-z}{|x-y|}\right)\cdot dDf(z),\label{diff}
	\end{align}
	where $\mathbf{e}_1=-\mathbf{e}_2=\frac{x-y}{|x-y|}$ and  for $e\in S^{d-1}$, $\varepsilon\in (0,1/100)$,  $\Theta^{\varepsilon,e}_1:\mathbb{R}^d\to \mathbb{R}_+$ and $\Theta^{\varepsilon,e}_2:\mathbb{R}^d\to \mathbb{R}^d$ are bounded functions satisfying
	$	\Theta^{\varepsilon,e}_1, \Theta^{\varepsilon,e}_2\in C^{\infty}(\mathbb{R}^d\backslash\{0\})$, 
	\begin{align}&\nonumber	\supp(\Theta^{\varepsilon,e}_1), \supp(\Theta^{\varepsilon,e}_2)\subset  B_{3/4}(0)\cap\left\{x: \left|e-\frac{x}{|x|}\right|\leq \varepsilon\right\},\\ \nonumber
	&|\Theta^{\varepsilon,e}_l(x)|+\varepsilon|x||\nabla \Theta^{\varepsilon,e}_l(x)| \lesssim 1~\forall~x\in \mathbb{R}^d, l=1,2,\\& \varepsilon^{-d+1}\int_{\mathbb{R}^d}|\Theta^{\varepsilon,e}_1(x)|dx+ \varepsilon^{-d+1}\int_{\mathbb{R}^d}|\Theta^{\varepsilon,e}_2(x)|dx \lesssim 1.\label{Z1}
	\end{align}
\end{proposition}
\begin{proof}[Proof of Proposition \ref{Jabin-lem}]   Let $\rho:\mathbb{R}\to [0,\infty)$ be a $C_c$ function such that $\rho\in C^\infty([0,1])$, $\rho(t)=1$ for $0\leq t\leq 1/4$, $\rho(t)=0$ for $t\geq \frac{3}{4}$ and $t<0$, $\rho(t)+\rho(1-t)=1$ for $0\leq t\leq 1$. Let $\psi:(0,\infty)\to (0,\infty)$ be a $C_b^\infty$ function such that $\psi(t)=0$ for $t>1$, $\psi(t)=1$ in $(0,\varepsilon_0)$ for some $\varepsilon_0\in (0,1)$  and $\int_{\mathbb{R}^{d-1}}\psi(|h|)dh=1$.\\ We define for $(a,b,c)\in S^{d-1}\times S^{d-1}\times (0,\infty)$,
	\begin{align*}
		&\Psi_1(a,b,c)=\frac{\rho((a.b)c)\psi\left(\frac{|a-(a.b)b|}{4(a.b)(1-(a.b)c)}\right)}{4^{d-1}(a.b)^{d}(1-(a.b)c)^{d-1}},~~\\&	\Psi_2(a,b,c)=\frac{\rho((a.b)c)\psi\left(\frac{|a-(a.b)b|}{4(a.b)(1-(a.b)c)}\right)}{4^{d-1}(a.b)^{d-1}(1-(a.b)c)^{d}}c (a-(a.b)b).
	\end{align*} 
	Since $\Psi_1(a,b,c)=\Psi_1(-a,-b,c), \Psi_2(a,b,c)=-\Psi_2(-a,-b,c)$, thus it is not hard to obtain from the proof of \cite[Proposition 4.2]{ChamJabin} that 
	\begin{align*}
	f(x)-f(y)&=\int_{\mathbb{R}^d}\frac{1}{|x-z|^{d-1}}\Psi_1\left(\frac{x-z}{|x-z|},\frac{x-y}{|x-y|},\frac{|x-z|}{|x-y|}\right)\frac{x-z}{|x-z|} \cdot dDf(z)
	\\&-\int_{\mathbb{R}^d}\frac{1}{|x-z|^{d-1}}\Psi_2\left(\frac{x-z}{|x-z|},\frac{x-y}{|x-y|},\frac{|x-z|}{|x-y|}\right)\cdot dDf(z)\\&-\int_{\mathbb{R}^d}\frac{1}{|y-z|^{d-1}}\Psi_1\left(\frac{y-z}{|y-z|},\frac{y-x}{|y-x|},\frac{|y-z|}{|x-y|}\right)\frac{y-z}{|z-y|}\cdot dDf(z)
	\\&+\int_{\mathbb{R}^d}\frac{1}{|y-z|^{d-1}}\Psi_2\left(\frac{y-z}{|y-z|},\frac{y-x}{|y-x|},\frac{|y-z|}{|x-y|}\right)\cdot dDf(z).
	\end{align*}
	Replacing $\psi$ by $\frac{8^{d-1}}{\varepsilon^{d-1}}\psi(8\frac{.}{\varepsilon})$, we obtain \eqref{diff} where $
	\Theta^{\varepsilon,e}_l(z)=\phi_l^{\varepsilon}(z/|z|,e,|z|)$ for $(e,z)\in S^{d-1}\times \mathbb{R}^d;$ and 
	$$\phi^\varepsilon_1(a,b,c)=2^{d-1}\frac{\rho((a.b)c)\psi\left(\frac{2|a-(a.b)b|}{\varepsilon(a.b)(1-(a.b)c)}\right)}{(a.b)^{d}(1-(a.b)c)^{d-1}},~~~~~~~~~~~~~~~~~~~~~~~~~~~~~~~~~~~~~~~~~~~~~~~~~~~~~~~~~~~~~~~~~~~$$
	$$\phi^\varepsilon_2(a,b,c)=2^{d-1}\frac{\rho((a.b)c)\psi\left(\frac{2|a-(a.b)b|}{\varepsilon(a.b)(1-(a.b)c)}\right)}{(a.b)^{d}(1-(a.b)c)^{d-1}}\frac{a-b}{\varepsilon}-2^{d-1}\frac{\rho((a.b)c)\psi\left(\frac{2|a-(a.b)b|}{\varepsilon(a.b)(1-(a.b)c)}\right)}{(a.b)^{d-1}(1-(a.b)c)^{d}}c \frac{(a-(a.b)b)}{\varepsilon}.$$
	Note that $
	\rho((a.b)c)\psi\left(\frac{2|a-(a.b)b|}{\varepsilon(a.b)(1-(a.b)c)}\right)\not =0$
	implies $
	|a-(a.b)b|\leq \frac{\varepsilon}{2}$ and $(a.b)c\leq 3/4$.
	So, 
	\begin{align*}
	|a-b|=\sqrt{2(1-(a.b))}\leq \sqrt{2(1-(a.b)^2)}=\sqrt{2|a-(a.b)b|^2}\leq \varepsilon/\sqrt{2},
	\end{align*}
	and $a.b\geq 1-\varepsilon/2\geq 1/2$, $c\leq 3/4$.
	Hence, it is easy to check that $	\Theta^{\varepsilon,e}_1, \Theta^{\varepsilon,e}_2$  belong to  $ C^{\infty}(\mathbb{R}^d\backslash\{0\})$ and satisfy \eqref{Z1}.
	The proof is complete. 
\end{proof}
\textbf{2.2} \textit{The  Hardy-Littlewood maximal function and Riesz potential.}
We recall some basic properties of the  Hardy-Littlewood maximal function and Riesz potential. Let $X$ be a vector subspace of $\mathbb{R}^d$ with $\dim(X)=k$~ $(k=1,...,d)$ and $\mu$ be  a positive  Radon measure in $X$. The  Hardy-Littlewood maximal function of $\mu$ on $X$ is defined by $$
\mathbf{M}^k(\mu,X)(x)=\sup_{r>0}\frac{1}{\mathcal{H}^k(B_r(x,X))}\int_{B_r(x,X)}d|\mu|~~\forall~x\in X.$$
If $X=\mathbb{R}^d$, we write $\mathbf{M}(\mu)$ instead  of  $\mathbf{M}^k(\mu,X)$. It is well known that $\mathbf{M}^k(.,X)$ is bounded from $L^p(X,d\mathcal{H}^k)$ to itself and $\mathcal{M}_b(X)$ to $L^{1,\infty}(X,d\mathcal{H}^k)$ for $1<p\leq \infty$ i.e 
\begin{align}\label{boundLpLp}
&||\mathbf{M}^k(\mu,X)||_{L^p(X,d\mathcal{H}^k)}\lesssim ||\mu||_{L^p(X,d\mathcal{H}^k)}~~\text{for any}~\mu\in L^p(X,d\mathcal{H}^k);\\& \label{weaktype11}
\sup_{\lambda>0}\lambda\mathcal{H}^k\left(\left\{\mathbf{M}^k(\mu,X)>\lambda\right\}\right)\lesssim |\mu|(X)~~\text{for any}~\mu\in \mathcal{M}_b(X),
\end{align}
 (see \cite{Stein1}, \cite{Stein2},\cite{AmTilli}). 
 
The Riesz potential of  $\mu$ on $X$ is defined by 
\begin{equation}
\mathbf{I}_\alpha^k(\mu,X)(x)=\int_X\frac{1}{|x-z|^{k-\alpha}}d|\mu|(z)~~\forall~x\in X, 0<\alpha<k.
\end{equation} If $X=\mathbb{R}^d$, we write $\mathbf{I}_\alpha(\mu)$ instead  of  $\mathbf{I}_\alpha^k(\mu,X)$.
We have that $\mathbf{I}_\alpha^k(.,X)$ is bounded from $L^p(X,d\mathcal{H}^k)$ to $L^{\frac{kp}{k-\alpha p}}(X,d\mathcal{H}^k)$  for $p>1,0<\alpha p <k$; and bounded from $\mathcal{M}_b^+(X)$ to $L^{\frac{k}{k-\alpha},\infty}(X,d\mathcal{H}^k)$ for $0<\alpha<k$, see \cite{Stein1}.
\\ It is easy to see that for $\alpha>0,$
\begin{align}
\sup_{r>0}r^{-\alpha}\int_{X}\frac{\mathbf{1}_{|x-z|\leq r}}{|x-z|^{k-\alpha}}d\mu(z)\lesssim\mathbf{M}^k(\mu,X)(x)~~\forall~x\in X.
\end{align}
Thanks to \eqref{weaktype11}, one gets
\begin{align*}
\lambda\mathcal{H}^k\left(\left\{\mathbf{M}^k(\mu,X)>\lambda\right\}\right)&\leq \lambda\mathcal{H}^k\left(\left\{\mathbf{M}^k(\mu^{s},X)>\lambda/2\right\}\right)+\lambda\mathcal{H}^k\left(\left\{\mathbf{M}^k(\mu^a\mathbf{1}_{|\mu^a|\geq  \lambda/4},X)>\lambda/2\right\}\right)\\&\lesssim |\mu|^{s}(X)+ \int_{X}\mathbf{1}_{|\mu|^a\geq \lambda/4}|\mu|^adx,
\end{align*}
provided $|\mu|(X)<\infty$. Thus, 
\begin{align} \label{weak11limsup}
\limsup_{\lambda\to\infty}\lambda\mathcal{H}^k\left(\left\{\mathbf{M}^k(\mu,X)>\lambda\right\}\right)\lesssim|\mu|^{s}(X).
\end{align} 
Moreover,  in \cite{BHS} we  showed that for any $\lambda>0$, 
\begin{align}\label{upperweaktype}
\lambda\mathcal{H}^k\left(\left\{\mathbf{M}^k(\mu,X)>\lambda\right\}\right)\gtrsim|\mu|^{s}(X).
\end{align}
Therefore, it follows from \eqref{weak11limsup} and \eqref{upperweaktype} that for any $B_R:=B_R(0,X)\subset X$, 
\begin{align}\label{es5a}
	|\mu|^{s}(B_R)\lesssim \limsup_{\lambda\to\infty} \lambda\mathcal{H}^k\left(\left\{\mathbf{M}^k(\mu,X)>\lambda\right\}\cap B_R\right)\lesssim	|\mu|^{s}(\overline{B_R}).
\end{align}
Again,  \eqref{weak11limsup} and \eqref{upperweaktype} imply that $\mu<<\mathcal{H}^k$ in $X$ if any only if  
\begin{align}\label{weak11limsup:0}
\limsup_{\lambda\to\infty}\lambda\mathcal{H}^k\left(\left\{\mathbf{M}^k(\mu,X)>\lambda\right\}\right)=0.
\end{align}
Next is a basic estimate of the  Hardy-Littlewood maximal function, it will be used several times in this paper.
\begin{lemma}\label{le2} Let $X$ be a vector subspace of $\mathbb{R}^d$ with $\dim(X)=k$ and  $q>1$. Then, for any $\mu\in \mathcal{M}_b(X)$ and ball $B_R:=B_R(0,X)\subset X$ and $f\in L^q(B_R)$ there holds,
	\begin{align}
	\limsup_{\delta\to 0}\frac{1}{|\log (\delta)|}\int_{B_R}(\delta^{-1}|f|)\wedge\mathbf{M}^k(\mu,X)d\mathcal{H}^k\lesssim	|\mu|^{s}(\overline{B_R}).\label{es5}
	\end{align}
Moreover,  for any $0<\delta<<1$, 
	\begin{align}\label{es5*}
\frac{1}{|\log(\delta)|}\int_{B_R}(\delta^{-1}|f|)\wedge\mathbf{M}^k(\mu,X)d\mathcal{H}^k\lesssim R^k+|\mu|(X)+ ||f||_{L^q(B_R)}^q.
	\end{align}
\end{lemma}
\begin{proof} Set $A(\lambda_1)=\sup_{\lambda>\lambda_1}\lambda\mathcal{H}^k\left(\left\{\mathbf{M}^k(\mu,X)>\lambda\right\}\cap B_R\right)\lesssim |\mu|(X)$. \\One has for any $0<\delta<<1$ and $0<\lambda_1<\lambda_2<\infty$, 
\begin{align*}
	&\frac{1}{|\log (\delta)|}\int_{B_R}(\delta^{-1}|f|)\wedge\mathbf{M}^k(\mu,X)d\mathcal{H}^k\\&~~=\frac{1}{|\log (\delta)|}\int_{0}^{\infty}\mathcal{H}^k\left(\left\{ (\delta^{-1}|f|)\wedge\mathbf{M}^k(\mu,X)>\lambda \right\}\cap B_R\right)d\lambda
	\\&~~\leq \frac{1}{|\log (\delta)|}\int_{0}^{\lambda_1}\mathcal{H}^k(B_R)d\lambda+\frac{1}{|\log (\delta)|}\int_{\lambda_1}^{\lambda_2}A(\lambda_1)\frac{d\lambda}{\lambda}+\frac{1}{|\log (\delta)|}\int_{\lambda_2}^\infty \mathcal{H}^k\left(\left\{ |f|>\delta\lambda \right\}\cap B_R\right)d\lambda
	\\&~~\leq \frac{\lambda_1}{|\log (\delta)|}\mathcal{H}^k(B_R)+\frac{\log(\lambda_2/\lambda_1)}{|\log (\delta)|}A(\lambda_1)+\frac{1}{q|\log (\delta)|\lambda_1^{q-1}\delta^q} ||f||_{L^q(B_R)}^q.
	\end{align*}
	Choosing $\lambda_1=|\log(\delta)|^{1/2},\lambda_2=\delta^{-\frac{q}{q-1}}$ and thanks to \eqref{weaktype11} and \eqref{es5a} we obtain \eqref{es5} and  \eqref{es5*}.
	The proof is complete.
\end{proof}
\textbf{2.3} \textit{Singular integral operators with rough kernels.}
In this subsection, we provide some basic properties of  singular integral operators with rough convolution kernels. In this paper, we consider the following general kernel in $\mathbb{R}^d$:
\begin{equation}\label{kernelK} \mathbf{K}(x)=\Omega(x)K(x)~~\forall~x\in \mathbb{R}^d\backslash\{0\},
\end{equation} 
where 
\begin{description}
	\item[i.)] $K\in C^1(\mathbb{R}^d\backslash\{0\})$, 
	\begin{align}\label{con1}
	|K(x)|+|x||\nabla K(x)|\leq \frac{1}{|x|^d}~~\forall ~x\in \mathbb{R}^d,
	\end{align}
	\item [ii.)]$\Omega(\theta)=\Omega(r\theta)$ for any $r>0,\theta\in S^{d-1}$ and 
	\begin{align}\label{con2}
	||\Omega||_{W^{\alpha_0,1}(B_2\backslash B_1)}:=\int_{B_2\backslash B_1}|\Omega|+\int_{B_2\backslash B_1}\int_{B_2\backslash B_1}\frac{|\Omega(x)-\Omega(y)|}{|x-y|^{d+\alpha_0}}dxdy\leq c_1,
	\end{align} for some $\alpha_0\in (0,1)$ and $c_1>0$.
	\item[iii.)] (the "cancellation" condition)
	\begin{align}
	\label{cance-condi}
	\sup_{0<R_1<R_2<\infty}\left|\int_{R_1<|x|<R_2}\mathbf{K}(x)dx\right|\leq c_2,
	\end{align}
	for some $c_2>0$.
\end{description}
We say that the kernel $\mathbf{K}$ is a singular Kernel of fundamental type in $\mathbb{R}^d$ if $\Omega\in C^1(S^{d-1})$. 
\begin{remark} From \eqref{con1} one has, 
	\begin{align}\label{con1b}
	|K(x-y)-K(x)|\leq \frac{2^{d+1}|y|}{|x|^{d+1}}~~\forall~ |y|<|x|/2.
	\end{align}
\end{remark}
\begin{remark}\label{re-cance-condi} If $K(x)=|x|^{-d}$ for any $x\in \mathbb{R}^d\backslash\{0\}$, then \eqref{cance-condi} implies  $\int_{S^{d-1}}\Omega(\theta)d\mathcal{H}^{d-1}(\theta)=0$. Moreover, if we set
\begin{equation}\label{omega-n}
\Omega_n(x):=\int_{0}^{\infty}\tilde{\Omega}\star \varrho_{n}\left(\frac{x}{|x|}r\right)r^{d-1}dr~~\forall~x\in\mathbb{R}^d,
\end{equation}
where  $\tilde{\Omega}(x):=\frac{1}{\log(2)}\frac{\Omega(x)}{|x|^d}\mathbf{1}_{1\leq |x|\leq 2}$, then $\int_{S^{d-1}}\Omega_n(\theta)d\mathcal{H}^{d-1}(\theta)=0$ for any $n$, $\Omega_n\in C^\infty_b(S^{d-1})$, $\Omega_n(\theta)=\Omega_n(r\theta)$ for any $r>0,\theta\in S^{d-1}$  and 
\begin{align}\label{con2''}
 ||\Omega_n-\Omega||_{\dot W^{\frac{\alpha_0}{2},1}(B_2\backslash B_1)}\lesssim c_1n^{-\frac{\alpha_0}{2}}~\forall~n.
\end{align} 
\end{remark}
\begin{remark} Since $\Omega(\theta)=\Omega(t\theta)$ for any $t>0,\theta\in S^{d-1}$,  so by Sobolev inequality one gets
	\begin{align}\nonumber
	&||\Omega||_{L^q(S^{d-1})}+\sup_{|h|\leq 1/2}|h|^{-\alpha_0/2}\left(||\Omega(.-h)-\Omega(.)||_{L^q(B_2\backslash B_1)}+||\Omega(.-h)-\Omega(.)||_{L^q(S^{d-1})}\right)\\&\quad\quad\quad\quad\quad+ \left(\int_{1<|x|<2}\sup_{0<\rho< 1/2}\fint_{B_\rho(0)}\frac{|\Omega(x-h)-\Omega(x)|^q}{|h|^{\frac{\alpha_0q}{2}}}dhdx\right)^{1/q}\nonumber\\&\quad\quad\quad\lesssim	||\Omega||_{W^{\alpha_0,1}(B_2\backslash B_1)}\lesssim c_1,\label{con2b}
	\end{align}
	for any $1\leq q\leq q_0=\frac{d}{d-\alpha_0/2}$. Moreover, we also have for any $\lambda_0\leq 2$ and $\theta\in S^{d-1}$
	\begin{equation}\label{Z3}
	\int_{|y|\leq \lambda_0} |\Omega(\theta-y)-\Omega(\theta)|\frac{ dy}{|y|^{d}}\lesssim \lambda_0^{\frac{\alpha_0}{2}}\sup_{r\leq 2} r^{-\frac{\alpha_0}{2}}\fint_{B_r(0)}  |\Omega(\theta-y)-\Omega(\theta)| dy.
 	\end{equation}
\end{remark}
\begin{remark}  Thanks to \eqref{con1} and Minkowski's inequality, one has
	\begin{align}\label{LqforKepso}
	||(\mathbf{1}_{|.|>\varepsilon}\mathbf{K}(.))\star \mu||_{L^{q_0}(\mathbb{R}^d)}\lesssim\varepsilon^{-\frac{(q_0-1)d}{q_0}} ||\Omega||_{L^{q_0}(S^{d-1})}|\mu|(\mathbb{R}^d),
	\end{align}
	for $\varepsilon>0$, $ q_0=\frac{d}{d-\alpha_0/2}$ and $\mu\in \mathcal{M}_b(\mathbb{R}^d)$.
	\end{remark}
The following is $L^p$ and weak type  $(1,1)$ boundedness of singular integral operators associated to the kernel $\mathbf{K}$. 
\begin{proposition}\label{pro-singularintegral-rough1}Let $\mathbf{K}$ be in \eqref{kernelK} with constants $c_1,c_2>0$, $\alpha_0\in (0,1)$. Let $\chi\in C_c(\mathbb{R}^d,[0,1])$ be  such that $\chi=1$ in $|x|>3$ and $\chi=0$ in $|x|<2$. For $f\in C^\infty_c(\mathbb{R}^d)$, we define for $x\in \mathbb{R}^d$, $$
	\mathbf{T}^1(f)(x)=\mathbf{K}\star f(x),~\mathbf{T}^{2}(f)(x)=\sup_{\varepsilon>0}\left|\left(\chi(\frac{\cdot}{\varepsilon})\mathbf{K}\right)\star f(x)\right|,~\mathbf{T}^{3}(f)(x)=\sup_{\varepsilon>0}\left|\left(\mathbf{1}_{|\cdot|>\varepsilon}\mathbf{K}\right)\star f(x)\right|.
$$ Then, $	\mathbf{T}^1$ and 	$\mathbf{T}^{2},\mathbf{T}^{3}$ extend to  bounded operator from $L^p$ to itself ($p>1$) and from $L^1$ to  $L^{1,\infty}$ with norms
\begin{align}\label{Lpweak11}
\sum_{j=1,2,3}||\mathbf{T}^j||_{L^p\to L^p}+||\mathbf{T}^j||_{L^1\to L^{1,\infty}}\lesssim c_1+c_2.
\end{align}
Moreover, we also get 
\begin{align}\label{maxisingular}
\sum_{j=1,2,3}||\mathbf{T}^j||_{\mathcal{M}_b\to L^{1,\infty}}\lesssim c_1+c_2, 
\end{align}
and for any $\mu\in \mathcal{M}_b(\mathbb{R}^d)$, there holds
	\begin{align}\label{esL1infa}
\sum_{j=1,2,3}\limsup_{\lambda\to \infty}\lambda\mathcal{L}^d\left(\{ |\mathbf{T}^j(\mu)|>\lambda\}\right)\lesssim (c_1+c_2)|\mu|^{s}(\mathbb{R}^d).
\end{align}
\end{proposition}
\begin{proof} First, we need to check that 
	\begin{align}\label{standcon1}
	&\sup_{R>0}\int_{R<|x|<2R}|\mathbf{K}(x)|dx\lesssim c_1,\\&
\label{standcon2}
	\sup_{y\not= 0}\int_{|x|\geq 2|y|}\left|\mathbf{K}(x-y)-\mathbf{K}(x)\right|dx \lesssim c_1.
	\end{align}
	Indeed, by \eqref{con1} one has
$$
	\sup_{R>0}\int_{R<|x|<2R}|\mathbf{K}(x)|dx\leq 	\sup_{R>0}\int_{R<|x|<2R}\frac{|\Omega(x)|}{|x|^d}dx=\log(2)\int_{S^{d-1}}|\Omega|\lesssim c_1,
$$
	which implies \eqref{standcon1}. Moreover, for any $y\not= 0$, 
	\begin{align*}
	&\int_{|x|\geq 2|y|}\left|\mathbf{K}(x-y)-\mathbf{K}(x)\right|dx\\&\quad\overset{\eqref{con1b}}\lesssim
	\int_{|x|\geq 2|y|}\frac{|\Omega(x)||y|}{|x|^{d+1}}dx+\int_{|x|\geq 2|y|}\frac{1}{|x|^{d}}\left|\Omega(x-y)-\Omega(x)\right|dx
	\\&\quad\lesssim \int_{2|y|}^\infty \int_{S^{d-1}}|\Omega(\theta)|\frac{|y|}{r^{2}}d\mathcal{H}^{d-1}(\theta) dr+ \sum_{j=1}^{\infty}\frac{1}{(2^{j}|y|)^d}\int_{2^{j}|y|<|x|<2^{j+1}|y|}\left|\Omega(x-y)-\Omega(x)\right|dx\\&\quad\lesssim \int_{S^{d-1}}|\Omega(\theta)| d\mathcal{H}^{d-1}(\theta) +  \int_{0}^{3/4}\sup_{|h|\leq \rho}\int_{1<|x|<2}|\Omega(x-h)-\Omega(x)|dx \frac{d\rho}{\rho}
	\\& \quad\overset{\eqref{con2b}}\lesssim c_1,
	\end{align*}
	which implies \eqref{standcon2}. \vspace{.2cm}\\
	Therefore, $\mathbf{K}$ satisfies \eqref{standcon1}, \eqref{standcon2} and \eqref{cance-condi}, so by \cite[Theorem 5.4.1, 5.4.5 and 5.3.5]{Gra1},  we obtain \eqref{Lpweak11}.\vspace{.2cm}\\
	We now prove \eqref{maxisingular}. Let $\mu\in \mathcal{M}_b(\mathbb{R}^d)$. Thanks to \eqref{LqforKepso}, one has $
(\chi(./\varepsilon)\mathbf{K})\star \mu\in L^1_{\loc}(\mathbb{R}^d)$ for any $\varepsilon>0$.
	Thus, for any $\varepsilon>0$, $
	\lim\limits_{n\to\infty} 	(\chi(./\varepsilon)\mathbf{K})\star (\varrho_n\star\mu)= (\chi(./\varepsilon)\mathbf{K})\star \mu~\text{a.e~in}~\mathbb{R}^d$. This implies that  $
\mathbf{1}_{|\mathbf{T}^2(\mu)|>\lambda}\leq \liminf_{n\to\infty} \mathbf{1}_{|\mathbf{T}^2(\varrho_n\star\mu)|>\lambda}$ a.e in $\mathbb{R}^d$,  for any $\lambda>0$. On the other hand, by \eqref{Lpweak11}, 
$$
\sup_{\lambda>0}\lambda \mathcal{L}^d\left(\left\{|\mathbf{T}^2(\varrho_n\star\mu)|>\lambda\right\}\right)\lesssim (c_1+c_2)||\varrho_n\star\mu||_{L^1(\mathbb{R}^d)}\lesssim(c_1+c_2)|\mu|(\mathbb{R}^d).$$
By applying Fatou's lemma, we find $$
\sup_{\lambda>0}\lambda \mathcal{L}^d\left(\left\{|\mathbf{T}^2(\mu)|>\lambda\right\}\right)\lesssim (c_1+c_2)|\mu|(\mathbb{R}^d).$$ Similarly, we also get $$
\sup_{\lambda>0}\lambda \mathcal{L}^d\left(\left\{|\mathbf{T}^{3}(\mu)|>\lambda\right\}\right)\lesssim (c_1+c_2)|\mu|(\mathbb{R}^d).$$
Hence, we conclude \eqref{maxisingular} since $|\mathbf{T}^1(\mu)|\leq |\mathbf{T}^{3}(\mu)|$. To get \eqref{esL1infa}, one has for $R>1$ and $\gamma>1$
\begin{align*}
&	\lambda\mathcal{L}^d\left(\{ |\mathbf{T}^j(\mu)|>\lambda\}\right)\leq \lambda\mathcal{L}^d\left(\{ |\mathbf{T}^j(\mu^{s})|>\lambda/4\}\right)+\lambda\mathcal{L}^d\left(\{ |\mathbf{T}^j(\mu^{\text{a}}\mathbf{1}_{B_R^c})|>\lambda/4\}\right)\\&~~~~~~~~~~~~~~~~~+\lambda\mathcal{L}^d\left(\{ |\mathbf{T}^j(\mu^{\text{a}}\mathbf{1}_{|\mu|^a> \gamma}\mathbf{1}_{B_R})|>\lambda/4\}\right)+\lambda\mathcal{L}^d\left(\{ |\mathbf{T}(\mu^{\text{a}}1_{|\mu|^a\leq \gamma}\mathbf{1}_{B_R})|>\lambda/4\}\right).
\end{align*} 
Using the boundedness of $\mathbf{T}$ from $\mathcal{M}_b(\mathbb{R}^d)$ to $L^{1,\infty}(\mathbb{R}^d)$  for first three terms and 	the boundedness of $\mathbf{T}$ from $L^2(\mathbb{R}^d)$ to itself for last term yields 
$$
\lambda\mathcal{L}^d\left(\{ |\mathbf{T}^j(\mu)|>\lambda\}\right)\lesssim(c_1+c_2)\left( |\mu|^{s}(\mathbb{R}^d)+\int_{B_R^c}|\mu|^a+ \int_{B_R}\mathbf{1}_{|\mu|^a> \gamma}|\mu|^a+\lambda^{-1}\int_{B_R}\mathbf{1}_{|\mu|^a\leq  \gamma}(|\mu|^a)^2\right).$$
This implies \eqref{esL1infa} by letting $\lambda\to \infty$ and then $\gamma\to\infty$, $R\to \infty$. The proof is complete.
\end{proof}
\begin{remark}\label{weak(11)inf1} Since $\mathbf{T}^j(\mathbf{1}_{B_{R+\varepsilon}^c}\mu)\in L^1(B_R)$ for any  $B_R\subset \mathbb{R}^d$ and $\varepsilon>0$, so 
	\begin{equation*}
	\limsup_{\lambda\to \infty}\lambda\mathcal{L}^d\left(\{ |\mathbf{T}^j(\mu)|>\lambda\}\cap B_R\right)\leq \limsup_{\lambda\to \infty}\lambda\mathcal{L}^d\left(\{ |\mathbf{T}^j(\mathbf{1}_{B_{R+\varepsilon}}\mu)|>\lambda/2\}\right)~\forall~\varepsilon>0.
	\end{equation*}
	Applying \eqref{esL1infa} to $\mathbf{1}_{B_{R+\varepsilon}}\mu$ and then letting $\varepsilon\to 0$, we find that  
		\begin{equation*}
	\sum_{j=1,2,3}\limsup_{\lambda\to \infty}\lambda\mathcal{L}^d\left(\{ |\mathbf{T}^j(\mu)|>\lambda\}\cap B_R\right)\lesssim (c_1+c_2)|\mu|^s(\overline{B}_R).
	\end{equation*}
\end{remark}
\begin{remark} It is unknown when $\mathbf{T}^{3}$ is bounded from $L^1(\mathbb{R}^d)$ to $L^{1,\infty}(\mathbb{R}^d)$ where $\Omega\in W^{\alpha_0,1}(B_2\backslash B_1)$ is replaced by $\Omega\in L^q(S^{d-1})$ for $q>1$. This is an interesting open problem posed by A. Seeger. 
\end{remark}
In this paper, we also need a boundedness of the following  singular maximal operator: for $f\in L^1(\mathbb{R}^d)$,
\begin{equation*}
\mathbf{M}^{\tilde{\Omega}}f(x)=\sup_{\rho>0}\fint_{B_\rho(x)}|\tilde{\Omega}(y/|y|)||f(x-y)|dy~~\forall~x\in \mathbb{R}^d,
\end{equation*}
where $\tilde{\Omega}\in L^1(S^{d-1})$.
\begin{proposition}[\cite{Stein2},\cite{Christ1},\cite{Christ2}]\label{pro-singularmaxi}There hold for  $p>1, q>1,$
	\begin{equation}\label{singularmaximala}
	||\mathbf{M}^{\tilde{\Omega}}||_{L^p\to L^p}\lesssim||\tilde{\Omega}||_{L^1(S^{d-1})}, ~~
||\mathbf{M}^{\tilde{\Omega}}||_{L^1\to L^{1,\infty} }\lesssim||\tilde{\Omega}||_{L^q(S^{d-1})}.
	\end{equation}	
\end{proposition}
By a standard approximation, we  obtain  from \eqref{singularmaximala} that for $q>1$.
	\begin{align}\label{singularmaximalb}
	||\mathbf{M}^{\tilde{\Omega}}||_{\mathcal{M}_b\to L^{1,\infty} }\lesssim||\tilde{\Omega}||_{L^q(S^{d-1})}.
\end{align}

\begin{proposition} \label{singu-operator} Let $\mathbf{K}$ be in \eqref{kernelK} with constants $c_1,c_2>0$, $\alpha_0\in (0,1)$. Let $\{\phi^e\}_e\subset C^1(\mathbb{R}^d\backslash \{0\})\cap L^\infty(\mathbb{R}^d)$ be a family of kernels such that $$
		\supp (\phi^e)\subset B_1,~~\sup_{x\in\mathbb{R}^d,e}|\phi^e(x)|+|x||\nabla \phi^e(x)|\leq c_0.$$
	For $\alpha\in (0,d)$ and $f\in C_c^\infty(\mathbb{R}^d)$ we define $$
		\mathbf{T}(f)(x)=\sup_{e}\sup_{\rho>0}\left|\left(\frac{\rho^{-\alpha}}{|\cdot|^{d-\alpha}}\phi^e(\frac{\cdot}{\rho})\right)\star\mathbf{K}\star f  (x)\right|~\forall~x\in \mathbb{R}^d.$$
	Then, $	\mathbf{T}$ extends to  bounded operator from $L^p(\mathbb{R}^d)$ to itself ($p>1$) and $L^1(\mathbb{R}^d)$ to $L^{1,\infty}(\mathbb{R}^d)$ with norms
	\begin{equation}\label{Lpweakmaxisingular1}
	||\mathbf{T}||_{L^p\to L^p}+||\mathbf{T}||_{\mathcal{M}_b\to L^{1,\infty}}\lesssim c_0(c_1+c_2).
	\end{equation}
	Moreover, for any $\mu\in \mathcal{M}_b(\mathbb{R}^d)$  
	\begin{equation}\label{esL1inf}
		\limsup_{\lambda\to \infty}\lambda\mathcal{L}^d\left(\{ \mathbf{T}(\mu)>\lambda\}\right)\lesssim c_0(c_1+c_2)|\mu|^{s}(\mathbb{R}^d),
	\end{equation}
	In particular,  for any $B_R\subset\mathbb{R}^d$ and $f\in L^q(B_R)$ for $q>1$, 
	\begin{equation}\label{es6-}
	\limsup_{\delta\to 0}\frac{1}{|\log (\delta)|}\int_{B_R}\min\left\{\delta^{-1}|f|,\mathbf{T}(\mu)\right\} \lesssim c_0(c_1+c_2) |\mu|^{s}(\overline{B}_R).
	\end{equation}
\end{proposition}
Proposition \ref{singu-operator} is still true for any $\alpha\geq d$. This was proven in \cite{BoCrip} for smooth kernel case (i.e  $\Omega \in C^1_b(S^{d-1})$). 

\begin{proof}[Proof of Proposition \ref{singu-operator}]  For  $f\in C_c^\infty(\mathbb{R}^d)$,  we set
$$
\mathbf{T}_1(f)(x)=\sup_{\rho>0}\left|\int_{|x-z|>2\rho}\mathbf{K}(x-z)f(z) dz\right|,~~~	\mathbf{T}_2(f)(x)=\sup_{e}\sup_{\rho>0}\left| \mathbf{K}_{e,\rho}\star f(x)\right|$$
for  any $x\in \mathbb{R}^{d}$, where $$
\mathbf{K}_{e,\rho}(x)=\int_{\mathbb{R}^d}\frac{\rho^{-\alpha}}{|y|^{d-\alpha}}\phi^e(\frac{y}{\rho})\mathbf{K}(x-y)dy-\int_{\mathbb{R}^d}\frac{1}{|y|^{d-\alpha}}\phi^e(y)dy \mathbf{1}_{|x|>2\rho}\mathbf{K}(x).$$
We show that 
	\begin{equation}\label{es8}
	|\mathbf{K}_{e,\rho}(x)|\lesssim \frac{c_0}{|x|^{d-\alpha}}\min\left\{\frac{1}{\rho^\alpha}, \frac{\rho^{\frac{\alpha_0}{2}}}{|x|^{\frac{\alpha_0}{2}+\alpha}}\right\}\Omega_1(x/|x|)~~\forall~x\in\mathbb{R}^d,
	\end{equation}
		where 
	\begin{align*}
	\Omega_1(\theta)=c_1+c_2+|\Omega(\theta)|+\sup_{r\in (0,2)}r^{-\frac{\alpha_0}{2}}\fint_{B_r(0)}\left|\Omega(\theta-z)-\Omega(\theta)\right|dz~~\forall~\theta\in S^{d-1}.
	\end{align*}
It is enough to prove that 	
\begin{align}
\label{es6+}
\left|\int_{\mathbb{R}^d}\mathbf{K}(x-y)\frac{1}{|y|^{d-\alpha}}\phi^e(\frac{y}{\rho})dy\right|&\lesssim \frac{c_0}{|x|^{d-\alpha}} \Omega_1(x/|x|)\quad\text{if}~~|x|\leq 2\rho,\\|\mathbf{K}_{e,\rho}(x)|&\lesssim  \frac{c_0\rho^{\frac{\alpha_0}{2}}}{|x|^{d+\frac{\alpha_0}{2}}} \Omega_1(x/|x|)\quad\text{if}~~|x|\geq 2\rho.\label{es6++}
\end{align}
	Indeed, \textit{Case:} $|x|\leq 2\rho$. Assume that $2^{-j_0}\rho<|x|\leq 2^{-j_0+1}\rho$ for $j_0\in \mathbb{N}$. \\ Let $\chi$ be a smooth function in $\mathbb{R}^d$ such that $\chi(y)=1$ if $|y|\leq 1$ and $\chi(y)=0$ if $|y|>\frac{11}{10}$.\\
	 Set 
	$$
	\mathbf{K}_0(x)=\int_{\mathbb{R}^d}\mathbf{K}(x-y)\frac{1}{|y|^{d-\alpha}}\phi^e(\frac{y}{\rho})\left(\chi(2^{j_0-2}\rho^{-1} y)-\chi(2^{j_0+1}\rho^{-1} y)\right)dy.$$ By definition of $\mathbf{K}$ and $\phi^e$, we have, 
	\begin{align}\nonumber
		&\left|\int_{\mathbb{R}^d}\mathbf{K}(x-y)\frac{1}{|y|^{d-\alpha}}\phi^e(\frac{y}{\rho})dy\right|\\&\quad\quad\lesssim 	\mathbf{K}_0(x)+c_0\int_{|y|\geq  2|x|}\frac{|\Omega(x-y)|}{|x-y|^d|y|^{d-\alpha}}dy+c_0\int_{|y|\leq |x|/2}\frac{|\Omega(x-y)|}{|x-y|^d|y|^{d-\alpha}}dy
	\nonumber	\\&\quad\quad\lesssim \nonumber	\mathbf{K}_0(x)+c_0\int_{|x-y|\geq  |x|}\frac{|\Omega(x-y)|}{|y-x|^{2d-\alpha}}dy+\frac{c_0}{|x|^{d-\alpha}}\int_{|y|\leq |x|/2}|\Omega(x-y)-\Omega(x)|\frac{dy}{|y|^{d}}+\frac{c_0|\Omega(x)|}{|x|^{d-\alpha}}
			\\&\quad\quad\lesssim 	\mathbf{K}_0(x)+c_0\frac{||\Omega||_{L^1(S^{d-1})}+|\Omega(x)|}{|x|^{d-\alpha}}+\frac{c_0}{|x|^{d-\alpha}}\int_{|y|\leq 1/2}|\Omega(\frac{x}{|x|}-y)-\Omega(\frac{x}{|x|})|\frac{dy}{|y|^d}\nonumber\\&\quad\quad\lesssim \mathbf{K}_0(x)+\frac{c_0}{|x|^{d-\alpha}} \Omega_1(x/|x|) .\label{Z2}
	\end{align}
	Here we have used \eqref{Z3} in the last inequality. \\
	Thanks to  Gagliardo-Nirenberg interpolation inequality, we find  $$
	\sup_{|y|\sim |x|}|\mathbf{K}_0(y)|\lesssim |x|^{1/2}||\nabla \mathbf{K}_0||_{L^{2d}(\mathbb{R}^d)}+\frac{1}{|x|^{d/2}}\left(\int_{\mathbb{R}^d}|\mathbf{K}_0|^2dy\right)^{1/2}.$$
	By boundedness of operator $f\mapsto \mathbf{K}\star f$ in $L^2$ and $\dot W^{1,2d}$ ( see Proposition  \ref{pro-singularintegral-rough1}), one obtains
	\begin{align}\nonumber
	\sup_{|y|\sim |x|}|\mathbf{K}_0(y)|&\lesssim (c_1+c_2) |x|^{1/2}\left(\int_{\mathbb{R}^d}\left|\nabla \left(\frac{\phi^e(\frac{y}{\rho})}{|y|^{d-\alpha}}\left(\chi(2^{j}\rho^{-1} y)-\chi(2^{j+1}\rho^{-1} y)\right)\right)\right|^{2d} dy\right)^{\frac{1}{2d}}\\&\nonumber+(c_1+c_2)\frac{1}{|x|^{d/2}}\left(\int_{\mathbb{R}^d}\left|\frac{\phi^e(\frac{y}{\rho})}{|y|^{d-\alpha}}\left(\chi(2^{j}\rho^{-1} y)-\chi(2^{j+1}\rho^{-1} y)\right)\right|^2dy\right)^{1/2}\\&\nonumber
	\lesssim c_0(c_1+c_2) \left(|x|^{1/2}(2^{-j_0}\rho)^{-d+\alpha-\frac{1}{2}}+\frac{1}{|x|^{d/2}} (2^{-j_0}\rho)^{-d/2+\alpha}\right)
	\\&\lesssim  \frac{c_0(c_1+c_2)}{|x|^{d-\alpha}} \label{es5+}.
	\end{align}
	Thus, it follows \eqref{es6+} from \eqref{Z2}  and \eqref{es5+}. \vspace{.2cm}\\
	\textit{Case:} $|x|> 2\rho$.  By \eqref{con1},\eqref{con1b}, we have 
	\begin{align}\nonumber
	|\mathbf{K}_{e,\rho}(x)|&\leq|\int_{\mathbb{R}^d}\frac{\rho^{-\alpha}}{|y|^{d-\alpha}}\phi^e(\frac{y}{\rho})\left(\mathbf{K}(x-y)-\mathbf{K}(x)\right)dy|\\&\nonumber\lesssim c_0|\Omega(x)|\int_{|y|<\rho}\frac{\rho^{-\alpha}}{|y|^{d-\alpha}}\frac{|y|}{|x|^{d+1}}dy+\frac{c_0}{|x|^d}\int_{|y|<\rho}\frac{\rho^{-\alpha}}{|y|^{d-\alpha}}|\Omega(x-y)-\Omega(x)|dy
	\\&\nonumber \lesssim \frac{c_0|\Omega(x)|\rho}{|x|^{d+1}}+\frac{c_0}{|x|^d}\int_{|y|<\rho/|x|}|\Omega(\frac{x}{|x|}-y)-\Omega(\frac{x}{|x|})|\frac{dy}{|y|^{d}}\\&\overset{\eqref{Z3}}\lesssim \frac{c_0\rho}{|x|^{d+1}}\Omega_1(x/|x|)+\frac{c_0\rho^{\frac{\alpha_0}{2}}}{|x|^{d+\frac{\alpha_0}{2}}}\Omega_1(x/|x|),\label{7april1}
	\end{align}
	which implies \eqref{es6++}.  \vspace{.3cm}\\
	Then, \eqref{es8} gives 
	\begin{align}\label{es7+}
	|\mathbf{T}(f)|\lesssim c_0\mathbf{T}_1(f)+\mathbf{T}_2(f)\lesssim 	c_0\mathbf{T}_1(f)+ c_0 \mathbf{M}^{\Omega_1}(f).
	\end{align}
 Thanks to Proposition \ref{pro-singularintegral-rough1} and \ref{pro-singularmaxi} and using the fact that 
$
 ||\Omega_1||_{L^{q}(S^{d-1})}\overset{\eqref{con2b}}\lesssim c_1+c_2,$ we get  $$
	||\left(\mathbf{T}^1, \mathbf{M}^{\Omega_1}\right)||_{L^p\to L^p}+||\left(\mathbf{T}^1, \mathbf{M}^{\Omega_1}\right)||_{L^1\to L^{1,\infty}}\lesssim  c_0(c_1+c_2).$$
	This gives \eqref{Lpweakmaxisingular1}.  Then, similar to the proof of \eqref{esL1infa} and \eqref{es5}, we  obtain \eqref{esL1inf} and \eqref{es6-} from \eqref{Lpweakmaxisingular1}.
	The proof is complete.
\end{proof}
\begin{lemma}
	\label{Remar0} 	We denote for $\rho_0>0$, and $\alpha_1\in (0,\alpha]$ and $\mu\in \mathcal{M}_b(\mathbb{R}^d)$
	$$
	\mathbf{T}^{\alpha_1}(\mu)(x)=\sup_{e}\sup_{\rho\in (0,\rho_0)}\left|\left(\frac{\rho^{\alpha_1-\alpha}}{|.|^{d-\alpha}}\phi^e(\frac{.}{\rho})\right)\star\mathbf{K}\star \mu (x)\right|~\forall~x\in \mathbb{R}^d.$$
	Then,
	\begin{align}\label{estmaxi-Riez}
	||	\mathbf{T}^{\alpha_1}(\mu)||_{L^{q_0}(\mathbb{R}^d)}\lesssim(c_1+c_2)|\mu|(\mathbb{R}^d), ~~ q_0=\frac{d}{d-\frac{1}{4}\min\{\alpha,\alpha_0,\alpha_1\}}>1.
	\end{align}
\end{lemma}
\begin{proof}We deduce from \eqref{es6+} and \eqref{es6++} that for any $x\in \mathbb{R}^d,$ 
	\begin{align*}
	|\left(\frac{\rho^{-\alpha}}{|.|^{d-\alpha}}\phi^e(\frac{.}{\rho})\right)\star\mathbf{K}(x)|&\lesssim \left(\frac{\rho^{-\alpha}}{|x|^{d-\alpha}}\wedge\frac{1}{|x|^d}\right)\Omega_1(\frac{x}{|x|})\\&\lesssim  (1+\rho^{\alpha_1})P(x),
	\end{align*}
	with $$P(x)=\left(\frac{1}{|x|^{d-\alpha_1}}\wedge\frac{1}{|x|^d}\right)\Omega_1(\frac{x}{|x|}).$$ Thus, $$\mathbf{T}^{\alpha_1}(\mu)(x)\lesssim (1+\rho_0^{\alpha_1}) P\star|\mu|(x).$$  Then, by Minkowski's inequality, one has $$
	||	\mathbf{T}^{\alpha_1}(\mu)||_{L^{q_0}(\mathbb{R}^d)}\lesssim (1+\rho_0^{\alpha_1}) ||\Omega_1||_{L^{q_0}(S^{d-1})}|\mu|(\mathbb{R}^d),$$
	which implies \eqref{estmaxi-Riez}. 
\end{proof}
\begin{remark}\label{Remar1} As Lemma \eqref{Remar0}, we also show that for $\rho_0>0$ and  $\mu\in \mathcal{M}_b(\mathbb{R}^d)$,
	\begin{align}\label{diff1}
\mathbf{P}(\mu)(x)=	\sup_{e}\sup_{\rho\in (0,\rho_0)}|\left(\frac{\rho^{-\alpha}}{|.|^{d-\alpha}}\phi^e(\frac{.}{\rho})\right)\star\mathbf{K}\star((\psi(.)-\psi(x))\mu)(x)|\in L^{q_0}_{\loc}(\mathbb{R}^d),
	\end{align}
	for some $q_0>1$, with $\psi\in W^{1,\infty}(\mathbb{R}^d)$, exactly 
	\begin{align}
	||\mathbf{P}(\mu)||_{L^{q_0}(B_R(0))}\lesssim_{R}||\psi||_{W^{1,\infty}(\mathbb{R}^d)}|\mu|(\mathbb{R}^d)~~\forall ~R>0.
	\end{align}
	Furthermore, if $\Omega\in C^1_b(S^{d-1})$ then $\mathbf{P}(\mu)(x)\lesssim \mathbf{I}_1(\mu)(x)$.
\end{remark}
\begin{remark}\label{rem3} If $\mu_t(x)=\mu(t,x)\in L^1([0,T],\mathcal{M}_b(\mathbb{R}^d))$ and $f\in L^1((0,T),L^q(B_R))$ for $q>1$, it follows from \eqref{es6-} and  the dominated convergence theorem that 
	\begin{align}\label{es6--}
	\limsup_{\delta\to 0}\frac{1}{|\log (\delta)|}\int_{0}^{T}\int_{B_R}\min\left\{\delta^{-1}|f(x,t)|,\mathbf{T}(\mu_t)(x)\right\}dx dt\lesssim \int_{0}^{T}|\mu_t|^{s}(\overline{B}_R)dt.
	\end{align}
\end{remark}
\begin{remark} We do not know how to prove Proposition \ref{singu-operator} when $\alpha_0=0$.
\end{remark}
\section{Kakeya singular integral operators}
This section we introduce the Kakeya singular integral operators and establish a strong version of \eqref{esL1inf} for this operator. It is a main tool of this paper. \\

Assume that  $\{\phi^{e,\varepsilon}\}_{\varepsilon,e}\subset C^1(\mathbb{R}^d\backslash \{0\},\mathbb{R}^d)\cap L^\infty_c(\mathbb{R}^d,\mathbb{R}^d)$ is a family of kernels such that 
\begin{align}\label{conmai}
\supp (\phi^{e,\varepsilon})\subset B_1(0)\cap\left\{x: \left|e-\frac{x}{|x|}\right|\leq \varepsilon\right\},~~and~~ |\phi^{e,\varepsilon}(x)|+\varepsilon|x||\nabla \phi^{e,\varepsilon}(x)| \leq c_0
\end{align}
for any $x\in \mathbb{R}^d,\varepsilon\in (0,1/10),~e\in S^{d-1}$. Let $\mathbf{K}$ be in \eqref{kernelK} with constants $c_1,c_2>0$, $\alpha_0\in (0,1)$. Assume that there exists a sequence of $\Omega_n\in C^2_b(S^{d-1})$ such that 
$\Omega_n(\theta)=\Omega_n(r\theta)$ for any $r>0,\theta\in S^{d-1}$ and 
\begin{align}\label{con2'}
||\Omega_n||_{W^{\alpha_0,1}(B_2\backslash B_1)}\leq 2c_1, ~~\lim\limits_{n\to\infty} ||\Omega_n-\Omega||_{W^{\alpha_0,1}(B_2\backslash B_1)}=0,
\end{align} 
and $\mathbf{K}_n(x):=\Omega_n(x)K(x)$ satisfies \eqref{cance-condi} i.e 
	\begin{align}
\label{cance-condi'}
\sup_{0<R_1<R_2<0}\left|\int_{R_1<|x|<R_2}\mathbf{K}_n(x)dx\right|\leq c_3
\end{align}
for some $c_3>0$, moreover, 
\begin{align}
\label{cance-condi''}
\lim\limits_{n\to\infty}\sup_{0<R_1<R_2<0}\left|\int_{R_1<|x|<R_2}\left(\mathbf{K}_n(x)-\mathbf{K}(x)\right)dx\right|=0.
\end{align}
For any $\mu\in C_c^\infty(\mathbb{R}^d,\mathbb{R}^d)$ and $\rho_0>0$, the Kakeya singular integral operator $\mathbf{T}_{\varepsilon}$ is given by 
\begin{align}
\mathbf{T}_{\varepsilon}(\mu)(x):=\mathbf{T}_{\varepsilon}^{\mathbf{K}}(\mu)(x)&=\sup_{\rho\in (0,\rho_0),e\in S^{d-1}}\frac{\varepsilon^{-d+1}}{\rho^\alpha}\left|\left(\frac{1}{|\cdot|^{d-\alpha}}\phi^{e,\varepsilon}_\rho(\cdot)\right)\star \mathbf{K}\star\mu(x)\right|~~\forall~x\in\mathbb{R}^d,\label{DET}
\end{align}
for some $\alpha\in (0,d)$, where $\phi^{e,\varepsilon}_{\rho}(\cdot)=\phi^{e,\varepsilon}(\frac{\cdot}{\rho})$. Set $$\mathbf{T}_{\varepsilon}^{1,n}:=\mathbf{T}_{\varepsilon}^{\mathbf{K}_n},~~~ \mathbf{T}_{\varepsilon}^{2,n}:=\mathbf{T}_{\varepsilon}^{\mathbf{K}_n-\mathbf{K}}.$$
Thanks to Proposition \ref{singu-operator} and conditions \eqref{con2'}, \eqref{cance-condi'}, \eqref{cance-condi''},  we have
 for any $\mu\in \mathcal{M}_b(\mathbb{R}^d)$  
\begin{align}\label{esL1inf'}
&\limsup_{\lambda\to \infty}\lambda\mathcal{L}^d\left(\{ \mathbf{T}_\varepsilon(\mu)>\lambda\}\right)+\lambda\mathcal{L}^d\left(\{ \mathbf{T}_\varepsilon^{1,n}(\mu)>\lambda\}\right)\lesssim\varepsilon^{-d+1}(c_1+c_2)|\mu|^{s}(\mathbb{R}^d), \\& \label{esL1inf''}
\limsup_{\lambda\to \infty}\lambda\mathcal{L}^d\left(\{ \mathbf{T}_\varepsilon^{2,n}(\mu)>\lambda\}\right)\lesssim\varepsilon^{-d+1}c_n|\mu|^{s}(\mathbb{R}^d),
\end{align} 
for any $\varepsilon\in (0,1/10),\forall n$, where $\lim\limits_{n\to\infty}c_n=0$.\\
\begin{remark}\label{remarforomega}
	$\Omega_n$ in Remark \ref{re-cance-condi} satisfies \eqref{con2'}, \eqref{cance-condi'} and \eqref{cance-condi''}.
\end{remark}
\begin{remark}
 if $\mathbf{K}=\sum_{j=1}^{d}\mathcal{R}_j^2=\delta_{0}$ where $\mathcal{R}_1,...,\mathcal{R}_d$ are the Riesz transforms in $\mathbb{R}^d$,  we thus get $\mathbf{T}_\varepsilon(\mu)\lesssim \mathbf{M}^\varepsilon(\mu),$ where $\mathbf{M}^\varepsilon$ is the Kakeya maximal function in $\mathbb{R}^d$ i.e $$
\mathbf{M}^\varepsilon(\mu)(x)=\sup_{\rho>0,e\in S^{d-1}}\fint_{B_\rho(x)}\varepsilon^{-d+1}\mathbf{1}_{|\frac{z-x}{|z-x|}-e|\leq \varepsilon}d|\mu|(z),~~\forall~x\in \mathbb{R}^d.$$
\end{remark}
	
	Our main result is the following:
\begin{theorem}\label{mainthm1}  Assume that $\mu=Df$, $f\in BV(\mathbb{R}^d)$. Then, we have 
	\begin{align}\label{esweakL1'}
	\limsup_{\lambda\to \infty}\lambda\mathcal{L}^d\left(\left\{ \mathbf{T}_{\varepsilon}(\mu)>\lambda \right\}\right)\lesssim|\log(\varepsilon)||\mu|^{s}(\mathbb{R}^d),
	\end{align}
	for any $\varepsilon\in (0,1/10)$. In particular, for any $B_R\subset\mathbb{R}^d$ and $f\in L^q(B_R)$ for $q>1$, 
	\begin{align}\label{es6}
		\limsup_{\delta\to 0}\frac{1}{|\log (\delta)|}\int_{B_R}\min\left\{\delta^{-1}|f|,\mathbf{T}_{\varepsilon}(\mu)\right\}dx \lesssim |\log(\varepsilon)| |\mu|^{s}(\overline{B}_R),
	\end{align}
	for any $\varepsilon\in (0,1/10)$.
\end{theorem}
\begin{remark} If $\mathbf{K}=\sum_{j=1}^{d}\mathcal{R}_j^2=c(d)\delta_{0}$ where $\mathcal{R}_1,...,\mathcal{R}_d$ are the Riesz transforms in $\mathbb{R}^d$, we can remove  $|\log(\varepsilon)|$ in \eqref{esweakL1'}.
\end{remark}
\begin{remark}\label{re2}
	Estimate \eqref{esweakL1'} is not true for all $\mu\in \mathcal{M}_{b}(\mathbb{R}^d,\mathbb{R}^d)$. Indeed,  let $d\mu=d\delta_{\{0\}}$ and $|\phi^{e,\varepsilon}(e)|\geq 1$ for any $e\in S^{d-1}$ and $\varepsilon>0$, let $\mathbf{T}_{j,\varepsilon}$ be $\mathbf{T}_\varepsilon$ associated to  $\mathbf{K}(x)=\mathbf{K}_j(x)=\frac{|x|^2-x_j^2d}{|x|^{d+2}}$.\\ One has $
	\sum_{j=1}^{d}\mathbf{T}_{j,\varepsilon}(\mu)(x)\gtrsim\frac{\varepsilon^{-d+1}}{|x|^{d}}|\phi^{x/|x|,\varepsilon}(x/|x|)|\gtrsim\frac{\varepsilon^{-d+1}}{|x|^{d}}.$
	Thus,  for any $\lambda>1$	$$
	\lambda\mathcal{L}^d\left(\left\{ \mathbf{T}_{1,\varepsilon}(\mu)>\lambda \right\}\right)\geq d^{-1}\mathcal{L}^d(\{ \sum_{j=1}^{d}\mathbf{T}_{j,\varepsilon}(\mu)>d\lambda \})\gtrsim \varepsilon^{-d+1}.$$
	It is well-known that $\{D_{x_1}f\in\mathcal{M}_b(\mathbb{R}^d):f\in BV(\mathbb{R}^d)\}$ is not dense in $L^1(\mathbb{R}^{d-1},\mathcal{M}_{b,x_1}(\mathbb{R}))$.
	So,  a natural question is that \textit{whether} \eqref{esweakL1'}  holds for any $\mu\in L^1(\mathbb{R}^{d-1},\mathcal{M}_b(\mathbb{R}))$.
\end{remark}

To prove Theorem \ref{mainthm1},  we need the following lemmas: 
\begin{lemma}\label{le-simineq}  Let $\omega\in \mathcal{M}_b^+(\mathbb{R})$ and $a:\mathbb{R}^d\to\mathbb{R}^+$ be a Borel function. Then,  for any $\rho>0$, 
	\begin{align}\label{simineq}
		\int_{\mathbb{R}^{d-1}}\int_{\mathbb{R}}\mathbf{1}_{\rho<|y_1+y_2|\leq 2 \rho} a(y_1+y_2) d\omega(y_1)d\mathcal{H}^{d-1}(y_2)\lesssim \rho^d \left[\int_{S^{d-1}}\sup_{r\in [\rho,2\rho]}a(r\theta)d\mathcal{H}^{d-1}(\theta)\right]\mathbf{M}^1(\omega,\mathbb{R})(0).
	\end{align}
\end{lemma}
\begin{proof}[Proof of Lemma \ref{le-simineq}]  Let $d\omega_{\kappa}(y)=\mathbf{1}_{|y|>\kappa}d\omega(y)$ for $\kappa\in (0,\rho/100)$. Let $\varrho_{m}$ be a standard sequence of mollifiers in $\mathbb{R}$. For any $m>4/\kappa$, we have $\supp(\varrho_{m}\star\omega_{\kappa})\subset \{z:|z|>\kappa/2\}$  and 
	\begin{align*}
		&	\int_{\mathbb{R}^{d-1}}\int_{\mathbb{R}}\mathbf{1}_{\rho<|y_1+y_2|\leq 2 \rho} a(y_1+y_2) (\varrho_{m}\star\omega_{\kappa})(y_1)d\mathcal{H}^{1}(y_1)d\mathcal{H}^{d-1}(y_2)\\&\quad\quad\quad=\int_{S^{d-1}}\int_{\rho}^{2\rho}r^{d-1}a(r\theta) \mathbf{1}_{|r\theta_1|>\kappa/2} (\varrho_{m}\star\omega_{\kappa})(r\theta_1)drd\mathcal{H}^{d-1}(\theta)\\&\quad\quad\quad\leq  (2\rho)^{d-1}\int_{S^{d-1}} \left(\sup_{r'\in [\rho,2\rho]}a(r'\theta)\right)\int_{\rho}^{2\rho} \mathbf{1}_{|r\theta_1|>\kappa/2} (\varrho_{m}\star\omega_{\kappa})(r\theta_1)drd\mathcal{H}^{d-1}(\theta).
	\end{align*}
	On the other hand, 	
	\begin{align*}
		&\int_{\rho}^{2\rho} \mathbf{1}_{|r\theta_1|>\kappa/2} (\varrho_{m}\star\omega_{\kappa})(r\theta_1)dr\leq\int_{\mathbb{R}}\int_{\rho}^{2\rho} \mathbf{1}_{|r\theta_1|>\kappa/2}\varrho_{m}(r\theta_1-z)drd\omega(z)\\&~~~\leq \frac{\mathbf{1}_{|\theta_1|>\frac{\kappa}{4\rho}}}{|\theta_1|}\int_\mathbb{R}\int_{-2|\theta_1|\rho}^{2|\theta_1|\rho}\varrho_{m}(r-z)drd\omega(z)\leq \frac{\mathbf{1}_{|\theta_1|>\frac{\kappa}{4\rho}}}{|\theta_1|}\int_\mathbb{R}\mathbf{1}_{|z|<2|\theta_1|\rho+\frac{2}{m}}d\omega(z)
		\\&~~\leq \frac{\mathbf{1}_{|\theta_1|>\frac{\kappa}{4\rho}}}{|\theta_1|}\int_{-4|\theta_1|\rho}^{4|\theta_1|\rho}d\omega(z)\leq 8\rho \mathbf{M}^1(\omega,\mathbb{R})(0).
	\end{align*}	Thus, by Fatou's Lemma, letting $m\to \infty$ and then $\kappa\to \infty$ we get \eqref{simineq}. The proof is complete.  
\end{proof}
\begin{remark} \label{rekakeyamaxi}From proof of Lemma \ref{le-simineq} we can see that for any $e_0\in S^{d-1}$ and $\mu\in \mathcal{M}_b^+(\mathbb{R}^d)$ and $\omega\in \mathcal{M}_b^+(\tilde{H}_{e_0})$ if $\mu\leq \omega\otimes\mathcal{H}^{d-1}$ then 
	\begin{align}\label{key-es}
	\mathbf{M}^\varepsilon(\mu)(x)\lesssim \mathbf{M}^1(\omega,\tilde{H}_{e_0})(\langle e_0,x\rangle e_0)~~\forall~x\in \mathbb{R}^d,\varepsilon>0.
	\end{align}
\end{remark}
\begin{lemma} \label{lem-es-hol} Let $\{e_1,...,e_d\}$ be an orthonormal basis in $\mathbb{R}^d$. Let  $y_{0i} \in \tilde{H}_{e_i}$  $i=1,...,d$ and $\varepsilon\in (0,1)$. For any $x_i\in \tilde{H}_{e_i}$,  $i=1,...,d$,  we denote $\nu^1_{k,\sum_{i=d-k+1}^{d}x_i}, \nu^2_{k,\sum_{i=d-k+1}^{d}x_i}\in \mathcal{M}^+(\bigotimes_{i=1}^{d-k} \tilde{H}_{e_i})$ for $k=1,..,d$,  by
	\begin{align*}
	&d\nu^1_{k,\sum_{i=d-k+1}^{d}x_i}(y_{d-k},...,y_1)=d|Df^{e_{d-k}}_{\sum_{i=1}^{d-k-1}y_i+\sum_{i=d-k+1}^{d}x_i}|(y_{d-k})d\mathcal{H}^1(y_{d-k-1})...d\mathcal{H}^1(y_1),\\& d\nu^2_{k,\sum_{i=d-k+1}^{d}x_i}(y_{d-k},...,y_1)=\mathbf{1}_{|\sum_{i=1}^{d-k}y_{0i}-\sum_{i=1}^{d-k}y_i|\leq 2\varepsilon}d\nu^1_{k,\sum_{i=d-k+1}^{d}x_i}(y_{d-k},...,y_1).
	\end{align*}
	Then, for any $x_i\in \tilde{H}_{e_i}$  $i=1,...,d$
	\begin{align*}
	M:&=\int_{\tilde{H}_{e_1}}...\int_{\tilde{H}_{e_d}}~1\wedge\left(\frac{\rho}{|\sum_{i=1}^{d}(x_i-y_i)|}\right)^{d+2}\mathbf{1}_{|\sum_{i=1}^{d}(y_{0i}-y_i)|\leq \varepsilon} \\&~~~~~~~~~~\times \left|f(\sum_{i=1}^{d}y_i)-f(y_1+\sum_{i=2}^{d}x_i)\right|d\mathcal{H}^1(y_d)...d\mathcal{H}^1(y_1)\\&\quad\quad\lesssim
	\sum_{k=0}^{d-2} \frac{\rho^{d+\frac{5}{4}}}{\varepsilon}\mathbf{I}_{\frac{3}{4}}^{d-k}(\nu^1_{k,\sum_{i=d-k+1}^{d}x_i},\bigotimes_{i=1}^{d-k} \tilde{H}_{e_i})(\sum_{i=1}^{d-k}x_i)  \\&\quad\quad\quad+\sum_{k=0}^{d-2} \rho^{d+1}\mathbf{1}_{|\sum_{i=1}^{d}(y_{0i}-x_i)|\leq 2\varepsilon}\mathbf{M}^{d-k}(\nu^2_{k,\sum_{i=d-k+1}^{d}x_i},\bigotimes_{i=1}^{d-k} \tilde{H}_{e_i})(\sum_{i=1}^{d-k}x_i),
	\end{align*}
	and 
	\begin{align*}\nonumber
	&\int_{\tilde{H}_{e_1}}...\int_{\tilde{H}_{e_d}} 1\wedge\left(\frac{\rho}{|\sum_{i=1}^{d}(x_i-y_i)|}\right)^{d+1} \left|f(\sum_{i=1}^{d}y_i)-f(y_1+\sum_{i=2}^{d}x_i)\right|d\mathcal{H}^1(y_d)...d\mathcal{H}^1(y_1)\\&\quad\quad\quad\quad\quad\lesssim
	\sum_{k=0}^{d-2}\rho^{d+\frac{1}{4}}\mathbf{I}_{\frac{3}{4}}^{d-k}(\nu^1_{k,\sum_{i=d-k+1}^{d}x_i},\bigotimes_{i=1}^{d-k} \tilde{H}_{e_i})(\sum_{i=1}^{d-k}x_i).
	\end{align*}
\end{lemma}
We will prove Lemma  \eqref{lem-es-hol}  in Appendix. 
Now, we are ready to prove Theorem \ref{mainthm1}.
\begin{proof}[Proof of Theorem \ref{mainthm1}] \textbf{Step 1:} We prove that 
	\begin{align}\label{esweakL1b}
	\limsup_{\lambda\to \infty}\lambda\mathcal{L}^d\left(\left\{ \mathbf{T}_{\varepsilon}^{1,n}(\mu)>\lambda \right\}\cap B_R\right)\lesssim |\log(\varepsilon)||\mu|^{s}(\mathbb{R}^d),
	\end{align} for any $R>0$, $\varepsilon\in (0,1/10)$ and $n\in \mathbb{N}.$  
We now assume that \eqref{esweakL1b} is proven. 
Let  $\chi\in C_c^\infty(\mathbb{R}^d)$ be such that $\chi=1$ in $B_{R/4}$ and $\chi=0$ in $B_{R/2}^c$.
Thanks to \eqref{esweakL1b} and using the fact that $\mathbf{T}_{\varepsilon}^{1,n}(D(\chi f))\in L^\infty(B_{R}^c)$, one gets
	\begin{align*}
&\limsup_{\lambda\to \infty}\lambda\mathcal{L}^d\left(\left\{ \mathbf{T}_{\varepsilon}^{1,n}(\mu)>\lambda \right\}\right)\\&~~~~\leq  \limsup_{\lambda\to \infty}\lambda\mathcal{L}^d\left(\left\{ \mathbf{T}_{\varepsilon}^{1,n}(\mu)>\lambda \right\}\cap B_{R}\right)+\limsup_{\lambda\to \infty}\lambda\mathcal{L}^d\left(\left\{ \mathbf{T}_{\varepsilon}^{1,n}(\mu)>\lambda \right\}\cap B_{R}^c\right)\\&~~~~\lesssim|\log(\varepsilon)||\mu|^{s}(\mathbb{R}^d)+\limsup_{\lambda\to \infty}\lambda\mathcal{L}^d\left(\left\{ \mathbf{T}_{\varepsilon}^{1,n}(D((1-\chi) f))>\lambda/2 \right\}\right).
\end{align*}
So, by   \eqref{esL1inf'}, \eqref{esL1inf''}  and using the fact that $\mathbf{T}_{\varepsilon}(\mu)\leq \mathbf{T}_{\varepsilon}^{1,n}(\mu)+\mathbf{T}_{\varepsilon}^{2,n}(\mu)$, we have 
\begin{align*}
\limsup_{\lambda\to \infty}\lambda\mathcal{L}^d\left(\left\{ \mathbf{T}_{\varepsilon}(\mu)>\lambda \right\}\right)&\lesssim|\log(\varepsilon)||\mu|^{s}(\mathbb{R}^d)+C(\varepsilon)|D^{s}((1-\chi) f)|(\mathbb{R}^d)+C(\varepsilon)c_n|\mu|^{s}(\mathbb{R}^d)\\&\lesssim |\log(\varepsilon)||\mu|^{s}(\mathbb{R}^d)+C(\varepsilon)|\mu|^{s}(B_{R/4}^c)+C(\varepsilon)c_n|\mu|^{s}(\mathbb{R}^d).
\end{align*}
This implies \eqref{esweakL1'} by 
letting $R\to \infty$,  $n\to\infty$.
Moreover,  as proof of Lemma \ref{le2} we also get \eqref{es6}.\\
We are going to prove \eqref{esweakL1b} in several steps. \\
\textbf{Step 2.} Let  $\supp(\mu^{s})$ be the support of $\mu^s$.
	Let $\eta:\mathbb{R}^d\to S^{d-1}$ be such that $\eta(x)=\frac{d\mu^{s}(x)}{d|\mu |^{s}(x)}$ if $x\in \supp(\mu^{s})$ and  $\eta(x)=(1,...,0)\in S^{d-1}$ if $x\notin \supp(\mu^{s})$. Let $\eta^{\kappa}: \mathbb{R}^d\to S^{d-1}$ be smooth functions such that $\eta^{\kappa}\to \eta$ $|\mu|^s-$a.e in $\mathbb{R}^d$ and $
	\lim_{\kappa\to 0}\int_{\mathbb{R}^d}|\eta^{\kappa}-\eta|d|\mu|^{s}=0.$
	Let $\varphi_{r}\in C^{\infty}_b(\mathbb{R}^d)$ be such that $\varphi_{r}(z)=1$ if $|z|>2r$ and  $\varphi_{r}(z)=0$ if $|z|\leq r$ and $||\nabla \varphi_{r}||_{L^\infty(\mathbb{R}^d)}\leq Cr^{-1}$.  \\
Let us define $
S_{\tau}=\{ y\in 2\tau\mathbb{Z}^d: y\in B_{R+4\rho_0}\},$ for  $\tau\in (0,\rho_0/100)$.
 There exists   a sequence of smooth functions  $\{\chi_{y_{\tau}}^{\tau}\}_{y_{\tau}\in S_{\tau}}$ such that 
$
0\leq \chi_{y_{\tau}}^{\tau}(y)\leq 1,~~~\sum_{y_{\tau}\in S_{\tau}} \chi_{y_{\tau}}^{\tau}(y)=1~~\forall y\in B_{R+4\rho_0}$
and $\chi_{y_{\tau}}^{\tau}=1$ in $B_{\tau}(y_{\tau})$, $\supp(\chi_{y_{\tau}}^{\tau})\subset B_{2\tau}(y_{\tau})$, $|\nabla \chi_{y_{\tau}}^{\tau}(y)|\leq C \tau^{-1}~\forall~y\in \mathbb{R}^d$.\\
Note that $\text{Card}(S_{\tau})\sim \left(\frac{R+\rho_0}{\tau}\right)^d$,  $B_{\tau}(y_{\tau})\cap B_{\tau}(y'_{\tau})=\emptyset$ for $y_{\tau}, y'_{\tau}\in S_\tau$,  $y_{\tau}\not= y'_{\tau},$ and 
\begin{align}\label{es1}
\sum_{y_{\tau}\in S_{\tau}}\mathbf{1}_{B_{100\tau}(y_{\tau})}(y)\lesssim \mathbf{1}_{B_{R+6\rho_0}}(y)~~\forall~~y\in \mathbb{R}^d.
\end{align}
Set $\chi_0=\sum_{y_{\tau}\in S_{\tau}} \chi_{y_{\tau}}^{\tau}$.  For any $y_{\tau}\in S_{\tau}$, we denote $$ \eta^{\kappa}_{y_{\tau}}=\eta^{\kappa}(y_{\tau}).$$ Because of  $\mu^{s}=\eta\langle\eta,\mu^{s}\rangle$, one has 
$$
\mu=(1-\chi_0)\mu+ \chi_0\mu^a+\chi_0(\eta-\eta^{\kappa})  \langle\eta,\mu^{s}\rangle+\chi_0\eta^{\kappa} \langle(\eta-\eta^{\kappa}),\mu^{s}\rangle+\chi_0\eta^{\kappa} \langle\eta^{\kappa},\mu^{s}\rangle,$$
and \begin{align*}
&\chi_0\eta^{\kappa} \langle\eta^{\kappa},\mu^{s}\rangle=(\sum_{y_{\tau}\in S_{\tau}} \chi_{y_{\tau}}^{\tau})\eta^{\kappa} \langle\eta^{\kappa},\mu^{s}\rangle=\sum_{y_{\tau}\in S_{\tau}}\chi_{y_{\tau}}^{\tau}\eta^{\kappa} \langle(\eta^{\kappa}-\eta^{\kappa}_{y_{\tau}}),\mu^{s}\rangle \\&+\sum_{y_{\tau}\in S_{\tau}}\chi_{y_{\tau}}^{\tau} (\eta^{\kappa}-\eta^{\kappa}_{y_{\tau}})\langle\eta^{\kappa}_{y_{\tau}},\mu^{s}\rangle+\sum_{y_{\tau}\in S_{\tau}}\chi_{y_{\tau}}^{\tau} \eta^{\kappa}_{y_{\tau}}\langle\eta^{\kappa}_{y_{\tau}},\mu\rangle-\sum_{y_{\tau}\in S_{\tau}}\chi_{y_{\tau}}^{\tau} \eta^{\kappa}_{y_{\tau}}\langle\eta^{\kappa}_{y_{\tau}},\mu^a\rangle.
\end{align*}
Hence, with $\tilde{\mathbf{K}}^n=\frac{\varepsilon^{-d+1}}{\rho^\alpha}\left(\frac{1}{|.|^{d-\alpha}}\phi^{e,\varepsilon}_\rho(.)\right)\star\mathbf{K}_{n}$ and $\zeta\in (0,1/10)$   we write 
\begin{align*}
&\tilde{\mathbf{K}}^n\star\mu=\tilde{\mathbf{K}}^n\star((1-\chi_0)\mu) +\tilde{\mathbf{K}}^n\star(\chi_0\mu^a)+\tilde{\mathbf{K}}^n\star (\chi_0(\eta-\eta^{\kappa})  \langle\eta,\mu^{s}\rangle)+\tilde{\mathbf{K}}^n\star (\chi_0\eta^{\kappa} \langle(\eta-\eta^{\kappa}),\mu^{s}\rangle)\\&~+\sum_{y_{\tau}\in S_{\tau}} \tilde{\mathbf{K}}^n\star(\chi_{y_{\tau}}^{\tau}\eta^{\kappa} \langle(\eta^{\kappa}-\eta^{\kappa}_{y_{\tau}}),\mu^{s}\rangle)+\sum_{y_{\tau}\in S_{\tau}}\tilde{\mathbf{K}}^n\star(\chi_{y_{\tau}}^{\tau} (\eta^{\kappa}-\eta^{\kappa}_{y_{\tau}})\langle\eta^{\kappa}_{y_{\tau}},\mu^{s}\rangle)\\&~-\sum_{y_{\tau}\in S_{\tau}}\tilde{\mathbf{K}}^n\star(\chi_{y_{\tau}}^{\tau} \eta^{\kappa}_{y_{\tau}}\langle\eta^{\kappa}_{y_{\tau}},\mu^a\rangle)+\sum_{y_{\tau}\in S_{\tau}}\frac{\varepsilon^{-d+1}}{\rho^\alpha}\left(\frac{(1-\varphi_{\zeta\rho})}{|.|^{d-\alpha}}\phi^{e,\varepsilon}_\rho(.)\right)\star\mathbf{K}_{n}\star(\chi_{y_{\tau}}^{\tau} \eta^{\kappa}_{y_{\tau}}\langle\eta^{\kappa}_{y_{\tau}},\mu\rangle)\\&~+\sum_{y_{\tau}\in S_{\tau}}\frac{\varepsilon^{-d+1}}{\rho^\alpha}\left(\frac{\varphi_{\zeta\rho}}{|.|^{d-\alpha}}\phi^{e,\varepsilon}_\rho(.)\right)\star\mathbf{K}_{n}\star(\chi_{y_{\tau}}^{\tau} \eta^{\kappa}_{y_{\tau}}\langle\eta^{\kappa}_{y_{\tau}},\mu\rangle):=\sum_{i=1}^9I_{i,\varepsilon}^{e,\rho}.
\end{align*}
\textbf{Step 3:} In this prove, we denote $$
A_{i}(\lambda,\varepsilon)=\lambda\mathcal{L}^d\left(\left\{\sup_{\rho\in (0,\rho_0),e\in S^{d-1}}|I_{i,\varepsilon}^{e,\rho}|>\lambda \right\}\cap B_R\right).$$
Thus, for $\lambda>1$, 
\begin{align}
\lambda\mathcal{L}^d\left(\left\{ \mathbf{T}_{\varepsilon}^{1,n}(\mu)>\lambda \right\}\cap B_R\right)\leq 9\sum_{i=1}^{9}A_{i}(\lambda/9,\varepsilon)\label{es1-8}.
\end{align}
Thanks to \eqref{esL1inf'}  we have
\begin{align}&\label{es2}
\limsup_{\lambda\to \infty} \sum_{i=2,7}A_i(\lambda,\varepsilon)=0,\\\label{es2b}
&\limsup_{\lambda\to \infty} \sum_{i=3,4}A_i(\lambda,\varepsilon)\leq C(\varepsilon)  |||\eta-\eta^{\kappa}||\mu^{s}||||_{\mathcal{M}(\mathbb{R}^d)},\\&\label{es2c}
\limsup_{\lambda\to \infty} \sum_{i=5,6}A_i(\lambda,\varepsilon)\leq C(\varepsilon)  ||\sum_{y_{\tau}\in S_{\tau}}\chi_{y_{\tau}}^{\tau}|\eta^{\kappa}-\eta^{\kappa}_{y_{\tau}}||\mu^s|||_{\mathcal{M}(\mathbb{R}^d)}\leq C(\varepsilon,\kappa) \tau |\mu|^{s}(\mathbb{R}^d).
\end{align} 
Here in the last inequality we have used the fact that
\begin{align*}
\sum_{y_{\tau}\in S_{\tau}}\chi_{y_{\tau}}^{\tau}(x)|\eta^{\kappa}(x)-\eta^{\kappa}_{y_{\tau}}|&\lesssim||\nabla \eta^\kappa||_{L^\infty(\mathbb{R}^d)}\sum_{y_{\tau}\in S_{\tau}}\mathbf{1}_{B_{2\tau}(y_{\tau})}(x)|x-y_\tau|
\\&\overset{\eqref{es1}}\lesssim||\nabla \eta^\kappa||_{L^\infty(\mathbb{R}^d)}\tau \mathbf{1}_{B_{R+6\rho_0}}(x)~~\forall x\in \mathbb{R}^d.
\end{align*} 
Again, applying  \eqref{esL1inf'} (where $\rho$ is replaced by $\zeta\rho$) yields
\begin{align}\label{es2d}
\limsup_{\lambda\to\infty}A_{8}(\lambda,\varepsilon)\leq C(\varepsilon) \zeta^{\alpha}||\sum_{y_{\tau}\in S_{\tau}}\chi_{y_{\tau}}^{\tau}|\mu^s||||_{\mathcal{M}(\mathbb{R}^d)}\leq C(\varepsilon) \zeta^{\alpha} |\mu|^s(\mathbb{R}^d).
\end{align}
On the other hand,  it is easy to see that $\sup_{\rho\in (0,\rho_0),e\in S^{d-1}}|I_{1,\varepsilon}^{e,\rho}(.)|\in L^\infty(B_{R}),$ so
\begin{align}\label{es2e}
\limsup_{\lambda\to \infty} A_1(\lambda,\varepsilon)=0.
\end{align}
Therefore, we deduce from \eqref{es1-8} and \eqref{es2}-\eqref{es2e} that 
\begin{align}\nonumber
&\limsup_{\lambda\to \infty}\lambda\mathcal{L}^d\left(\left\{ \mathbf{T}_{\varepsilon}^{1,n}(\mu)>\lambda \right\}\cap B_R\right) \lesssim C(n,\varepsilon)  |||\eta-\eta^{\kappa}||\mu|^{s}||_{\mathcal{M}(\mathbb{R}^d)}\\&~~~~~~+C(n,\varepsilon,\kappa) \tau |\mu|^{s}(\mathbb{R}^d)+C(n,\varepsilon) \zeta^{\alpha} |\mu|^{s}(\mathbb{R}^d)+ \limsup_{\lambda\to \infty} A_{9}(\lambda,\varepsilon)\label{ES1-8}.
\end{align}
In next steps, we will deal with $A_9(\lambda,\varepsilon)$. \medskip\\
\textbf{Step 4:} One has 
\begin{align*}
I_{9,\varepsilon}^{e,\rho}(x)&= \sum_{y_{\tau}\in S_{\tau}}\frac{\varepsilon^{-d+1}}{\rho^\alpha}\int_{\mathbb{R}^d}\left[\int_{\mathbb{R}^d}\mathbf{K}_{n}(z)\varphi_{\zeta\rho}(x-y-z)\frac{\langle\phi^{e,\varepsilon}_\rho(x-y-z),\eta^{\kappa}_{y_{\tau}}\rangle}{|x-y-z|^{d-\alpha}} dz\right] \chi_{y_{\tau}}^{\tau}(y)d\langle\eta^{\kappa}_{y_{\tau}},\mu(y)\rangle\\&=\sum_{y_{\tau}\in S_{\tau}}\int_{\mathbb{R}^d}\mathbf{K}^{\varepsilon,n}_{e,\rho}(x-y) \chi_{y_{\tau}}^{\tau}(y)d\langle\eta^{\kappa}_{y_{\tau}},\mu(y)\rangle+\mathbf{c}(\varepsilon,\kappa,\tau,\zeta)\sum_{y_{\tau}\in S_{\tau}}(\varphi_{\rho}\mathbf{K}_{n})\star (\chi_{y_{\tau}}^{\tau} \langle\eta^{\kappa}_{y_{\tau}},\mu\rangle)(x)\\&=:I_{10,\varepsilon}^{e,\rho}(x)+I_{11,\varepsilon}^{e,\rho}(x),
\end{align*}
where 
\begin{align}\label{ker12}
\mathbf{K}^{\varepsilon,n}_{e,\rho}(z')=\frac{\varepsilon^{-d+1}}{\rho^\alpha}\int_{\mathbb{R}^d}\mathbf{K}_{n}(z)\varphi_{\zeta\rho}(z'-z)\frac{\langle\phi^{e,\varepsilon}_\rho(z'-z),\eta^{\kappa}_{y_{\tau}}\rangle}{|z'-z|^{d-\alpha}}dz-\mathbf{c}(\varepsilon,\kappa,\tau,\zeta)\varphi_{\rho}(z')\mathbf{K}_{n}(z')~~\forall~~z'\in \mathbb{R}^d,
\end{align}
and
\begin{align}
\mathbf{c}(\varepsilon,\kappa,\tau,\zeta)=\frac{\varepsilon^{-d+1}}{\rho^\alpha}\int_{\mathbb{R}^d}\varphi_{\zeta\rho}(z'-z)\frac{\langle\phi^{e,\varepsilon}((z'-z)/\rho),\eta^{\kappa}_{y_{\tau}}\rangle}{|z'-z|^{d-\alpha}}dz=\varepsilon^{-d+1}\int_{\mathbb{R}^d}\varphi_{\zeta}(z)\frac{\langle\phi^{e,\varepsilon}(z),\eta^{\kappa}_{y_{\tau}}\rangle}{|z|^{d-\alpha}}dz.\label{conker12}
\end{align}
Note that $|\mathbf{c}(\varepsilon,\kappa,\tau,\zeta)|\lesssim 1$
for all $\kappa,\varepsilon,\zeta>0, e\in S^{d-1}$ and by \eqref{es8} in the proof of Proposition \ref{singu-operator}, we have for any $x\in \mathbb{R}^d\backslash\{0\}$,  $$
|\mathbf{K}^{\varepsilon,n}_{e,\rho}(x)|\leq C(n,\varepsilon,\zeta)\frac{1}{|x|^{d-\alpha}}\min\left\{\frac{1}{\rho^\alpha},\frac{\rho}{|x|^{1+\alpha}}\right\}.$$
Similarly, we also have for any $x\in \mathbb{R}^d\backslash\{0\}$, $$
|\nabla \mathbf{K}^{\varepsilon,n}_{e,\rho}(x)|\leq C(n,\varepsilon,\zeta)\frac{1}{|x|^{d-\alpha+1}}\min\left\{\frac{1}{\rho^\alpha},\frac{\rho}{|x|^{1+\alpha}}\right\}.$$ 
Moreover, since $|\varphi_{\zeta\rho}(z)|\leq C1_{|z|>\zeta\rho}$, so  we have for any $|x|\leq \zeta\rho/4$ that $
|\mathbf{K}^{\varepsilon,n}_{e,\rho}(x)|+\rho|\nabla\mathbf{K}^{\varepsilon,n}_{e,\rho}(x)|\leq C(n,\varepsilon,\zeta)\frac{1}{\rho^d}.$ 
Thus, 
\begin{align}\label{es29}
|\mathbf{K}^{\varepsilon,n}_{e,\rho}(x)|\leq C(n,\varepsilon,\zeta)\min\left\{\frac{1}{\rho^d},\frac{\rho}{|x|^{d+1}}\right\},~~ |\nabla \mathbf{K}^{\varepsilon,n}_{e,\rho}(x)|\leq C(n,\varepsilon,\zeta)\min\left\{\frac{1}{\rho^{d+1}},\frac{\rho}{|x|^{d+2}}\right\}.
\end{align}
Thanks to Proposition \ref{pro-singularintegral-rough1}, we get 
\begin{align}\label{esA11'}
\limsup_{\lambda\to\infty}A_{11}(\lambda,\varepsilon)\lesssim||
\sum_{y_{\tau}\in S_{\tau}} \chi_{y_{\tau}}^{\tau}|\mu|^{s}||_{\mathcal{M}(\mathbb{R}^d)}\lesssim |\mu|^{s}(\mathbb{R}^d).
\end{align}
Using integration by parts, we have 
\begin{align*}
&\int_{\mathbb{R}^d}\mathbf{K}^{\varepsilon,n}_{e,\rho}(x-y) \chi_{y_{\tau}}^{\tau}(y)d\langle\eta^{\kappa}_{y_{\tau}},\mu(y)\rangle=-\int_{\mathbb{R}^d}\eta^{\kappa}_{y_{\tau}}.\nabla_y\left[\mathbf{K}^{\varepsilon,n}_{e,\rho}(x-y) \chi_{y_{\tau}}^{\tau}(y)\right]f(y)dy,\\&
\int_{H_{\eta^{\kappa}_{y_{\tau}}}}\int_{\tilde{H}_{\eta^{\kappa}_{y_{\tau}}}}\mathbf{K}^{\varepsilon,n}_{e,\rho}(x-y_1-y_2) \chi_{y_{\tau}}^{\tau}(y_1+y_2)dDf^{\eta^{\kappa}_{y_{\tau}}}_{\tilde{x}_{\kappa,y_{\tau}}}(y_1)d\mathcal{H}^{d-1}(y_2)\\&~~~~~~~~~~~=-\int_{\mathbb{R}^d}\eta^{\kappa}_{y_{\tau}}.\nabla_y\left[\mathbf{K}^{\varepsilon,n}_{e,\rho}(x-y) \chi_{y_{\tau}}^{\tau}(y)\right]f(\langle y,\eta^{\kappa}_{y_{\tau}}\rangle \eta^{\kappa}_{y_{\tau}}+\tilde{x}_{\kappa,y_{\tau}})dy.
\end{align*}
So, 
 \begin{align*}
I_{10,\varepsilon}^{e,\rho}(x)&= -\sum_{y_{\tau}\in S_{\tau}}\int_{\mathbb{R}^d}\eta^{\kappa}_{y_{\tau}}.\nabla_y\left[\mathbf{K}^{\varepsilon,n}_{e,\rho}(x-y) \chi_{y_{\tau}}^{\tau}(y)\right]\left[f(y)-f(\langle y,\eta^{\kappa}_{y_{\tau}}\rangle \eta^{\kappa}_{y_{\tau}}+\tilde{x}_{\kappa,y_{\tau}})\right]dy\\&+\sum_{y_{\tau}\in S_{\tau}}\int_{H_{\eta^{\kappa}_{y_{\tau}}}}\int_{\tilde{H}_{\eta^{\kappa}_{y_{\tau}}}}\mathbf{K}^{\varepsilon,n}_{e,\rho}(x-y_1-y_2) \chi_{y_{\tau}}^{\tau}(y_1+y_2)dD^{s}f^{\eta^{\kappa}_{y_{\tau}}}_{\tilde{x}_{\kappa,y_{\tau}}}(y_1)d\mathcal{H}^{d-1}(y_2)\\&+\sum_{y_{\tau}\in S_{\tau}}\int_{H_{\eta^{\kappa}_{y_{\tau}}}}\int_{\tilde{H}_{\eta^{\kappa}_{y_{\tau}}}}\mathbf{K}^{\varepsilon,n}_{e,\rho}(x-y_1-y_2) \chi_{y_{\tau}}^{\tau}(y_1+y_2)\langle\eta^{\kappa}_{y_{\tau}},D^af(y_1+\tilde{x}_{\kappa,y_{\tau}})\rangle d\mathcal{H}^1(y_1)d\mathcal{H}^{d-1}(y_2)\\&=:\sum_{i=12}^{14}I_{i,\varepsilon}^{e,\rho}(x),
 \end{align*}
 where throughout this proof we denote  $$\tilde{x}_{\kappa,y_{\tau}}=x-\langle x,\eta^{\kappa}_{y_{\tau}}\rangle \eta^{\kappa}_{y_{\tau}}.$$
Thus, 
\begin{align}\label{esA6}
A_9(\lambda,\varepsilon)\leq 4\sum_{i=12}^{14}A_i(\lambda/4,\varepsilon).
\end{align}
\textbf{Step 5:}. To treat $A_{13}(\lambda,\varepsilon)$ and $A_{14}(\lambda,\varepsilon)$, we need to show the following inequality:
\begin{align}\nonumber
&\left|\int_{H_{\eta^{\kappa}_{y_{\tau}}}}\int_{\tilde{H}_{\eta^{\kappa}_{y_{\tau}}}}\mathbf{K}^{\varepsilon,n}_{e,\rho}(x-y_1-y_2) \chi_{y_{\tau}}^{\tau}(y_1+y_2)d\nu(y_1)d\mathcal{H}^{d-1}(y_2)\right|\\&\nonumber~~~~\lesssim|\log(\varepsilon)|\mathbf{1}_{B_{4\tau}(y_\tau)}(x)\mathbf{M}^1(1_{B_{2\tau}(y_{\tau,1})}\nu,\tilde{H}_{\eta^{\kappa}_{y_{\tau}}})(x_1)\\&~~~~~+C(n,\varepsilon,\zeta,\tau)\rho\mathbf{1}_{B_{4\tau}(y_\tau)^c}(x)\int_{H_{\eta^{\kappa}_{y_{\tau}}}}\int_{\tilde{H}_{\eta^{\kappa}_{y_{\tau}}}}\mathbf{1}_{B_{2\tau}(y_\tau)}(y_1+y_2) d\nu(y_1)d\mathcal{H}^{d-1}(y_2),\label{es-maxifun-1}
\end{align}
for any $\nu\in \mathcal{M}_b(\tilde{H}_{\eta^{\kappa}_{y_{\tau}}})$ and $x\in \mathbb{R}^d$
where
$$x_1=\langle x,\eta^{\kappa}_{y_{\tau}}\rangle \eta^{\kappa}_{y_{\tau}},~~~y_{\tau,1}=\langle y_{\tau},\eta^{\kappa}_{y_{\tau}}\rangle \eta^{\kappa}_{y_{\tau}}.$$
Indeed, set $\tau_{x_1}(z)=x_1-z$ for any $z\in \tilde{H}_{\eta^{\kappa}_{y_{\tau}}}$.\vspace{.2cm}\\
 By Lemma \ref{le-simineq} (with $a=|\mathbf{K}_{\varepsilon}^{e,\rho}(.)|,\omega=(\tau_{x_1})_{\#}(\mathbf{1}_{B_{2\tau}(y_{\tau,1},\tilde{H}_{\eta^{\kappa}_{y_{\tau}}})}|\nu|)$), we have 
\begin{align*}
&\sum_{j=-\infty}^{\infty}\int_{H_{\eta^{\kappa}_{y_{\tau}}}}\int_{\tilde{H}_{\eta^{\kappa}_{y_{\tau}}}}\mathbf{1}_{2^{j}\rho<|x-y_1-y_2|\leq 2^{j+1}\rho}|\mathbf{K}^{\varepsilon,n}_{e,\rho}(x-y_1-y_2)| \chi_{y_{\tau}}^{\tau}(y_1+y_2)d|\nu|(y_1)d\mathcal{H}^{d-1}(y_2)\\&\lesssim \sum_{j=-\infty}^{\infty}
\int_{H_{\eta^{\kappa}_{y_{\tau}}}}\int_{\tilde{H}_{\eta^{\kappa}_{y_{\tau}}}}\mathbf{1}_{2^{j}\rho<|y_1+y_2|\leq 2^{j+1}\rho}|\mathbf{K}^{\varepsilon,n}_{e,\rho}(y_1+y_2)| d(\tau_{x_1})_{\#}(\mathbf{1}_{B_{2\tau}(y_{\tau,1},\tilde{H}_{\eta^{\kappa}_{y_{\tau}}})}|\nu|)(y_1)d\mathcal{H}^{d-1}(y_2)\\&
\lesssim \sum_{j=-\infty}^{\infty}\left[(2^{j}\rho)^d\int_{S^{d-1}}\sup_{r\in [2^{j}\rho,2^{j+1}\rho]}|\mathbf{K}^{\varepsilon,n}_{e,\rho}(r\theta)|d\mathcal{H}^{d-1}(\theta)\right] \mathbf{M}^1((\tau_{x_1})_{\#}(\mathbf{1}_{B_{2\tau}(y_{\tau,1},\tilde{H}_{\eta^{\kappa}_{y_{\tau}}})}|\nu|),\tilde{H}_{\eta^{\kappa}_{y_{\tau}}})(0)
\\&= \left[\sum_{j=-\infty}^{\infty}(2^{j}\rho)^d\int_{S^{d-1}}\sup_{r\in [2^{j}\rho,2^{j+1}\rho]}|\mathbf{K}^{\varepsilon,n}_{e,\rho}(r\theta)|d\mathcal{H}^{d-1}(\theta)\right] \mathbf{M}^1(\mathbf{1}_{B_{2\tau}(y_{\tau,1},\tilde{H}_{\eta^{\kappa}_{y_{\tau}}})}|\nu|,\tilde{H}_{\eta^{\kappa}_{y_{\tau}}})(x_1).
\end{align*}
So, by \eqref{es20} in Lemma \ref{le-es-ker-log} below, we obtain that \begin{align}
\left|\int_{H_{\eta^{\kappa}_{y_{\tau}}}}\int_{\tilde{H}_{\eta^{\kappa}_{y_{\tau}}}}\mathbf{K}^{\varepsilon,n}_{e,\rho}(x-y_1-y_2) \chi_{y_{\tau}}^{\tau}(y_1+y_2)d\nu(y_1)d\mathcal{H}^{d-1}(y_2)\right|\label{es-maxifun-1'}\lesssim|\log(\varepsilon)|\mathbf{M}^1(1_{B_{2\tau}(y_{\tau,1})}\nu,\tilde{H}_{\eta^{\kappa}_{y_{\tau}}})(x_1).
\end{align}
On the other hand,  for any  $x\notin B_{4\tau}(y_\tau)$, 
\begin{align*}
&\left|\int_{H_{\eta^{\kappa}_{y_{\tau}}}}\int_{\tilde{H}_{\eta^{\kappa}_{y_{\tau}}}}\mathbf{K}^{\varepsilon,n}_{e,\rho}(x-y_1-y_2) \chi_{y_{\tau}}^{\tau}(y_1+y_2)d\nu(y_1)d\mathcal{H}^{d-1}(y_2)\right|\\&~~\lesssim \int_{H_{\eta^{\kappa}_{y_{\tau}}}}\int_{\tilde{H}_{\eta^{\kappa}_{y_{\tau}}}}\mathbf{1}_{|x-y_1-y_2|>2\tau}\mathbf{1}_{|y_1+y_2-y_\tau|<2\tau} |\mathbf{K}^{\varepsilon,n}_{e,\rho}(x-y_1-y_2)|d|\nu|(y_1)d\mathcal{H}^{d-1}(y_2)\\&\overset{\eqref{es29}}\leq C(n,\varepsilon,\zeta,\tau)\rho\int_{H_{\eta^{\kappa}_{y_{\tau}}}}\int_{\tilde{H}_{\eta^{\kappa}_{y_{\tau}}}}\mathbf{1}_{|y_1+y_2-y_\tau|<2\tau} d|\nu|(y_1)d\mathcal{H}^{d-1}(y_2).
\end{align*}
From this and \eqref{es-maxifun-1'}, we find \eqref{es-maxifun-1}.\\\medskip\\
\textbf{Step 6:} Estimate $A_{13}(\lambda,\varepsilon)$ and $A_{14}(\lambda,\varepsilon)$.\\ We set $$\omega^\tau_{y_\tau,z_2}:=\mathbf{1}_{B_{2\tau}(y_{\tau,1},\tilde{H}_{\eta^{\kappa}_{y_{\tau}}})}|D^{s}f^{\eta^{\kappa}_{y_{\tau}}}_{z_2}|~~\forall~z_2\in H_{\eta^{\kappa}_{y_{\tau}}}.$$ We then apply \eqref{es-maxifun-1} for $\nu(y_1)= D^{s}f^{\eta^{\kappa}_{y_{\tau}}}_{\tilde{x}_{\kappa,y_{\tau}}}(y_1)$ to get that
\begin{align*}
I_{13,\varepsilon}^{e,\rho}(x)&\lesssim |\log(\varepsilon)|\sum_{y_{\tau}\in S_{\tau}} \mathbf{1}_{B_{4\tau}(y_\tau)}(x) \mathbf{M}^1(\omega^\tau_{y_\tau,\tilde{x}_{\kappa,y_{\tau}}},\tilde{H}_{\eta^{\kappa}_{y_{\tau}}})(\langle x,\eta^{\kappa}_{y_{\tau}}\rangle \eta^{\kappa}_{y_{\tau}})\\&~~~+ \sum_{y_{\tau}\in S_{\tau}} C(\varepsilon,\zeta,\tau)\rho\int_{H_{\eta^{\kappa}_{y_{\tau}}}}\int_{\tilde{H}_{\eta^{\kappa}_{y_{\tau}}}}\mathbf{1}_{B_{2\tau}(y_\tau)}(y_1+y_2) d |D^{s}f^{\eta^{\kappa}_{y_{\tau}}}_{\tilde{x}_{\kappa,y_{\tau}}}|(y_1)d\mathcal{H}^{d-1}(y_2).
\end{align*}
By \eqref{equa-eta} in Proposition \ref{ML1-BV} and \eqref{es1}, we have 
\begin{align*}
I_{13,\varepsilon}^{e,\rho}(x)&\lesssim |\log(\varepsilon)|\sum_{y_{\tau}\in S_{\tau}} \mathbf{1}_{B_{4\tau}(y_\tau)}(x) \mathbf{M}^1(\omega^\tau_{y_\tau,\tilde{x}_{\kappa,y_{\tau}}},\tilde{H}_{\eta^{\kappa}_{y_{\tau}}})(\langle x,\eta^{\kappa}_{y_{\tau}}\rangle \eta^{\kappa}_{y_{\tau}})\\&~~~+\sum_{y_{\tau}\in S_{\tau}} C(\varepsilon,\zeta,\tau)\rho\int_{\mathbb{R}^d}\mathbf{1}_{B_{2\tau}(y_\tau)}(y) d|\mu|^{s}(y)\\
&\lesssim |\log(\varepsilon)|\sum_{y_{\tau}\in S_{\tau}} \mathbf{1}_{B_{4\tau}(y_\tau)}(x) \mathbf{M}^1(\omega^\tau_{y_\tau,\tilde{x}_{\kappa,y_{\tau}}},\tilde{H}_{\eta^{\kappa}_{y_{\tau}}})(\langle x,\eta^{\kappa}_{y_{\tau}}\rangle \eta^{\kappa}_{y_{\tau}})+C(\varepsilon,\zeta,\tau)\rho|\mu|^{s}(\mathbb{R}^d).
\end{align*}
Thus,  for $\lambda>>1$
\begin{align}\nonumber
&A_{13}(\lambda,\varepsilon)\leq\lambda\mathcal{L}^d\left(\left\{x\in B_R: |\log(\varepsilon)|\sum_{y_{\tau}\in S_{\tau}} \mathbf{1}_{B_{4\tau}(y_\tau)}(x) \mathbf{M}^1(\omega^\tau_{y_\tau,\tilde{x}_{\kappa,y_{\tau}}},\tilde{H}_{\eta^{\kappa}_{y_{\tau}}})(\langle x,\eta^{\kappa}_{y_{\tau}}\rangle \eta^{\kappa}_{y_{\tau}})\gtrsim\lambda \right\}\right)\\ &\nonumber\leq \sum_{y'_{\tau}\in S_{\tau}} \lambda\mathcal{L}^d\left(\left\{x\in B_{2\tau}(y'_{\tau}): |\log(\varepsilon)|\sum_{y_{\tau}\in S_{\tau}} \mathbf{1}_{B_{4\tau}(y_\tau)}(x) \mathbf{M}^1(\omega^\tau_{y_\tau,\tilde{x}_{\kappa,y_{\tau}}},\tilde{H}_{\eta^{\kappa}_{y_{\tau}}})(\langle x,\eta^{\kappa}_{y_{\tau}}\rangle \eta^{\kappa}_{y_{\tau}})\gtrsim\lambda \right\}\right) \\&\nonumber\overset{\eqref{es1}}\lesssim \sum_{y_{\tau}\in S_{\tau}} \lambda\mathcal{L}^d\left(\left\{x\in B_{8\tau}(y_{\tau}): |\log(\varepsilon)| \mathbf{M}^1(\omega^\tau_{y_\tau,\tilde{x}_{\kappa,y_{\tau}}},\tilde{H}_{\eta^{\kappa}_{y_{\tau}}})(\langle x,\eta^{\kappa}_{y_{\tau}}\rangle \eta^{\kappa}_{y_{\tau}})\gtrsim\lambda \right\}\right).
\end{align}
Thanks to  the boundedness of $\mathbf{M}^1(.,\tilde{H}_{\eta^{\kappa}_{y_{\tau}}})$ from $\mathcal{M}(\tilde{H}_{\eta^{\kappa}_{y_{\tau}}})$ to $L^{1,\infty}(\tilde{H}_{\eta^{\kappa}_{y_{\tau}}})$ yields
for $\lambda>>1$
\begin{align*}
&A_{13}(\lambda,\varepsilon) \lesssim     \sum_{y_{\tau}\in S_{\tau}}\lambda \int_{H_{\eta^{\kappa}_{y_{\tau}}}} 1_{|z_2-(y_{\tau}-y_{\tau,1})|\leq 8\tau}\mathcal{H}^1\left(\left\{|\log(\varepsilon)|\mathbf{M}^1\left(\omega^\tau_{y_\tau,z_2},\tilde{H}_{\eta^{\kappa}_{y_{\tau}}}\right)\gtrsim\lambda \right\}\right) d\mathcal{H}^{d-1}(z_2)\\&\lesssim |\log(\varepsilon)|\sum_{y_{\tau}\in S_{\tau}} \int_{H_{\eta^{\kappa}_{y_{\tau}}}}1_{|z_2-(y_{\tau}-y_{\tau,1})|\leq 8\tau}\int_{\tilde{H}_{\eta^{\kappa}_{y_{\tau}}}} \mathbf{1}_{B_{2\tau}(y_{1,\tau},\tilde{H}_{\eta^{\kappa}_{y_{\tau}}})}(z_1) d|D^{s} f^{\eta^{\kappa}_{y_{\tau}}}_{z_2}|(z_1)d\mathcal{H}^{d-1}(z_2)
 \\&\lesssim 
 \nonumber |\log(\varepsilon)|\sum_{y_{\tau}\in S_{\tau}} \int_{B_{10\tau}(y_{\tau})} d|D^{s} f|(z)
  \\&\lesssim|\log(\varepsilon)| |\mu|^{s}(\mathbb{R}^d).
\end{align*}
Here  we have used \eqref{equa-eta} in Proposition \ref{ML1-BV} for the third inequality and \eqref{es1} for the last one.\\
Thus, 
\begin{align}\label{es14+}
\limsup_{\lambda\to\infty} A_{13}(\lambda,\varepsilon)\lesssim|\log(\varepsilon)| |\mu|^{s}(\mathbb{R}^d).
\end{align}
Similarly, we also have 
$$A_{14}(\lambda,\varepsilon)\lesssim |\log(\varepsilon)|||\mu^a||_{L^1(B_{R+6\rho_0})} ~~\forall~~\lambda>>1.$$
Since,
$
\limsup_{\lambda\to\infty}\lambda\mathcal{H}^{1}\left(\left\{ \mathbf{M}^1(1_{B_{2\tau}(y_{1,\tau},\tilde{H}_{\eta^{\kappa}_{y_{\tau}}})}|D^af(.+z_2)|,\tilde{H}_{\eta^{\kappa}_{y_{\tau}}})>\lambda \right\}\right)=0$
for $\mathcal{H}^{d-1}$-a.e $z_2$  in $ H_{\eta^{\kappa}_{y_{\tau}}}$, so  by dominated convergence theorem we get
\begin{align}\label{es28}
\limsup_{\lambda\to \infty}A_{14}(\lambda,\varepsilon)=0.
\end{align}
\textbf{Step 7:}. We will prove that 
\begin{align}\label{esA13}
\limsup_{\lambda\to \infty} A_{12}(\lambda,\varepsilon)\leq  C(n,\varepsilon,\zeta)||(\eta-\eta^\kappa)|\mu|^{s}||_{\mathcal{M}(\mathbb{R}^d)}+C(n,\varepsilon,\zeta,\kappa)\tau |\mu|^{s}(\mathbb{R}^d).
\end{align}
Let $\{\eta_1^{\kappa}(y_{\tau}), \eta_2^{\kappa}(y_{\tau}),...,\eta_d^{\kappa}(y_{\tau})\}$ be an orthonormal basis in $\mathbb{R}^d$ such that $\eta_1^{\kappa}(y_{\tau})=\eta^{\kappa}_{y_{\tau}}$. So,  for any $x\in\mathbb{R}^d$, throughout this proof we  denote $$
x_{\eta_i^{\kappa}(y_{\tau})}=\langle x, \eta_i^{\kappa}(y_{\tau})\rangle \eta_i^{\kappa}(y_{\tau}),~~x_{\eta_i^{\kappa}(y_{\tau})}^{1,j}=\sum_{i=1}^{j}x_{\eta_i^{\kappa}(y_{\tau})},~~ x_{\eta_i^{\kappa}(y_{\tau})}^{2,j}=\sum_{i=j+1}^{d}x_{\eta_i^{\kappa}(y_{\tau})}.$$
By \eqref{es29}, we have 
\begin{align*}
|I_{12,\varepsilon}^{e,\rho}(x)|&\leq C(n,\varepsilon,\zeta)\frac{1}{\rho^{d+1}}\sum_{y_{\tau}\in S_{\tau}}\int_{\mathbb{R}^d}\left(1\wedge\left(\frac{\rho}{|x-y|}\right)^{d+2}\right) \mathbf{1}_{|y_{\tau}-y|\leq 2\tau} \left|f(y)-f(y_{\eta^{\kappa}_{y_{\tau}}}+\sum_{i=2}^{d}x_{\eta_i^{\kappa}(y_{\tau})})\right|dy \\&+C(n,\varepsilon,\tau,\zeta)\frac{1}{\rho^{d}}\sum_{y_{\tau}\in S_{\tau}}\int_{\mathbb{R}^d}\left(1\wedge\left(\frac{\rho}{|x-y|}\right)^{d+1}\right)\left|f(y)-f(y_{\eta^{\kappa}_{y_{\tau}}}+\sum_{i=2}^{d}x_{\eta_i^{\kappa}(y_{\tau})})\right|dy.
\end{align*}
Applying Lemma \ref{lem-es-hol}  to $\{e_1,...,e_d\}=\{\eta_1^{\kappa}(y_{\tau}), \eta_2^{\kappa}(y_{\tau}),...,\eta_d^{\kappa}(y_{\tau})\}$ and $x_i=x_{\eta_i^{\kappa}(y_{\tau})}$ for $i=1,...,d$ and $\varepsilon=2\tau$, we find that  
\begin{align*}
I_{12,\varepsilon}^{e,\rho}(x)&\leq C(n,\varepsilon,\tau,\zeta)\rho^{\frac{1}{4}}\sum_{k=0}^{d-2} \sum_{y_{\tau}\in S_{\tau}}  \mathbf{I}_{\frac{3}{4}}^{d-k}(\nu^1_{k,x_{\eta_i^{\kappa}(y_{\tau})}^{2,d-k}},\bigotimes_{i=1}^{d-k} \tilde{H}_{\eta_i^{\kappa}(y_{\tau})})(x_{\eta_i^{\kappa}(y_{\tau})}^{1,d-k}) \nonumber \\&+C(n,\varepsilon,\zeta)\sum_{k=0}^{d-2} \sum_{y_{\tau}\in S_{\tau}} 1_{B_{4\tau}(y_\tau) }(x)\mathbf{M}^{d-k}(\nu^2_{k,x_{\eta_i^{\kappa}(y_{\tau})}^{2,d-k}},\bigotimes_{i=1}^{d-k} \tilde{H}_{\eta_i^{\kappa}(y_{\tau})})(x_{\eta_i^{\kappa}(y_{\tau})}^{1,d-k}),
\end{align*}
where 
\begin{align*}
&d\nu^1_{k,z}(y_{d-k},...,y_1)=d|Df^{\eta_{d-k}^{\kappa}(y_{\tau})}_{\sum_{i=1}^{d-k-1}y_i+z}|(y_{d-k})d\mathcal{H}^1(y_{d-k-1})...d\mathcal{H}^1(y_1),\\& d\nu^2_{k,z}(y_{d-k},...,y_1)=1_{|\sum_{i=1}^{d-k}(y_{\tau})_{\eta_i^{\kappa}(y_{\tau})}-\sum_{i=1}^{d-k}y_i|\leq 4\tau}d\nu^1_{k,z}(y_{d-k},...,y_1),
\end{align*} 
for any $z\in \bigotimes_{i=d-k+1}^{d} \tilde{H}_{\eta_i^{\kappa}(y_{\tau})}.$
Hence, 
\begin{align}\nonumber
&\limsup_{\lambda\to\infty}A_{12}(\lambda,\varepsilon)\\&\nonumber\leq    \limsup_{\lambda\to\infty} \lambda\mathcal{L}^d \left(\left\{x\in B_R: C(n,\varepsilon,\tau,\zeta)\rho_0^{\frac{1}{4}}\sum_{k=0}^{d-2} \sum_{y_{\tau}\in S_{\tau}}  \mathbf{I}_{\frac{3}{4}}^{d-k}(\nu^1_{k,x_{\eta_i^{\kappa}(y_{\tau})}^{2,d-k}},\bigotimes_{i=1}^{d-k} \tilde{H}_{\eta_i^{\kappa}(y_{\tau})})(x_{\eta_i^{\kappa}(y_{\tau})}^{1,d-k})>\lambda \right\}\right)\\&+\limsup_{\lambda\to\infty} \lambda\mathcal{L}^d\left(\left\{x\in B_R: C(n,\varepsilon,\zeta)\sum_{k=0}^{d-2} \sum_{y_{\tau}\in S_{\tau}} \mathbf{1}_{B_{4\tau}(y_\tau) }(x)\mathbf{M}^{d-k}(\nu^2_{k,x_{\eta_i^{\kappa}(y_{\tau})}^{2,d-k}},\bigotimes_{i=1}^{d-k} \tilde{H}_{\eta_i^{\kappa}(y_{\tau})})(x_{\eta_i^{\kappa}(y_{\tau})}^{1,d-k})>\lambda \right\}\right).\label{es-nov234}
\end{align}
We easily derive from the boundedness of $\mathbf{I}_{\frac{3}{4}}^{d-k}(.,X)$ from $\mathcal{M}_b(X)$ to $L^{\frac{d-k}{d-k-\frac{3}{4}},\infty}(X)$ with $X=\bigotimes_{i=1}^{d-k} \tilde{H}_{\eta_i^{\kappa}(y_{\tau})}$  that  the first term in the right hand-side of \eqref{es-nov234} equals zero.
Thanks to \eqref{weak11limsup}, we  get that the second term in the right hand-side of \eqref{es-nov234}  is bounded by 
\begin{align*}
&  \limsup_{\lambda\to\infty} \sum_{k=0}^{d-2}\sum_{y_{\tau}\in S_{\tau}}\lambda\mathcal{L}^d\left(\left\{x\in B_{8\tau}(y_\tau): C(n,\varepsilon,\zeta)\mathbf{M}^{d-k}(\nu^2_{k,x_{\eta_i^{\kappa}(y_{\tau})}^{2,d-k}},\bigotimes_{i=1}^{d-k} \tilde{H}_{\eta_i^{\kappa}(y_{\tau})})(x_{\eta_i^{\kappa}(y_{\tau})}^{1,d-k})>\lambda \right\}\right)\\&\leq C(n,\varepsilon,\zeta) \sum_{k=0}^{d-2}\sum_{y_{\tau}\in S_{\tau}} \int_{\tilde{H}_{\eta_d^{\kappa}(y_{\tau})}}...\int_{\tilde{H}_{\eta_1^{\kappa}(y_{\tau})}} \mathbf{1}_{|\sum_{i=d-k+1}^d((y_{\tau})_{\eta_i^{\kappa}(y_{\tau})}-x_{i})|\leq 8\tau}\\&~~~\times d\nu^{2,s}_{k,\sum_{i=d-k+1}^{d}x_{i}}(x_1,...,x_{d-k})d\mathcal{H}^1(x_{d-k+1})...d\mathcal{H}^1(x_d)\\&\leq  C(n,\varepsilon,\zeta) \sum_{k=0}^{d-2}\sum_{y_{\tau}\in S_{\tau}} \int_{H_{\eta_{d-k}^{\kappa}(y_{\tau})}} \int_{\tilde{H}_{\eta_{d-k}^{\kappa}(y_{\tau})}} \mathbf{1}_{|y_{\tau}-(z_1+z_2)|\leq 16\tau} d|D^{s}f^{\eta_{d-k}^{\kappa}(y_{\tau})}_{z_2}|(z_1)d\mathcal{H}^{d-1}(z_{2}).
\end{align*}
where $\nu^{2,s}_{k,\sum_{i=d-k+1}^{d}x_{i}}(x_1,...,x_{d-k})$ is the singular part of $\nu^{2}_{k,\sum_{i=d-k+1}^{d}x_{i}}(x_1,...,x_{d-k})$.\\
Thanks to \eqref{equa-eta} in Proposition \ref{ML1-BV} and definition of $\eta$, one has 
\begin{align*}\nonumber
\limsup_{\lambda\to\infty}A_{12}(\lambda,\varepsilon)&\leq  C(n,\varepsilon,\zeta) \sum_{k=0}^{d-2}\sum_{y_{\tau}\in S_{\tau}} \int_{\mathbb{R}^d}\mathbf{1}_{B_{20\tau}(y_{\tau})}(x)|\langle \eta_{d-k}^{\kappa}(y_{\tau}), \eta(x)\rangle| d|\mu|^{s}(x).
\end{align*}
 Because of $\langle \eta_{d-k}^{\kappa}(y_{\tau}), \eta^{\kappa}(y_{\tau})\rangle =0$ for any $k=0,1,..,d-2$, so
\begin{align*}
|\langle \eta_{d-k}^{\kappa}(y_{\tau}), \eta(x)\rangle|&\leq|\langle \eta_{d-k}^{\kappa}(y_{\tau}), \eta(x)-\eta^\kappa(x)\rangle|+\langle \eta_{d-k}^{\kappa}(y_{\tau}), \eta^\kappa(x)-\eta^\kappa(y_\tau)\rangle| \\&\leq | (\eta-\eta^{\kappa})(x)|+||\nabla \eta^{\kappa}||_{L^\infty(\mathbb{R}^d)}|x-y_{\tau}|,
\end{align*}
which implies that 
\begin{align*}\nonumber
\limsup_{\lambda\to\infty}A_{12}(\lambda,\varepsilon)&\leq  C(n,\varepsilon,\zeta) \sum_{y_{\tau}\in S_{\tau}} \int_{\mathbb{R}^d}\mathbf{1}_{B_{20\tau}(y_{\tau})}(x)\left[| (\eta-\eta^{\kappa})(x)|+||\nabla \eta^{\kappa}||_{L^\infty(\mathbb{R}^d)}\tau\right] d|\mu|^s(x)\\&\overset{\eqref{es1}}\leq  C(n,\varepsilon,\zeta)||(\eta-\eta^\kappa)|\mu|^s||_{\mathcal{M}(\mathbb{R}^d)}+C(n,\varepsilon,\zeta,\kappa)\tau|\mu|^{s}(\mathbb{R}^d).
\end{align*}
Therefore, we get \eqref{esA13}.\\
\textbf{Step 8:} Estimate $A_{9}(\lambda,\varepsilon)$ and finish the proof. \\Hence, we derive from \eqref{esA6} and  \eqref{esA11'}, \eqref{es14+}, \eqref{es28},
 \eqref{esA13} that 
\begin{align*}
\limsup_{\lambda\to\infty}A_{9}(\lambda,\varepsilon)\lesssim |\log(\varepsilon)||\mu|^{s}(\mathbb{R}^d)+C(n,\varepsilon,\zeta) |||\eta-\eta^{\kappa}||\mu|^{s}||_{\mathcal{M}(\mathbb{R}^d)}+ C(n,\varepsilon,\kappa,\zeta) \tau|\mu|^{s}(\mathbb{R}^d).
\end{align*}
Combining this with   \eqref{ES1-8}  yields 
\begin{align*}
&\limsup_{\lambda\to \infty}\lambda\mathcal{L}^d\left(\left\{ \mathbf{T}^{1,n}_{\varepsilon}(\mu)>\lambda \right\}\cap B_R\right)\lesssim|\log(\varepsilon)||\mu|^{s}(\mathbb{R}^d)\\&~~~+C(n,\varepsilon,\zeta) |||\eta-\eta^{\kappa}||\mu|^{s}||_{\mathcal{M}(\mathbb{R}^d)}+ C(n,\varepsilon,\kappa,\zeta) \tau|\mu|^{s}(\mathbb{R}^d)+C(n,\varepsilon)\zeta^{\alpha}|\mu|^{s}(\mathbb{R}^d).
\end{align*}
At this point, sending $\tau\to 0$, then $\kappa\to 0$ and $\zeta\to 0$, we obtain \eqref{esweakL1b}. 
 The proof is complete.
\end{proof}
\begin{lemma}\label{le-es-ker-log} Let $\mathbf{K}^{\varepsilon,n}_{e,\rho}$ be in \eqref{ker12}. Then, for any $e\in S^{d-1}$ there holds
	\begin{align}\label{es20}
	\sum_{j=-\infty}^{\infty}(2^{j}\rho)^d\int_{S^{d-1}}\sup_{r\in [2^{j}\rho,2^{j+1}\rho]}|\mathbf{K}^{\varepsilon,n}_{e,\rho}(r\theta)|d\mathcal{H}^{d-1}(\theta)\lesssim|\log(\varepsilon)|.
	\end{align}
\end{lemma}
\begin{proof} \textbf{1. Case}: $j\geq 1$. For any $r\in [2^{j}\rho,2^{j+1}\rho],\theta\in S^{d-1}$, we can estimate
	\begin{align*}
	|\mathbf{K}^{\varepsilon,n}_{e,\rho}(r\theta)|&=|\frac{\varepsilon^{-d+1}}{\rho^\alpha}\int_{\mathbb{R}^d}\left[\mathbf{K}_{n}(y)-\mathbf{K}_{n}(r\theta)\right]\varphi_{\zeta\rho}(r\theta-y)\frac{\langle\phi^{e,\varepsilon}((r\theta-y)/\rho),\eta^{\kappa}_{y_{\tau}}\rangle}{|r\theta-y|^{d-\alpha}}dy|
	\\&\lesssim  \frac{\varepsilon^{-d+1}}{\rho^\alpha}\int_{\mathbb{R}^d}\left|\mathbf{K}_{n}(r\theta-y)-\mathbf{K}_{n}(r\theta)\right|\frac{\mathbf{1}_{|y|\leq\rho}\mathbf{1}_{|\frac{y}{|y|}-e|\leq \varepsilon}}{|y|^{d-\alpha}}dy.
	\end{align*}
	By \eqref{con1b}, one has  for $|y|<r/2$,  $$
	\left|\mathbf{K}_{n}(r\theta-y)-\mathbf{K}_{n}(r\theta)\right|\lesssim\frac{|\Omega_n(\theta)||y|}{r^{d+1}}+\frac{1}{r^d}|\Omega_n(r\theta-y)-\Omega_n(r\theta)|.$$ So,  for any $r\in [2^{j}\rho,2^{j+1}\rho]$
	\begin{align*}
	|\mathbf{K}^{\varepsilon,n}_{e,\rho}(r\theta)|&\lesssim \frac{\varepsilon^{-d+1}|\Omega_n(\theta)|}{\rho^\alpha r^{d+1}}\int_{\mathbb{R}^d}\frac{\mathbf{1}_{|y|\leq\rho}\mathbf{1}_{|\frac{y}{|y|}-e|\leq \varepsilon}}{|y|^{d-\alpha-1}}dy\\&+ \frac{\varepsilon^{-d+1}}{\rho^\alpha r^d}\int_{\mathbb{R}^d}|\Omega_n(r\theta-y)-\Omega_n(r\theta)|\frac{\mathbf{1}_{|y|\leq\rho}\mathbf{1}_{|\frac{y}{|y|}-e|\leq \varepsilon}}{|y|^{d-\alpha}}dy\\&\lesssim \frac{2^{-j}|\Omega_n(\theta)|}{(2^j\rho)^{d}}+\frac{2^{j\alpha}\varepsilon^{-d+1}}{(2^j\rho)^{d}}\int_{\mathbb{R}^d}|\Omega_n(\theta-y)-\Omega_n(\theta)|\frac{\mathbf{1}_{|y|\leq 2^{-j}}\mathbf{1}_{|\frac{y}{|y|}-e|\leq \varepsilon}}{|y|^{d-\alpha}}dy.
	\end{align*}
	Thus, 
	\begin{align*}
&	(2^{j}\rho)^d\int_{S^{d-1}}\sup_{r\in [2^{j}\rho,2^{j+1}\rho]}|\mathbf{K}^{\varepsilon,n}_{e,\rho}(r\theta)|d\mathcal{H}^{d-1}(\theta)\\&\leq  2^{-j}||\Omega_n||_{L^1(S^{d-1})}+\varepsilon^{-d+1}2^{j\alpha}\int_{\mathbb{R}^d}\frac{\mathbf{1}_{|y|\leq 2^{-j}}\mathbf{1}_{|\frac{y}{|y|}-e|\leq \varepsilon}}{|y|^{d-\alpha-\alpha_0/2}}dy \sup_{|h|\leq 1/2}\frac{||\Omega_n(\cdot-h)-\Omega_n(\cdot)||_{L^1(S^{d-1})}}{|h|^{\alpha_0/2}}\\& \overset{\eqref{con2'}, \eqref{con2b}}\lesssim 2^{-j}+2^{-j\alpha_0/2}\lesssim 2^{-j\alpha_0/2},
	\end{align*}
	which implies 
	\begin{align}
	\sum_{j=1}^{\infty}(2^{j}\rho)^d\int_{S^{d-1}}\sup_{r\in [2^{j}\rho,2^{j+1}\rho]}|\mathbf{K}^{\varepsilon,n}_{e,\rho}(r\theta)|d\mathcal{H}^{d-1}(\theta)\lesssim 1.\label{es21}
	\end{align}
	\textbf{2. Case}:  $j\leq 0$. We prove that 
	\begin{align}\label{es22}
A_j:=	(2^{j}\rho)^d\int_{S^{d-1}}\sup_{r\in [2^{j}\rho,2^{j+1}\rho]}|\mathbf{K}^{\varepsilon,n}_{e,\rho}(r\theta)|\leq C|\log(\varepsilon)|2^{j\frac{1}{2}\min\{\alpha,1\}}.
	\end{align}
	Indeed, let  $\psi$ be a smooth function in $\mathbb{R}^d$ such that $\psi(x)=1$ if $|x|\leq 1$ and $\psi(x)=0$ if $|x|>2$.\\
Assume $r\in (2^{j}\rho,2^{j+1}\rho]$. 
 One has for any $\theta\in S^{d-1}$,
	\begin{align}\nonumber
A_j&\leq \sum_{i=-\infty}^{1}	(2^{j}\rho)^d\int_{S^{d-1}}\sup_{r\in [2^{j}\rho,2^{j+1}\rho]}|\mathbf{K}_{i,n}(r\theta)| + (2^{j}\rho)^d\int_{S^{d-1}}\sup_{r\in [2^{j}\rho,2^{j+1}\rho]}|\varphi_{\rho}(r\theta)\mathbf{K}_{n}(r\theta)|\\&\lesssim 2^{jd}+ \sum_{i=-\infty}^{1}(2^{j}\rho)^d\int_{S^{d-1}}\sup_{r\in [2^{j}\rho,2^{j+1}\rho]}|\mathbf{K}_{i,n}(r\theta)| ,\label{Z4}
	\end{align}
where 
	\begin{align*}
	\mathbf{K}_{i,n}(r\theta)= \frac{\varepsilon^{-d+1}}{\rho^\alpha}\int_{\mathbb{R}^d}\mathbf{K}_{n}(r\theta-y)\varphi_{\zeta\rho}(y)\frac{\langle\phi^{e,\varepsilon}(y/\rho),\eta^{\kappa}_{y_{\tau}}\rangle}{|y|^{d-\alpha}}\left(\psi(2^{-i}\rho^{-1} y)-\psi(2^{-i+1}\rho^{-1} y)\right)dy.
	\end{align*}
	We now estimate $$\left(\sum_{i=-\infty}^{j-3}+\sum_{i=j+3}^{1}\right)(2^{j}\rho)^d\int_{S^{d-1}}\sup_{r\in [2^{j}\rho,2^{j+1}\rho]}|\mathbf{K}_{i,n}(r\theta)|.$$
To do this, we have 
		\begin{align*}
	&\left(\sum_{i=-\infty}^{j-3}+\sum_{i=j+3}^{1}\right)|\mathbf{K}_{i,n}(r\theta)|\\&\quad\quad\quad\lesssim \left(\sum_{i=-\infty}^{j-3}+\sum_{i=j+3}^{1}\right) \frac{\varepsilon^{-d+1}}{\rho^\alpha}\int_{\mathbb{R}^d}\frac{|\Omega_n(r\theta-y)|}{|r\theta-y|^d}\frac{\mathbf{1}_{|\frac{y}{|y|}-e|\leq \varepsilon}}{|y|^{d-\alpha}}\mathbf{1}_{2^{i-1}\rho<|y|<2^{i+1}\rho}dy\\&\quad\quad\quad\lesssim\left(\sum_{i=-\infty}^{j-3}+\sum_{i=j+3}^{1}\right) \frac{\varepsilon^{-d+1}}{\rho^d 2^{(d-\alpha)j}}\int_{\mathbb{R}^d}\frac{|\Omega_n(\theta-y)|}{1+|y |^d}\frac{\mathbf{1}_{|\frac{y}{|y|}-e|\leq \varepsilon}}{|y|^{d-\alpha}}\mathbf{1}_{2^{i-j-2}<|y|<2^{i-j+1}}dy\\&\quad\quad\quad \lesssim \sum_{i=-\infty}^{j-3} \frac{2^{(i-j)\alpha}}{\rho^d2^{j(d-\alpha)}}G_{2^{i-j+1}}(\theta)+ \sum_{i=j+3}^{1} \frac{1}{\rho^d 2^{i(d-\alpha)}}G_{2^{i-j+1}}(\theta),
	\end{align*}
	where 
	$$G_{\vartheta}(\theta)=\frac{\varepsilon^{-d+1}}{\vartheta^d}\int_{\mathbb{R}^d}|\Omega_n(\theta-y)|\mathbf{1}_{|\frac{y}{|y|}-e|\leq \varepsilon}\mathbf{1}_{\vartheta/8<|y|<\vartheta}dy.$$
	We claim that 
	$$\int_{S^{d-1}}	G_{\vartheta}(\theta)\lesssim 1+\vartheta^{d-1}~~\text{if}~~\vartheta\geq 16~\text{or}~\vartheta\leq 1/2.$$ 
In fact, 
\textit{Case:} $\vartheta\leq 1/2$. Thanks to $\Omega_n(\theta)=\Omega_n(\varsigma\theta)$ for any $\varsigma>0,\theta\in S^{d-1}$, we have
\begin{align*}
\int_{S^{d-1}}G_{\vartheta}(\theta)&=\frac{5}{2} \frac{\varepsilon^{-d+1}}{\vartheta^d}\int_{\mathbb{R}^d} \int_{4/5<|x|<6/5}|\Omega_n(x-|x| y)|dx\mathbf{1}_{|\frac{y}{|y|}-e|\leq \varepsilon}\mathbf{1}_{\vartheta/8<|y|<\vartheta}dy
\\&\lesssim||\Omega_n||_{L^1(S^{d-1})} \frac{\varepsilon^{-d+1}}{\vartheta^d}\int_{\mathbb{R}^d}\mathbf{1}_{|\frac{y}{|y|}-e|\leq \varepsilon}\mathbf{1}_{\vartheta/16<|y|<2\vartheta}dy\overset{\eqref{con2b}}\lesssim 1.
\end{align*}
\textit{Case:} $\vartheta\geq 16$. We have 
	\begin{align*}
\int_{S^{d-1}}G_{\vartheta}(\theta)
&=\frac{5}{2} \frac{\varepsilon^{-d+1}}{\vartheta^d}\int_{\mathbb{R}^d} \int_{4/5<|x|<6/5}|\Omega_n(x-|x| y)|dx\mathbf{1}_{|\frac{y}{|y|}-e|\leq \varepsilon}\mathbf{1}_{\vartheta/16<|y|<2\vartheta}dy  \\&\lesssim \frac{\varepsilon^{-d+1}}{\vartheta^d}\int_{\mathbb{R}^d}\int_{|y|-2<|x|<2+|y|}|\Omega_n(x)|dx\mathbf{1}_{|\frac{y}{|y|}-e|\leq \varepsilon}\mathbf{1}_{\vartheta/32<|y|<4\vartheta}dy dh\\&\lesssim \frac{\varepsilon^{-d+1}}{\vartheta^d}\int_{\mathbb{R}^d}|y|^{d-1}||\Omega_n||_{L^1(S^{d-1})}\mathbf{1}_{|\frac{y}{|y|}-e|\leq \varepsilon}\mathbf{1}_{\vartheta/32<|y|<4\vartheta}dy dh\\&\lesssim\vartheta^{d-1}.
\end{align*}
Therefore, 
	\begin{align}\nonumber
	&\left(\sum_{i=-\infty}^{j-3}+\sum_{i=j+3}^{1}\right)(2^{j}\rho)^d\int_{S^{d-1}}\sup_{r\in [2^{j}\rho,2^{j+1}\rho]}|\mathbf{K}_{i,n}(r\theta)|\\&~\nonumber~~~~\lesssim \sum_{i=-\infty}^{j-3} \frac{2^{jd}2^{(i-j)\alpha}}{2^{j(d-\alpha)}} \int_{S^{d-1}}G_{2^{i-j+1}}(\theta)+ \sum_{i=j+3}^{1} \frac{2^{jd}}{ 2^{i(d-\alpha)}} \int_{S^{d-1}}G_{2^{i-j+1}}(\theta)
	\\&~\nonumber~~~~\lesssim \sum_{i=-\infty}^{j-3} 2^{i\alpha}+ \sum_{i=j+3}^{1} 2^j 2^{i(\alpha-1)}\\&~~~~~\label{es23}\lesssim 2^{j\frac{1}{2}\min\{\alpha,1\}}.
	\end{align}
Here we have used the fact that $j\leq 0$ in the last inequality. \vspace{.2cm}\\
	Next, we estimate 
	$$\sum_{i=j-2}^{j+2}(2^{j}\rho)^d\int_{S^{d-1}}\sup_{r\in [2^{j}\rho,2^{j+1}\rho]}|\mathbf{K}_{i,n}(r\theta)|.$$
 We can decompose 
$$
	\mathbf{K}_{i,n}(r\theta)=\sum_{l=-4}^{\infty} \mathbf{K}_{i,n,l}(r\theta) ~~~i=j-2,...,j+2\leq 2$$
	where 
	\begin{align*}
	\mathbf{K}_{i,n,l}(r\theta)&= \frac{\varepsilon^{-d+1}}{\rho^\alpha}\int_{\mathbb{R}^d}\mathbf{1}_{2^{i-l-1}\rho<|r\theta-y|\leq 2^{i-l}\rho }\mathbf{K}_{n}(r\theta-y)\varphi_{\zeta\rho}(y)\frac{\langle\phi^{e,\varepsilon}(y/\rho),\eta^{\kappa}_{y_{\tau}}\rangle}{|y|^{d-\alpha}}\\&~~~~\times\left(\psi(2^{-i}\rho^{-1} y)-\psi(2^{-i+1}\rho^{-1} y)\right)dy.
	\end{align*}
	First we will show that 	\begin{align}
	\sum_{i=j-2}^{j+2}(2^j\rho)^d\int_{S^{d-1}}\sup_{r\in [2^{j}\rho,2^{j+1}\rho]}|\mathbf{K}_{i,n,l}(r\theta)|\lesssim2^{j\alpha}~~\forall~~l\geq -4.\label{es26}
	\end{align}
In fact, one has
	\begin{align*}
	&(2^j\rho)^d\int_{S^{d-1}}\sup_{r\in [2^{j}\rho,2^{j+1}\rho]}|\mathbf{K}_{i,n,l}(r\theta)|\lesssim  (2^j\rho)^d \frac{\varepsilon^{-d+1}}{\rho^\alpha}\frac{1}{(2^{i-l}\rho)^{d}(2^{i}\rho)^{d-\alpha}}\\&~~~~~~~~~\times\int_{S^{d-1}}\sup_{r\in [2^{j}\rho,2^{j+1}\rho]}\int_{\mathbb{R}^d}|\Omega_n(r\theta-y)|\mathbf{1}_{|\frac{y}{|y|}-e|\leq \varepsilon}\mathbf{1}_{|r\theta-y|\sim 2^{i-l}\rho}\mathbf{1}_{|y|\sim2^{i}\rho}dyd\mathcal{H}^{d-1}(\theta).
	\end{align*}
	We change variable to get that 
	\begin{align*}
	&(2^j\rho)^d\int_{S^{d-1}}\sup_{r\in [2^{j}\rho,2^{j+1}\rho]}|\mathbf{K}_{i,n,l}(r\theta)|\\&\quad\quad\quad\lesssim \varepsilon^{-d+1}2^{j\alpha}2^{ld}\int_{S^{d-1}}\int_{\mathbb{R}^d}|\Omega_n(\theta-y)|\mathbf{1}_{|\frac{y}{|y|}-e|\leq \varepsilon}\mathbf{1}_{|\theta-y|\sim 2^{-l}} dyd\mathcal{H}^{d-1}(\theta).
	\end{align*}
	On the other hand, 
	\begin{align*}
	&\int_{S^{d-1}}\int_{\mathbb{R}^d}|\Omega_n(\theta-y)|\mathbf{1}_{|\frac{y}{|y|}-e|\leq \varepsilon}\mathbf{1}_{|\theta-y|\sim 2^{-l}} dyd\mathcal{H}^{d-1}(\theta)\\&\quad\quad\quad\lesssim  2^{l}\int_{||h|-1|\leq  2^{-l-10}}\int_{\mathbb{R}^d}|\Omega_n(h-y)|\mathbf{1}_{|\frac{y}{|y|}-e|\leq \varepsilon}\mathbf{1}_{|h-y|\sim 2^{-l}}\mathbf{1}_{||y|-1|\lesssim 2^{-l}} dydh\\& \quad\quad\quad\lesssim 2^{l}\int_{\mathbb{R}^d}\left[\int_{\mathbb{R}^d}|\Omega_n(h-y)|\mathbf{1}_{|h-y|\sim 2^{-l}}dh\right]\mathbf{1}_{|\frac{y}{|y|}-e|\leq \varepsilon}\mathbf{1}_{||y|-1|\lesssim 2^{-l}} dy
	\\& \quad\quad\quad\lesssim 2^{-(d-1)l}\int_{\mathbb{R}^d}\mathbf{1}_{|\frac{y}{|y|}-e|\leq \varepsilon}\mathbf{1}_{||y|-1|\lesssim 2^{-l}} dy
		\\&\quad\quad\quad\lesssim 2^{-dl}\varepsilon^{d-1}
	\end{align*}
Consequently,
		\begin{align*}
	(2^j\rho)^d\int_{S^{d-1}}\sup_{r\in [2^{j}\rho,2^{j+1}\rho]}|\mathbf{K}_{i,n,l}(r\theta)|\lesssim 2^{j\alpha},
	\end{align*}
this implies \eqref{es26}. \\
	Next, thanks to  \ref{cance-condi'},  we have for $l_0>100$, 
	\begin{align*}
	&\sum_{l=l_0}^{\infty}|\mathbf{K}_{i,n,l}(r\theta)|\\&\quad\quad\lesssim \sum_{l=l_0}^{\infty} \frac{\varepsilon^{-d+1}}{\rho^\alpha}\int_{\mathbb{R}^d}\mathbf{1}_{|y|\sim 2^{i-l}\rho }|\mathbf{K}_{n}(y)||\Theta(r\theta-y)-\Theta(r\theta)|dy+ \frac{\varepsilon^{-d+1}}{\rho^\alpha}|\Theta(r\theta)|,
	\end{align*}
	where $$
	\Theta(y)=\varphi_{\zeta\rho}(y)\frac{\langle\phi^{e,\varepsilon}(y/\rho),\eta^{\kappa}_{y_{\tau}}\rangle}{|y|^{d-\alpha}}\left(\psi(2^{-i}\rho^{-1} y)-\psi(2^{-i+1}\rho^{-1} y)\right).$$
	Since $|\varphi_{\zeta\rho}(y)|\leq C\mathbf{1}_{|y|>\zeta\rho}, |\nabla\varphi_{\zeta\rho}(y)|\leq \frac{C\mathbf{1}_{\zeta\rho<|y|\leq 2\zeta\rho}}{|y|}$, so we easily see that
$$|\Theta(r\theta)|\lesssim \frac{1_{|\theta-e|\leq \varepsilon}}{(2^{i}\rho)^{d-\alpha}},\quad\quad	|\Theta(r\theta-y)-\Theta(r\theta)|\lesssim \frac{|y|}{\varepsilon}\frac{1}{(2^i\rho)^{d-\alpha+1}},
$$ for any $l>100$, $2^{i-l-1}\rho<|r\theta-y|\leq 2^{i-l}\rho$ and $r\in [2^{j}\rho,2^{j+1}\rho]$. \\
	Thus, we get 
	\begin{align}\nonumber
	&\sum_{l=l_0}^{\infty}\sum_{i=j-2}^{j+2}(2^{j}\rho)^d\int_{S^{d-1}}\sup_{r\in [2^{j}\rho,2^{j+1}\rho]}|\mathbf{K}_{i,n,l}(r\theta)|d\mathcal{H}^{d-1}(\theta)\\&\nonumber\quad \quad\lesssim \sum_{l=l_0}^{\infty}\sum_{i=j-2}^{j+2}(2^{j}\rho)^d\int_{S^{d-1}} \frac{\varepsilon^{-d+1}}{\rho^\alpha}\int_{|y|\sim 2^{i-l}\rho}\frac{|\Omega_n(y)|}{\varepsilon (2^i\rho)^{d-\alpha+1}|y|^{d-1}}dyd\mathcal{H}^{d-1}(\theta)\\\nonumber&\quad\quad\quad+ \sum_{i=j-2}^{j+2} (2^{j}\rho)^d\int_{S^{d-1}}\frac{\varepsilon^{-d+1}}{\rho^\alpha}\frac{\mathbf{1}_{|\theta-e|\leq \varepsilon}}{(2^{i}\rho)^{d-\alpha}}d\mathcal{H}^{d-1}(\theta)\\&\nonumber\quad\quad \lesssim \sum_{l=l_0}^{\infty}\sum_{i=j-2}^{j+2}(2^{j}\rho)^d \frac{\varepsilon^{-d+1}}{\rho^\alpha}\frac{2^{i-l}\rho}{\varepsilon (2^i\rho)^{d-\alpha+1}}+ 2^{j\alpha}
	\\&\quad\quad \lesssim 2^{j\alpha} \left(\varepsilon^{-d}2^{-l_0}+1\right).\label{es27}
	\end{align}
	Therefore, it follows from \eqref{es26} and \eqref{es27} that 
$$
	\sum_{i=j-2}^{j+2}(2^{j}\rho)^d\int_{S^{d-1}}\sup_{r\in [2^{j}\rho,2^{j+1}\rho]}|\mathbf{K}_{i,n}(r\theta)|d\mathcal{H}^{d-1}(\theta)\leq  C2^{j\alpha}(1+l_0+\varepsilon^{-d}2^{-l_0}).$$
	At this point we take $2^{l_0}\sim \varepsilon^{-d}$ and obtain that 
$$
	\sum_{i=j-2}^{j+2}(2^{j}\rho)^d\int_{S^{d-1}}\sup_{r\in [2^{j}\rho,2^{j+1}\rho]}|\mathbf{K}_{i,n}(r\theta)|d\mathcal{H}^{d-1}(\theta)\leq C  2^{j\alpha}|\log(\varepsilon)|.$$
	From this and \eqref{es23}, \eqref{Z4} we get \eqref{es22}.\vspace{.2cm}\\ Then, \eqref{es20} follows from \eqref{es21} and \eqref{es22}. The proof is complete.
\end{proof}

\section{Regular Lagrangian flows and quantitative estimates with $BV$ vector fields }
We first recall some definitions and properties of Regular Lagrangian flows introduced in \cite{BoCrip}.  Given a vector field $\mathbf{B}(t,x):(0,T)\times \mathbb{R}^d\to\mathbb{R}^d$, we assume the following growth condition:\\
\textbf{(R1)} The vector field $\mathbf{B}(t,x)$ can be decomposed as $$
\frac{\mathbf{B}(t,x)}{1+|x|}=\tilde{B}_1(t,x)+\tilde{B}_2(t,x)$$
with $
\tilde{B}_1\in L^1((0,T);L^1(\mathbb{R}^d))$ and $\tilde{B}_2\in L^1((0,T);L^\infty(\mathbb{R}^d))$.\\
We  denote by $L^0_{loc}$ the space of measurable functions endowed with local convergence in measure, and $\mathcal{B}(E_1;E_2)$ the space of bounded functions between the sets $E_1$ and $E_2$, $\log L_{loc}(\mathbb{R}^d)$ the space of measurable functions~$u:\mathbb{R}^d\to \mathbb{R}$ such that $\int_{B_r}\log(1+|u(x)|)dx$ is finite for any $r>0$.
The following is definition of Regular Lagrangian flow:
\begin{definition}\label{de-RLF} If $\mathbf{B}$ is a vector field satisfying \textbf{(R1)}, then for fixed $t_0\in [0,T)$, a map $$
	X\in C([t_0,T];L^0_{\loc}(\mathbb{R}^d))\cap \mathcal{B}([t_0,T];\log  L_{\loc}(\mathbb{R}^d))$$
	is a regular Lagrangian flow in the renormalized sense relative to $\mathbf{B}$ starting at $t_0$ if we have the following: \vspace{.2cm}\\
	\noindent $\mathbf{i)}$ The equation $
	\partial_t(h(X(t,x)))=(\nabla h)(X(t,x))\mathbf{B}(t,X(t,x))$
	holds in $\mathcal{D}'((t_0,T)\times\mathbb{R}^d)$, for every function $h\in C^1(\mathbb{R}^d,\mathbb{R})$
	that satisfies $
	|h(z)|\leq C(1+\log(1+|z|))$ and $|\nabla h(z)|\leq \frac{C}{1+|z|}$ for all $z\in \mathbb{R}^d$,
	 \\
	\noindent $\mathbf{ii)}$ $X(t_0,x)=x$ for $\mathcal{L}^d-$a.e $x\in \mathbb{R}^d$,\\
		\noindent $\mathbf{iii)}$  There exists a constant $L>0$ such that  $X(t,.)_{\#}\mathcal{L}^d \leq L\mathcal{L}^d$ for any $t\in [t_0,T]$ i.e $$
		\int_{\mathbb{R}^d}\varphi(X(t,x))dx
		\leq L\int_{\mathbb{R}^d}\varphi(x)dx,$$
		for all measurable $\varphi:\mathbb{R}^d\to[0,\infty)$. 
The constant $L$ in \textbf{iii)} will be called the compressibility constant of $X$. 
\end{definition}

We define the sub-level of the flow as 
$$
G_R=\left\{x\in \mathbb{R}^d: |X(t,x)|\leq R~~\text{for almost all }~~t\in [t_0,T] \right\}.$$
The following Lemma gives a basic estimate for the decay of the super-levels of a regular Lagrangian flow. This Lemma was proven in \cite{BoCrip}.
\begin{lemma}\label{decay-lem}  Let $\mathbf{B}$ be a vector field satisfying \textbf{(R1)} and let $X$ be a regular Lagrangian flow relative  to $\mathbf{B}$ starting at time $t_0$, with compressibility constant $L$. Then for all $r,\lambda>0$ we have $
	\mathcal{L}^d(B_r\backslash G_R)\leq g(r,R)$
	where the function $g$ depends only on $L$, $||\tilde{B}_1||_{L^1((0,T);L^1(\mathbb{R}^d))}$ and $||\tilde{B}_2||_{L^1((0,T);L^\infty(\mathbb{R}^d))}$ and satisfies $g(r,R)\downarrow 0$ for $r$ fixed and $R\uparrow \infty.$
\end{lemma}
The following is our main theorem.
\begin{theorem}\label{mainthm2} Let $\mathbf{B}\in L^1([0,T];L^1_{\loc}(\mathbb{R}^d,\mathbb{R}^d))$ and $R>1$. Assume that  
	\begin{align}\label{vectorfiB}	\mathbf{B}^i=\sum_{j=1}^{m}\mathbf{K}_j^i\star b_j~~\text{in}~B_{2R}, ~\text{with}~b_j\in L^1([0,T],BV(\mathbb{R}^d)),
	\end{align}
	where $(\mathbf{K}_j^i)_{i,j}$ are singular kernels in $\mathbb{R}^d$  satisfying conditions of singular kernel  $\mathbf{K}$ in Theorem \ref{mainthm1} with constants $c_1,c_2>0$.
	Let  $t_0\in [0,T)$, $\mathbf{B}_1,\mathbf{B}_2\in L^1([0,T];L^1_{\loc}(\mathbb{R}^d,\mathbb{R}^d))$ and let $X_1,X_2$ be regular Lagrangian flows starting at time $t_0$ associated to $\mathbf{B}_1,\mathbf{B}_2$ resp. with  compression constants $L_1,L_2\leq L_0$ for some $L_0>0$. Assume that  $
		||(\mathbf{B}_1,\mathbf{B}_2)||_{L^1([0,T]\times B_{R})}\leq c_R.$ Then, if $\operatorname{div}(\mathbf{B})\in L^1((0,T),\mathcal{M}_{b}(B_{2R}))$ and  $(\operatorname{div}(\mathbf{B}))^+\in L^1((0,T),L^1(B_{2R}))$, for any $\kappa\in (0,1), r>1$ there exists $\delta_0=\delta_0(d,T,r,R,c_R,c_1,c_2,L_0,b,\kappa)\in (0,1/100)$ such that 
\begin{align}\nonumber
&\sup_{t_1\in [t_0,T]}\mathcal{L}^d\left(\left\{x\in B_r:|X_{1t_1}(x)-X_{2t_1}(x)|>\delta^{1/2}\right\}\right)\lesssim \mathcal{L}^d\left(B_r\backslash G_{1,R}\right)+\mathcal{L}^d\left(B_r\backslash G_{2,R}\right)\\& + \frac{L_0}{\delta }||\left(\mathbf{B}_1-\mathbf{B},\mathbf{B}_2-\mathbf{B}\right)||_{L^1([0,T]\times B_R)}+\kappa~~\text{for any}~~\delta\in (0,\delta_0).\label{quatity}
\end{align}
 where $
G_{i,R}=\left\{x\in \mathbb{R}^N:|X_i(s,x)|\leq R~~\text{ for almost all}~s\in [t_0,T]\right\}$ for i=1,2.
\end{theorem}
We derive from Theorem \ref{mainthm2} and Lemma \ref{decay-lem} that 
\begin{corollary}
\label{maincorollary} Let $\mathbf{B}\in L^1([0,T];L^1_{\loc}(\mathbb{R}^d,\mathbb{R}^d))$. Assume that  for any $R>0$, there exist singular kernels $(\mathbf{K}_j^i)_{i,j}$ $(i=1,..,d,j=1,...,m(R))$ in $\mathbb{R}^d$  satisfying conditions of singular kernel  $\mathbf{K}$  in Theorem \ref{mainthm1} with constants $c_{1R},c_{2R}>0$;  and   $b_{jR}\in L^1([0,T],BV(\mathbb{R}^d))$ such that 
\begin{align}\label{vectorfiB'}	\mathbf{B}^i=\sum_{j=1}^{m}\mathbf{K}_{jR}^i\star b_{jR}~~\text{in}~B_{2R}.
\end{align}
Let  $t_0\in [0,T)$, $\mathbf{B}_1,\mathbf{B}_2\in L^1([0,T];L^1_{\loc}(\mathbb{R}^d,\mathbb{R}^d))$ and let $X_1,X_2$ be regular Lagrangian flows starting at time $t_0$ associated to $\mathbf{B}_1,\mathbf{B}_2$ resp. with  compression constants $L_1,L_2\leq L_0$ for some $L_0>0$. Assume that $\mathbf{B}_1,\mathbf{B}_2$ satisfy  $(\mathbf{R}_1)$ i.e  $\frac{\mathbf{B}_l(t,x)}{|x|+1}=\tilde{B}_{1l}
(t,x)+\tilde{B}_{2l}(t,x)$ $l=1,2$ with
$$\sum_{l=1,2}||\tilde{B}_{1l}||_{L^1((0,T);L^1(\mathbb{R}^d))}+||\tilde{B}_{2l}||_{L^1((0,T);L^\infty(\mathbb{R}^d))}\leq C_0.$$ Then, if $\operatorname{div}(\mathbf{B})\in L^1((0,T),\mathcal{M}_{\loc}(\mathbb{R}^d))$ and  $(\operatorname{div}(\mathbf{B}))^+\in L^1((0,T),L^1_{\loc}(\mathbb{R}^d))$, for any $\kappa\in (0,1), r>1$ there exist $R_0=R_0(d,T,r,C_0,L_0,\kappa)>1$, $\delta_0=\delta_0(d,T,r,C_0,c_{1R_0},c_{2R_0},L_0,b_{R_0},\kappa)\in (0,1/100)$ such that 
\begin{align}
&\sup_{t_1\in [t_0,T]}\mathcal{L}^d\left(\left\{x\in B_r:|X_{1t_1}(x)-X_{2t_1}(x)|>\delta^{1/2}\right\}\right)\lesssim \frac{L_0}{\delta }||\left(\mathbf{B}_1-\mathbf{B},\mathbf{B}_2-\mathbf{B}\right)||_{L^1([0,T]\times B_{R_0})}+\kappa,\label{quatityb}
\end{align}
for any $\delta\in (0,\delta_0)$.
\end{corollary}
\begin{proof}[Proof of Theorem \ref{mainthm2}] Without loss generality, we assume $t_0=0$. \\ \textbf{Step 1:} By Proposition \ref{albertithm},  there exist unit vectors $\xi_t(x)\in \mathbb{R}^m,\eta_t(x)\in \mathbb{R}^d$ such that $D^{s}b_t(x)=\xi_t(x)\otimes\eta_t(x)|D^{s}b_t|(x)$ i.e $D^{s}_{x_j}b_{tk}(x)=\xi_{tk}(x)\eta_{tj}(x)|D^{s}b_t|(x)$ for any $k=1,...,m,j=1,...,d$.\\ 
	Let $\eta_t^{\varepsilon}\in C^\infty((0,T)\times\mathbb{R}^d,\mathbb{R}^d), \xi_t^{\varepsilon}\in C^\infty((0,T)\times\mathbb{R}^d,\mathbb{R}^m)$ be such that $
	|\eta_t^\varepsilon|=|\xi_t^\varepsilon|=1$
	and 
	\begin{align*}
	\lim_{\varepsilon\to 0}\int_{0}^{T}\int_{\mathbb{R}^d}|\eta_t-\eta_t^\varepsilon|d|D^{s}b_t|dt+\lim_{\varepsilon\to 0}\int_{0}^{T}\int_{\mathbb{R}^d}|\xi_t-\xi_t^\varepsilon|d|D^{s}b_t|dt=0.
	\end{align*}
	For $\delta\in (0,\frac{1}{100}),1<\gamma<|\log(\delta)|,\varepsilon>0$, and $t\in [0,T]$, let us define the quantity
	\begin{align}
	\Phi_\delta^{\gamma,\varepsilon}(t)=\frac{1}{2}\int_{D}\log\left(1+\frac{|X_{1t}(x)-X_{2t}(x)|^2+\gamma \langle \eta_t^\varepsilon(X_{1t}(x)),X_{1t}(x)-X_{2t}(x)\rangle ^2}{\delta^2}\right)dx.
	\end{align}
	where $D=B_r\cap G_{1,R}\cap G_{2,R}$. Since $\partial_t X_{jt}=\mathbf{B}_{jt}(X_{jt})$ $j=1,2$, one has for any $t_1\in [0,T]$
	\begin{align}\nonumber
\sup_{t_1\in [0,T]}	\Phi_\delta^{\gamma,\varepsilon}(t_1)&=\sup_{t_1\in [0,T]}\int_{0}^{t_1}	\frac{d \Phi_\delta^{\gamma,\varepsilon}(t)}{dt} dt\\&\nonumber\leq \sup_{t_1\in [0,T]}\int_{0}^{t_1}\int_{D}\frac{\langle X_{1t}-X_{2t},\mathbf{B}_{1t}(X_{1t})-\mathbf{B}_{2t}(X_{2t})\rangle}{\delta^2+|X_{1t}-X_{2t}|^2+\gamma \langle \eta_t^\varepsilon(X_{1t}),X_{1t}-X_{2t}\rangle ^2}dxdt\\&\nonumber+\sup_{t_1\in [0,T]} \int_{0}^{t_1}\int_{D}\frac{\gamma\langle \eta^\varepsilon(X_{1t}),X_{1t}-X_{2t}\rangle \langle \eta_t^\varepsilon(X_{1t}),\mathbf{B}_{1t}(X_{1t})-\mathbf{B}_{2t}(X_{2t})\rangle}{\delta^2+|X_{1t}-X_{2t}|^2+\gamma \langle \eta_t^\varepsilon(X_{1t}),X_{1t}-X_{2t}\rangle ^2}dxdt
	\\&\nonumber+ \sup_{t_1\in [0,T]}\int_{0}^{t_1}\int_{D}\frac{\gamma\langle \eta^\varepsilon(X_{1t}),X_{1t}-X_{2t}\rangle \langle (\nabla\eta_t^\varepsilon)(X_{1t})\mathbf{B}_{1t}(X_{1t}),X_{1t}-X_{2t}\rangle}{\delta^2+|X_{1t}-X_{2t}|^2+\gamma \langle \eta_t^\varepsilon(X_{1t}),X_{1t}-X_{2t}\rangle ^2}dxdt	\\&\nonumber+\sup_{t_1\in [0,T]} \int_{0}^{t_1}\int_{D}\frac{\gamma\langle \eta^\varepsilon(X_{1t}),X_{1t}-X_{2t}\rangle \langle (\partial_t\eta_t^\varepsilon)(X_{1t}),X_{1t}-X_{2t}\rangle}{\delta^2+|X_{1t}-X_{2t}|^2+\gamma \langle \eta_t^\varepsilon(X_{1t}),X_{1t}-X_{2t}\rangle ^2}dxdt\\&=I_1(\delta,\varepsilon,\gamma)+I_2(\delta,\varepsilon,\gamma)+I_3(\delta,\varepsilon,\gamma)+I_4(\delta,\varepsilon,\gamma).\label{quantity1}
	\end{align}
	By $||\eta^\varepsilon||_{L^\infty((0,T)\times\mathbb{R}^d)}\leq 1,$  and changing variable along the flows with $(X_{jt})_{\#}\mathcal{L}^d\leq L_0\mathcal{L}^d$ for all $t\in [0,T]$ and $j=1,2$, we get 
	\begin{align}\nonumber
	I_1(\delta,\varepsilon,\gamma)+I_2(\delta,\varepsilon,\gamma)&\lesssim \frac{L_0\gamma^{1/2}}{\delta}||\left(\mathbf{B}_1-\mathbf{B},\mathbf{B}_2-\mathbf{B}\right)||_{L^1([0,T]\times B_R)}\\&+I_{5}(\delta,\varepsilon,\gamma)+I_{6}(\delta,\varepsilon,\gamma),\label{esI1I2}
	\end{align}
	and 
	\begin{align}\nonumber
	&|I_3(\delta,\varepsilon,\gamma)|\lesssim L_0\gamma^{1/2}||\nabla \eta^\varepsilon||_{L^\infty((0,T)\times\mathbb{R}^d)}||\mathbf{B}_1||_{L^1([0,T]\times B_R)},\\& |I_4(\delta,\varepsilon,\gamma)|\lesssim \gamma^{1/2} r^d T ||\partial_t \eta^\varepsilon||_{L^\infty((0,T)\times\mathbb{R}^d)},\label{esI3}
	\end{align}
	where
	\begin{align*}
&	I_{5}(\delta,\varepsilon,\gamma)=\sup_{t_1\in [0,T]}\int_{0}^{t_1}\int_{D}\frac{\langle X_{1t}-X_{2t},\mathbf{B}_{t}(X_{1t})-\mathbf{B}_{t}(X_{2t})\rangle}{\delta^2+|X_{1t}-X_{2t}|^2+\gamma \langle \eta_t^\varepsilon(X_{1t}),X_{1t}-X_{2t}\rangle ^2}dxdt;
	\\&I_{6}(\delta,\varepsilon,\gamma)=\sup_{t_1\in [0,T]}\int_{0}^{t_1}\int_{D}\frac{\gamma\langle \eta_t^\varepsilon(X_{1t}),X_{1t}-X_{2t}\rangle \langle \eta_t^\varepsilon(X_{1t}),\mathbf{B}_{t}(X_{1t})-\mathbf{B}_{t}(X_{2t})\rangle}{\delta^2+|X_{1t}-X_{2t}|^2+\gamma \langle \eta_t^\varepsilon(X_{1t}),X_{1t}-X_{2t}\rangle ^2}dxdt.
	\end{align*}
On the other hand, 
\begin{align}\nonumber
\sup_{t_1\in [0,T]}	\Phi_\delta^{\gamma,\varepsilon}(t_1)&\geq \frac{1}{2}|\log(\delta)| \sup_{t_1\in [0,T]}\mathcal{L}^d\left(\left\{x\in D:|X_{1t_1}(x)-X_{2t_1}(x)|>\delta^{1/2}\right\}\right)\\&\nonumber\geq \frac{1}{2}|\log(\delta)| \sup_{t_1\in [0,T]}\mathcal{L}^d\left(\left\{x\in B_r:|X_{1t_1}(x)-X_{2t_1}(x)|>\delta^{1/2}\right\}\right)\\&~~~~- \frac{1}{2}|\log(\delta)| \left(\mathcal{L}^d\left(B_r\backslash G_{1,R}\right)+\mathcal{L}^d\left(B_r\backslash G_{2,R}\right)\right).\label{belowes-quanti}
\end{align}
It follows from \eqref{quantity1},\eqref{esI1I2},\eqref{esI3} and \eqref{belowes-quanti}  and $\gamma<|\log(\delta)|$ that  for any $t_1\in [0,T]$
\begin{align}\nonumber
&\sup_{t_1\in [0,T]}\mathcal{L}^d\left(\left\{x\in D:|X_{1t_1}(x)-X_{2t_1}(x)|>\delta^{1/2}\right\}\right)\\&\quad \lesssim \mathcal{L}^d\left(B_r\backslash G_{1,R}\right)+\mathcal{L}^d\left(B_r\backslash G_{2,R}\right)\nonumber +\frac{C(\varepsilon,\gamma,r,T)}{|\log(\delta)|}\left(L_0 ||B_1||_{L^1([0,T]\times B_R)}+1\right)\\&\quad\quad + \frac{L_0}{\delta }||\left(\mathbf{B}_1-\mathbf{B},\mathbf{B}_2-\mathbf{B}\right)||_{L^1([0,T]\times B_R)}+\frac{I_{5}(\delta,\varepsilon,\gamma)}{|\log(\delta)|}+\frac{I_{6}(\delta,\varepsilon,\gamma)}{|\log(\delta)|}.\label{quatity2}
\end{align}
\textbf{Step 2:} We prove that for any $\varepsilon_1\in (0,1/100)$,
	\begin{align}\nonumber
&	\limsup_{\delta\to 0}\frac{I_{5}(\delta,\varepsilon,\gamma)}{|\log(\delta)|}\leq C(\varepsilon_1)\int_{0}^{T}\int_{\mathbb{R}^d} |\eta_t-\eta_t^{\varepsilon}| d|D^{s}b_{t}|dt\\&+C(L_0)\varepsilon_1|\log(\varepsilon_1)|\int_{0}^{T}\int_{\mathbb{R}^d}  d|D^{s}b_{t}|dt+ C(L_0,\varepsilon_1)\gamma^{-1/2}\int_{0}^{T}\int_{\mathbb{R}^d} d|D^{s}b_{t}|dt.\label{esI5}
	\end{align}
	Indeed, thanks to \eqref{A1-sec-thm} in Lemma \ref{le-mainthm2} below with $x_1= X_{1t},x_2=X_{2t}\in B_{R}$ and changing variable along the flows with $(X_{lt})_{\#}\mathcal{L}^d\leq L_0\mathcal{L}^d$ for all $t\in [0,T]$ and $l=1,2$, we find that 
	\begin{align}\nonumber
\limsup_{\delta\to 0}	\frac{	I_{5}(\delta,\varepsilon,\gamma)}{|\log(\delta)|}&\leq \limsup_{\delta\to 0}\frac{2L_0}{|\log(\delta)|} \sum_{i,j} \int_{0}^{T}\int_{B_R} \frac{\mathbf{P}_1(Db)}{\delta}\wedge \mathbf{T}^{1}_{\varepsilon_1,i,j}(D^ab_j) dxdt\\&\nonumber+ \limsup_{\delta\to 0}\frac{2L_0}{|\log(\delta)|}\sum_{i,j} \int_{0}^{T}\int_{B_R} \frac{\mathbf{P}_1(Db)}{\delta}	\wedge \mathbf{T}^{1}_{\varepsilon_1,i,j}(\omega_{tij}^\varepsilon) dxdt\\ &\nonumber+ \limsup_{\delta\to 0}\frac{2L_0  }{|\log(\delta)|} \int_{0}^{T}\int_{B_R} \mathbf{P}_1(Db)dxdt \\&\nonumber+ \limsup_{\delta\to 0}\frac{2L_0\varepsilon_1}{|\log(\delta)|} \sum_{i,j} \int_{0}^{T}\int_{B_R} \frac{\mathbf{P}_1(Db)}{\delta}\wedge \mathbf{T}^{2}_{\varepsilon_1,i,j}(Db_{tj}) dxdt
	\\ &\nonumber+ \limsup_{\delta\to 0}\frac{2L_0  \gamma^{-1/2}}{|\log(\delta)|}\sum_{i,j} \int_{0}^{T}\int_{B_R} \frac{\mathbf{P}_1(Db)}{\delta}\wedge \mathbf{T}^{1}_{\varepsilon_1,i,j}(\xi_{tj}|D^{s}b_{tj}|)dxdt\\&= (1)+(2)+(3)+(4)+(5),\label{I_5-lim}
	\end{align}
	where $\sum_{i,j}:=\sum_{i=1}^{d}\sum_{j=1}^{m}$, $\omega_{tij}^\varepsilon:=(\eta_t-\eta_t^{\varepsilon})\xi_{tj}|D^{s}b_{tj}|$ and $\mathbf{T}^{1}_{\varepsilon_1,i,j},  \mathbf{T}^{2}_{\varepsilon_1,i,j}$ are defined in Lemma \ref{le-mainthm2} and $\mathbf{P}_1(Db)\in L^1((0,T),L^{q_0}_{\loc}(\mathbb{R}^d))$ for some $q_0>1$.\medskip\\	
	Clearly, $(3)=0$. We can apply \eqref{es6-} in Proposition \ref{singu-operator} (and Remark \ref{rem3}) to $\mathbf{T}^{1}_{\varepsilon_1,i,j}$ and $f=\mathbf{P}_1(Db)$  to get that $(1)=0$,$$
	(2)\leq C(\varepsilon_1)\int_{0}^{T}\int_{\mathbb{R}^d} |\eta_t-\eta_t^{\varepsilon}| d|D^{s}b_{t}|dt; ~~(5)\leq C(\varepsilon_1)\gamma^{-1/2}\int_{0}^{T}\int_{\mathbb{R}^d} d|D^{s}b_{t}|dt.$$
On the other hand, 	it is clear to see that $\mathbf{K}_j^{i}$ and $\Theta^{\varepsilon_1,e}_2$ satisfy Theorem \ref{mainthm1}. So, we can apply  \eqref{es6} in Theorem \ref{mainthm1} to $\mathbf{T}^{2}_{\varepsilon_1,i,j}$ and  $f=\mathbf{P}_1(Db)$, (with  $ \alpha=1, \varepsilon=\varepsilon_1$) and obtain that 
$$
(4)\leq C\varepsilon_1|\log(\varepsilon_1)|\int_{0}^{T}\int_{\mathbb{R}^d}  d|D^{s}b_{t}|dt.$$
Plugging above estimates into \eqref{I_5-lim} gives \eqref{esI5}.\\
	\textbf{Step 3:} 
	We prove that for any $\varepsilon_2\in (0,1/100)$
	\begin{align}
	&	\limsup_{\delta\to 0}\frac{I_{6}(\delta,\varepsilon,\gamma)}{|\log(\delta)|}\leq C(\varepsilon_2)\gamma^{1/2}\int_{0}^{T}\int_{\mathbb{R}^d} |\eta_t-\eta_t^{\varepsilon}| d|D^{s}b_{t}|dt+C\gamma^{1/2}\varepsilon_2|\log(\varepsilon_2)|\int_{0}^{T}\int_{\mathbb{R}^d} d|D^{s}b_{t}|dt.\label{esI6}
	\end{align}
	Indeed, thanks to \eqref{A2-sec-thm} in Lemma \ref{le-mainthm2} below with $x_1= X_{1t},x_2=X_{2t}\in B_R,\varepsilon_1=\varepsilon_2$ and changing variable along the flows with $(X_{lt})_{\#}\mathcal{L}^d\leq L_0\mathcal{L}^d$ for all $t\in [0,T]$ and $l=1,2$, we find that 
	\begin{align*}
	\limsup_{\delta\to 0}	\frac{I_{6}(\delta,\varepsilon,\gamma)}{|\log(\delta)|}&\leq	\limsup_{\delta\to 0} \frac{4\gamma^{1/2}L_0 }{|\log(\delta)|}\sum_{i,j} \int_{0}^{T}\int_{B_R} \frac{\mathbf{P}_2(Db)}{\delta}\wedge \mathbf{T}^{1}_{\varepsilon_2,i,j}(D^ab_j) dxdt\\&+ 	\limsup_{\delta\to 0}\frac{4\gamma^{1/2}L_0}{|\log(\delta)|} \sum_{i,j} \int_{0}^{T}\int_{B_R} \frac{\mathbf{P}_2(Db)}{\delta}	\wedge \mathbf{T}^{1}_{\varepsilon_2,i,j}(\omega_{tij}^\varepsilon)dxdt\\ &+	\limsup_{\delta\to 0} \frac{4L_0  \gamma^{1/2}}{|\log(\delta)|} \sum_{i,j} \int_{0}^{T}\int_{B_R} \mathbf{P}_2(Db)dxdt \\&+ \limsup_{\delta\to 0} \frac{2L_0\gamma^{1/2}\varepsilon_2}{|\log(\delta)|} \sum_{i,j} \int_{0}^{T}\int_{B_R} \frac{\mathbf{P}_2(Db)}{\delta}\wedge \mathbf{T}^{2}_{\varepsilon_2,i,j}(Db_j) dxdt
	\\ &+	\limsup_{\delta\to 0} \frac{C(\varepsilon_2,\gamma)}{|\log(\delta)|} \int_{0}^{T}\int_{B_R} \frac{\mathbf{I}_1(\mathbf{1}_{B_{4\lambda}}(\operatorname{div}^a(B_t))^+)}{\delta}\wedge \mathbf{M}(\mathbf{1}_{B_{4\lambda}}(\operatorname{div}^a(B_t))^+)dxdt\\&= (6)+(7)+(8)+(9)+(10),
	\end{align*}
	where $\omega_{tij}^\varepsilon:=(\eta_t-\eta_t^{\varepsilon})\xi_{tj}|D^{s}b_{tj}|$ and $\mathbf{P}_2(Db)\in L^1((0,T),L^{q_0}_{\loc}(\mathbb{R}^d))$ for some $q_0>1$. 
	Similarly, we also obtain that $(6)+(8)=0$ and 
$$	(7)\leq C(\varepsilon_2)\gamma^{1/2}\int_{0}^{T}\int_{\overline{B}_\lambda} |\eta_t-\eta_t^{\varepsilon}| d|D^{s}b_{t}|dt;~~~(9)\leq C\gamma^{1/2}\varepsilon_2|\log(\varepsilon_2)|\int_{0}^{T}\int_{\overline{B}_\lambda} d|D^{s}b_{t}|dt.$$
Moreover, by \ref{es5} in Lemma \eqref{le2}, one has $(10)=0$.
	Thus, we get \eqref{esI6}. Therefore, we derive from \eqref{quatity2} and  \eqref{esI5}, \eqref{esI6}  that 
	\begin{align}\nonumber
	&\sup_{t_1\in [0,T]}\mathcal{L}^d\left(\left\{x\in D:|X_{1t_1}(x)-X_{2t_1}(x)|>\delta^{1/2}\right\}\right)\\&\quad \lesssim \mathcal{L}^d\left(B_r\backslash G_{1,R}\right)+\mathcal{L}^d\left(B_r\backslash G_{2,R}\right)+ \frac{L_0}{\delta }||\left(\mathbf{B}_1-\mathbf{B},\mathbf{B}_2-\mathbf{B}\right)||_{L^1([0,T]\times B_R)}+A(\delta),\label{quatity3}
	\end{align}
	and
	\begin{align}\nonumber\limsup_{\delta\to 0} A(\delta)
	&\leq 2 \limsup_{\delta\to 0}\frac{I_{5}(\delta,\varepsilon,\gamma)}{|\log(\delta)|}+2 \limsup_{\delta\to 0}\frac{I_{6}(\delta,\varepsilon,\gamma)}{|\log(\delta)|}\\&
\nonumber	\leq C(\varepsilon_1)\int_{0}^{T}\int_{\mathbb{R}^d} |\eta_t-\eta_t^{\varepsilon}| d|D^{s}b_{t}|dt+C\varepsilon_1|\log(\varepsilon_1)|\int_{0}^{T}\int_{\mathbb{R}^d}  d|D^{s}b_{t}|dt\\&\nonumber+ C(\varepsilon_1)\gamma^{-1/2}\int_{0}^{T}\int_{\mathbb{R}^d} d|D^{s}b_{t}|dt+C(\varepsilon_2)\gamma^{1/2}\int_{0}^{T}\int_{\mathbb{R}^d} |\eta_t-\eta_t^{\varepsilon}| d|D^{s}b_{t}|dt\\&~~~~~~~~~~~~~~~~~~~~~~~~~+C\gamma^{1/2}\varepsilon_2|\log(\varepsilon_2)|\int_{0}^{T}\int_{\mathbb{R}^d} d|D^{s}b_{t}|dt.\label{esA}
	\end{align}
In the right hand side of \eqref{esA}, we  let $\varepsilon\to0$, then $\varepsilon_2\to0$,  $\gamma\to\infty$ and $\varepsilon_1\to 0$ to get that $
\limsup_{\delta\to 0} A(\delta)\leq 0$. Combining this and \eqref{quatity3}  yields \eqref{quatity}. The proof is complete.
\end{proof}
Let $\Theta^{\varepsilon,e}_1, \Theta^{\varepsilon,e}_2$ be in Lemma \ref{Jabin-lem}.  For any   $i=1,...,d$, $x_1\not= x_2\in B_R(0)$ and $\varepsilon_1\in (0,1/100)$, we define:
$\mathbf{e}_1=-\mathbf{e}_2=\frac{x_1-x_2}{|x_1-x_2|}$, $r=|x_1-x_2|$ and 
\begin{enumerate} \abovedisplayskip=0pt \belowdisplayskip=0pt
	\item[] 
	\begin{multline}
	\Theta^{\varepsilon_1,e}_{l,r}(.)=\Theta^{\varepsilon_1,e}_{l}(\frac{.}{r}),~~~ \tilde{\Theta}^{\varepsilon_1,e}_{l,r}(.)=\frac{1}{r}\frac{\varepsilon_1^{-d+1}}{|.|^{d-1}}\Theta^{\varepsilon_1,e}_{l,r}(.),~ l=1,2;~~~~~~~~~~
	\end{multline}
	\begin{multline}
	A_{i1}^{\text{reg}}:=\sum_{j=1}^{m}\left[\mathbf{K}_j^i\star	\tilde{\Theta}^{\varepsilon_1,\mathbf{e}_1}_{1,r}\star (\mathbf{e}_1.D^ab_j)\right](x_1)-\left[\mathbf{K}_j^i\star\tilde{\Theta}^{\varepsilon_1,\mathbf{e}_2}_{1,r}\star (\mathbf{e}_2.D^ab_j)\right](x_2);
	\end{multline}
	\begin{multline}
	A_{i1}^{\text{appro}}:=\sum_{j=1}^{m}\left[\mathbf{K}_j^i\star\tilde{\Theta}^{\varepsilon_1,\mathbf{e}_1}_{1,r}\star \left((\mathbf{e}_1.(\eta_t-\eta_t^{\varepsilon}))\xi_{tj}|D^{s}b_{tj}|\right)\right](x_1) \\ 
	+\left[\mathbf{K}_j^i\star\tilde{\Theta}^{\varepsilon_1,\mathbf{e}_2}_{1,r}\star \left((\mathbf{e}_1.(\eta_t-\eta_t^{\varepsilon}))\xi_{tj}|D^{s}b_{tj}|\right)\right](x_2);
	\end{multline} 
	\begin{multline}
	A_{i1}^{\text{diff-1}}:=\sum_{j=1}^{m}\left[\mathbf{K}_j^i\star\tilde{\Theta}^{\varepsilon_1,\mathbf{e}_1}_{1,r}\star \left((\mathbf{e}_1.(\eta_t^{\varepsilon}-\eta_t^{\varepsilon}(x_1)))\xi_{tj}|D^{s}b_{tj}|\right)\right](x_1)\\+\left[\mathbf{K}_j^i\star\tilde{\Theta}^{\varepsilon_1,\mathbf{e}_2}_{1,r}\star \left((\mathbf{e}_1.(\eta_t^{\varepsilon}-\eta_t^{\varepsilon}(x_2)))\xi_{tj}|D^{s}b_{tj}|\right)\right](x_2);
	\end{multline}
	\begin{multline}
	A_{i1}^{\text{diff-2}}:=\sum_{j=1}^{m}(\mathbf{e}_1.(\eta_t^{\varepsilon}(x_2)-\eta_t^{\varepsilon}(x_1)))\left[\mathbf{K}_j^i\star\tilde{\Theta}^{\varepsilon_1,\mathbf{e}_2}_{1,r}\star \left(\xi_{tj}|D^{s}b_{tj}|\right)\right](x_2);~~~~~~~
	\end{multline}
	\begin{multline}
	A_{i1}^{\text{sing}}:=\sum_{j=1}^{m}\left[\mathbf{K}_j^i\star\tilde{\Theta}^{\varepsilon_1,\mathbf{e}_1}_{1,r}\star \left(\xi_{tj}|D^{s}b_{tj}|\right)\right](x_1)
	+\left[\mathbf{K}_j^i\star\tilde{\Theta}^{\varepsilon_1,\mathbf{e}_2}_{1,r}\star \left(\xi_{tj}|D^{s}b_{tj}|\right)\right](x_2);
	\end{multline}
	\begin{multline}
	A_{i2}:=\sum_{j=1}^{m}\left[\mathbf{K}_j^i\star\tilde{\Theta}^{\varepsilon_1,\mathbf{e}_1}_{2,r}\star Db_j\right](x_1)
	-\left[\mathbf{K}_j^i\star\tilde{\Theta}^{\varepsilon_1,\mathbf{e}_2}_{2,r}\star Db_j\right](x_2);~~~~~~~~~~~~~
	\end{multline}
	\begin{multline}
	E^{\text{reg}}:=\sum_{i=1}^{d}\sum_{j=1}^{m}-\left[\mathbf{K}_j^i\star\tilde{\Theta}^{\varepsilon_1,\mathbf{e}_1}_{1,r}\star \left(D^{a}_{x_i}b_{tj}\right)\right](x_1)
	-\left[\mathbf{K}_j^i\star\tilde{\Theta}^{\varepsilon_1,\mathbf{e}_2}_{1,r}\star \left(D^{a}_{x_i}b_{tj}\right)\right](x_2);
	\end{multline}
	\begin{multline}
	E^{\text{appro}}:=\sum_{i=1}^{d}\sum_{j=1}^{m}\left[\mathbf{K}_j^i\star\tilde{\Theta}^{\varepsilon_1,\mathbf{e}_1}_{1,r}\star \left((\eta_{ti}^\varepsilon-\eta_{ti})\xi_{tj}|D^{s}b_{tj}|\right)\right](x_1)
	\\+\left[\mathbf{K}_j^i\star\tilde{\Theta}^{\varepsilon_1,\mathbf{e}_2}_{1,r}\star \left((\eta_{ti}^\varepsilon-\eta_{ti})\xi_{tj}|D^{s}b_{tj}|\right)\right](x_2);
	\end{multline}
	\begin{multline}
	E^{\text{diff-1}}:=\sum_{i=1}^{d}\sum_{j=1}^{m}\left[\mathbf{K}_j^i\star\tilde{\Theta}^{\varepsilon_1,\mathbf{e}_1}_{1,r}\star \left((\eta_{ti}^\varepsilon(x_1)-\eta_{ti}^\varepsilon)\xi_{tj}|D^{s}b_{tj}|\right)\right](x_1)
	\\+\left[\mathbf{K}_j^i\star\tilde{\Theta}^{\varepsilon_1,\mathbf{e}_2}_{1,r}\star \left((\eta_{ti}^\varepsilon(x_1)-\eta_{ti}^\varepsilon)\xi_{tj}|D^{s}b_{tj}|\right)\right](x_2);
	\end{multline}
	\begin{multline}
	E^{\text{diff-2}}:=\sum_{i=1}^{d}\sum_{j=1}^{m}(\eta_{ti}^\varepsilon(x_1)-\eta_{ti}^\varepsilon(x_2))\left[\mathbf{K}_j^i\star\tilde{\Theta}^{\varepsilon_1,\mathbf{e}_2}_{1,r}\star \left(\xi_{tj}|D^{s}b_{tj}|\right)\right](x_2).~~~~~~~
	\end{multline}
\end{enumerate}
 Then, we have the following identities:
\begin{lemma}\label{coro-Jab-Hu} There hold, 
	\begin{equation}
\mathbf{B}^i_{t}(x_1)-\mathbf{B}^i_{t}(x_2)= rA_{i1}^{\text{reg}}+ rA_{i1}^{\text{appro}}+rA_{i1}^{\text{diff-1}}+rA_{i1}^{\text{diff-2}}+r\varepsilon_1 A_{i2}+r(\mathbf{e}_1.\eta_t^\varepsilon(x_1))A_{i1}^{\text{sing}}, \label{diff-equaB}
\end{equation}
and \footnote[1]{Here $A_{1}^{\text{sing}}=(A_{11}^{\text{sing}},A_{21}^{\text{sing}},...,A_{d1}^{\text{sing}})$} 
\begin{align}
\langle \eta_t^\varepsilon(x_1), A_{1}^{\text{sing}}\rangle =E^{\text{reg}}+E^{\text{appro}}+	E^{\text{diff-1}}+E^{\text{diff-2}}	+\sum_{i=1,2}\tilde{\Theta}^{\varepsilon_1,\mathbf{e}_i}_{1,r}\star\left[\operatorname{div}(\mathbf{B}_t)\right](x_i),\label{div-equ}
\end{align}
\end{lemma}
\begin{proof} \textbf{Step 1.} By  Proposition \ref{Jabin-lem} with $\varepsilon=\varepsilon_1$ we have
	\begin{align*}	b_{tj}(x_1-z)-b_{tj}(x_2-z) &=r\tilde{\Theta}^{\varepsilon_1,\mathbf{e}_1}_{1,r}\star (\mathbf{e}_1.Db_{tj})(x_1-z)+\varepsilon_1r\tilde{\Theta}^{\varepsilon_1,\mathbf{e}_1}_{2,r}\star Db_{tj}(x_1-z)
	\\&-r\tilde{\Theta}^{\varepsilon_1,\mathbf{e}_2}_{1,r}\star (\mathbf{e}_2.Db_{tj})(x_2-z)-\varepsilon_1 r\tilde{\Theta}^{\varepsilon_1,\mathbf{e}_2}_{2,r}\star Db_{tj}(x_2-z),
	\end{align*} for any $z\in\mathbb{R}^d$. So, by \eqref{vectorfiB}, we get 
	\begin{align*}
		\mathbf{B}_{t}^i(x_1)-	\mathbf{B}_{t}^i(x_2)&=\sum_{j=1}^{m}\left(\mathbf{K}_j^i\star b_{tj}(x_1)-\mathbf{K}_j^i\star b_{tj}(x_2)\right)\\&=\sum_{j=1}^{m}r\left[\mathbf{K}_j^i\star\tilde{\Theta}^{\varepsilon_1,\mathbf{e}_1}_{1,r}\star (\mathbf{e}_1.Db_{tj})\right](x_1)+r\varepsilon_1\left[\mathbf{K}_j^i\star{\Theta}^{\varepsilon_1,\mathbf{e}_1}_{2,r}\star Db_{tj}\right](x_1)
		\\&-r\left[\mathbf{K}_j^i\star\tilde{\Theta}^{\varepsilon_1,\mathbf{e}_2}_{1,r}\star (\mathbf{e}_2.Db_{tj})\right](x_2)-r\varepsilon_1\left[\mathbf{K}_j^i\star\tilde{\Theta}^{\varepsilon_1,\mathbf{e}_2}_{2,r}\star Db_{tj}\right](x_2).
	\end{align*}
	Using  $Db_{tj}=D^ab_{tj}+\xi_{tj}\eta_{t}|D^{s}b_t|$ yields
	\begin{align*}
		\mathbf{B}_{t}^i(x_1)-	\mathbf{B}_{t}^i(x_2)&=rA_{i1}^{\text{reg}}+r\varepsilon_1 A_{i2}
	+\sum_{j=1}^{m}r\left[\mathbf{K}_j^i\star\tilde{\Theta}^{\varepsilon_1,\mathbf{e}_1}_{1,r}\star \left((\mathbf{e}_1.\eta_t)\xi_{tj}|D^{s}b_{tj}|\right)\right](x_1)
	\\&+r\left[\mathbf{K}_j^i\star\tilde{\Theta}^{\varepsilon_1,\mathbf{e}_2}_{1,r}\star \left((\mathbf{e}_1.\eta_t)\xi_{tj}|D^{s}b_{tj}|\right)\right](x_2).
	\end{align*}
	Since $\mathbf{e}_1.\eta_t=\mathbf{e}_1.(\eta_t-\eta_t^{\varepsilon})+\mathbf{e}_1.(\eta_t^{\varepsilon}-\eta_t^{\varepsilon}(x_1))+\mathbf{e}_1.\eta_t^{\varepsilon}(x_1)$, 
		\begin{align*}
	&	\mathbf{B}_{t}^i(x_1)-	\mathbf{B}_{t}^i(x_2)=rA_{i1}^{\text{reg}}+r\varepsilon_1 A_{i2}+ rA_{i1}^{\text{appro}}+rA_{i1}^{\text{diff-1}}+rA_{i1}^{\text{diff-2}}\\&+\sum_{j=1}^{m}r\left[\mathbf{K}_j^i\star\tilde{\Theta}^{\varepsilon_1,\mathbf{e}_1}_{1,r}\star \left((\mathbf{e}_1.\eta_t^\varepsilon(x_1))\xi_{tj}|D^{s}b_{tj}|\right)\right](x_1)
	+r\left[\mathbf{K}_j^i\star\tilde{\Theta}^{\varepsilon_1,\mathbf{e}_2}_{1,r}\star \left((\mathbf{e}_1.\eta_t^\varepsilon(x_1))\xi_{tj}|D^{s}b_{tj}|\right)\right](x_2)\\&= rA_{i1}^{\text{reg}}+ rA_{i1}^{\text{appro}}+rA_{i1}^{\text{diff-1}}+rA_{i1}^{\text{diff-2}}+r\varepsilon_1 A_{i2}+r(\mathbf{e}_1.\eta_t^\varepsilon(x_1))A_{i1}^{\text{sing}},
	\end{align*}
	which implies \eqref{diff-equaB}. \\
	\textbf{Step 2.} 
We have 
	\begin{align*}
&\langle \eta_t^\varepsilon(x_1), A_{1}^{\text{sing}}\rangle  =\sum_{i=1}^{d}\sum_{j=1}^{m}\left[\mathbf{K}_j^i\star\tilde{\Theta}^{\varepsilon_1,\mathbf{e}_1}_{1,r}\star \left(\eta_{ti}^\varepsilon(x_1)\xi_{tj}|D^{s}b_{tj}|\right)\right](x_1)
\\&~~~~+\left[\mathbf{K}_j^i\star\tilde{\Theta}^{\varepsilon_1,\mathbf{e}_2}_{1,r}\star \left(\eta_{ti}^\varepsilon(x_1)\xi_{tj}|D^{s}b_{tj}|\right)\right](x_2).
	\end{align*}
Since $\eta_{ti}^\varepsilon(x_1)=(\eta_{ti}^\varepsilon(x_1)-\eta_{ti}^\varepsilon)+(\eta_{ti}^\varepsilon-\eta_{ti})+\eta_{ti}$,  and $\eta_{ti}\xi_{tj}|D^{s}b_{tj}|=D^{s}_{x_i}b_{tj}=-D^{a}_{x_i}b_{tj}+D_{x_i}b_{tj}$, thus 
	\begin{align*}
&\langle \eta_t^\varepsilon(x_1), A_{1}^{\text{sing}}\rangle  = 		E^{\text{diff-1}}+	E^{\text{diff-2}}+	E^{\text{appro}}+E^{\text{reg}}\\&+
\sum_{i=1}^{d}\sum_{j=1}^{m}\left[\mathbf{K}_j^i\star\tilde{\Theta}^{\varepsilon_1,\mathbf{e}_1}_{1,r}\star D_{x_i}b_{tj}\right](x_1)
+\left[\mathbf{K}_j^i\star\tilde{\Theta}^{\varepsilon_1,\mathbf{e}_2}_{1,r}\star D_{x_i}b_{tj}\right](x_2).
\end{align*}
Using associative and commutativity
 properties of convolution and $\sum_{i=1}^{d}\sum_{j=1}^{m}\mathbf{K}_j^i\star D_{x_i}b_{tj}=\sum_{i=1}^{d}D_{x_i}\left(\sum_{j=1}^{m}\mathbf{K}_j^i\star b_{tj}\right)=\operatorname{div}(\mathbf{B}_t)$ yields  
	\begin{align*}
&\langle \eta_t^\varepsilon(x_1), A_{1}^{\text{sing}}\rangle  = 	E^{\text{reg}}+E^{\text{appro}}+	E^{\text{diff-1}}+E^{\text{diff-2}}\\&+\tilde{\Theta}^{\varepsilon_1,\mathbf{e}_1}_{1,r}\star\left[\sum_{i=1}^{d}\sum_{j=1}^{m}\mathbf{K}_j^i\star D_{x_i}b_{tj}\right](x_1)+\tilde{\Theta}^{\varepsilon_1,\mathbf{e}_2}_{1,r}\star\left[\sum_{i=1}^{d}\sum_{j=1}^{m}\mathbf{K}_j^i\star D_{x_i}b_{tj}\right](x_2)\\&=	E^{\text{reg}}+E^{\text{appro}}+	E^{\text{diff-1}}+E^{\text{diff-2}}+\tilde{\Theta}^{\varepsilon_1,\mathbf{e}_1}_{1,r}\star\left[\operatorname{div}(\mathbf{B}_t)\right](x_1)+\tilde{\Theta}^{\varepsilon_1,\mathbf{e}_2}_{1,r}\star\left[\operatorname{div}(\mathbf{B}_t)\right](x_2).
\end{align*}
This gives \eqref{div-equ}. The proof is complete.
\end{proof}
Lemma \ref{coro-Jab-Hu} implies that 
\begin{lemma} \label{le-mainthm2}We define for $\varepsilon_1\in(0,1/100)$
		\begin{align*}
&	\mathbf{T}^l_{\varepsilon_1,i,j}(\mu_l)(x)=\sup_{\rho\in (0,2R),e\in S^{d-1}}\frac{\varepsilon_1^{-d+1}}{\rho}\left|\left(\frac{1}{|.|^{d-1}}\Theta^{\varepsilon_1,e}_{l,\rho}(.)\right)\star \mathbf{K}_j^i\star\mu_l(x)\right|~~\forall~~x\in \mathbb{R}^d,
	\end{align*}
	with $\mu_2\in \mathcal{M}_b(\mathbb{R}^d,\mathbb{R}^d),\mu_1\in \mathcal{M}_b(\mathbb{R}^d)$ or $\mu_1\in \mathcal{M}_b(\mathbb{R}^d,\mathbb{R}^d)$.  There exists $\mathbf{P}_1(Db)(x,t),\mathbf{P}_2(Db)(x,t)\in L^1((0,T),L^{q_0}_{\loc}(\mathbb{R}^d)$ for some $q_0>1$ such that $||\mathbf{P}_1(Db)||_{L^1((0,T),L^{q_0}(B_R(0)))}+||\mathbf{P}_2(Db)||_{L^1((0,T),L^{q_0}(B_R(0)))}\leq C(R,\varepsilon_1,\varepsilon)||b||_{L^1((0,T),BV(\mathbb{R}^d))}$ for any $R>0$ and 
	 for any $x_1\not=x_2\in B_\lambda$, we have 
	\begin{multline}
A_1:=\frac{\left|\langle x_1-x_2,\mathbf{B}_{t}(x_1)-\mathbf{B}_{t}(x_2)\rangle \right|}{\delta^2+|x_1-x_2|^2+\gamma \langle \eta_t^\varepsilon(x_1),x_1-x_2\rangle ^2}\\\leq \sum_{l,i,j}~ \frac{\mathbf{P}_1(Db)(x_l,t)}{\delta}\wedge \mathbf{T}^{1}_{\varepsilon_1,i,j}(D^ab_j)(x_l)+ \sum_{l,i,j}~\frac{\mathbf{P}_1(Db)(x_l,t)}{\delta}	\wedge \mathbf{T}^{1}_{\varepsilon_1,i,j}(\omega_{tij}^\varepsilon)(x_l)+ \mathbf{P}_1(Db)(x_l,t)\\+\varepsilon_1\sum_{l,i,j}~\frac{\mathbf{P}_1(Db)(x_l,t)}{\delta}\wedge \mathbf{T}^{2}_{\varepsilon_1,i,j}(Db_{tj})(x_l) +\gamma^{-1/2}\sum_{l,i,j}~\frac{\mathbf{P}_1(Db)(x_l,t)}{\delta}\wedge \mathbf{T}^{1}_{\varepsilon_1,i,j}(\xi_{tj}|D^{s}b_{tj}|)(x_l)\label{A1-sec-thm};
	\end{multline}
	\begin{multline}
	A_2:=\frac{\gamma\langle \eta_t^\varepsilon(x_1),x_1-x_2\rangle \langle \eta_t^\varepsilon(x_1),\mathbf{B}_{t}(x_1)-\mathbf{B}_{t}(x_2)\rangle}{\delta^2+|x_1-x_2|^2+\gamma \langle \eta_t^\varepsilon(x_1),x_1-x_2\rangle ^2}\\\leq 2\gamma^{1/2}\sum_{l,i,j}~ \frac{\mathbf{P}_2(Db)(x_l,t)}{\delta}\wedge \mathbf{T}^{1}_{\varepsilon_1,i,j}(D^ab_j)(x_l)+ 2\gamma^{1/2}\sum_{l,i,j}~\frac{\mathbf{P}_2(Db)(x_l,t)}{\delta}	\wedge \mathbf{T}^{1}_{\varepsilon_1,i,j}(\omega_{tij}^\varepsilon)(x_l)\\~~~~~~~+ \gamma^{1/2}\sum_{l} \mathbf{P}_2(Db)(x_l,t)+\gamma^{1/2}\varepsilon_1\sum_{l,i,j}~\frac{\mathbf{P}_2(Db)(x_l,t)}{\delta}\wedge \mathbf{T}^{2}_{\varepsilon_1,i,j}(Db_{tj})(x_l)\\+ C(\varepsilon_1,\gamma)~ \sum_{l=1}^{2}~\frac{\mathbf{I}_1(\mathbf{1}_{B_{4\lambda}}(\operatorname{div}^a(B_t))^+)(x_l)}{\delta}\wedge \mathbf{M}(\mathbf{1}_{B_{4\lambda}}(\operatorname{div}^a(B_t))^+)(x_l); \label{A2-sec-thm}
	\end{multline}
	where $\sum_{l,i,j}:=\sum_{l=1}^{2}\sum_{i=1}^{d}\sum_{j=1}^{m}$, $\omega_{tij}^\varepsilon:=(\eta_t-\eta_t^{\varepsilon})\xi_{tj}|D^{s}b_{tj}|$.
\end{lemma}

\begin{proof} Set 	\begin{align*}
	\mathbf{T}^{l,1}_{\varepsilon_1,i,j}(\mu_l)(x)=\sup_{\rho\in (0,2R),e\in S^{d-1}}\varepsilon_1^{-d+1}\left|\left(\frac{1}{|.|^{d-1}}\Theta^{\varepsilon_1,e}_{l,\rho}(.)\right)\star \mathbf{K}_j^i\star\mu_l(x)\right|~~\forall~~x\in \mathbb{R}^d.
	\end{align*}
	1. Thanks to \eqref{diff-equaB}, we obtain that  
	\begin{align}\nonumber
	&A_1=\frac{r\langle \mathbf{e}_1,\mathbf{B}_{t}(x_1)-\mathbf{B}_{t}(x_2)\rangle }{\delta^2+r^2+\gamma r^2( \mathbf{e}_1.\eta_t^\varepsilon(x_1))^2}\\&\nonumber=\frac{r^2\langle \mathbf{e}_1,A_{1}^{\text{reg}}\rangle }{\delta^2+r^2+\gamma r^2( \mathbf{e}_1.\eta_t^\varepsilon(x_1))^2}+\frac{r^2\langle \mathbf{e}_1,A_{1}^{\text{appro}}\rangle }{\delta^2+r^2+\gamma r^2( \mathbf{e}_1.\eta_t^\varepsilon(x_1))^2}+\frac{r^2\langle \mathbf{e}_1,A_{1}^{\text{diff-1}}\rangle }{\delta^2+r^2+\gamma r^2( \mathbf{e}_1.\eta_t^\varepsilon(x_1))^2}\\&+\frac{r^2\langle \mathbf{e}_1,A_{1}^{\text{diff-2}}\rangle }{\delta^2+r^2+\gamma r^2( \mathbf{e}_1.\eta_t^\varepsilon(x_1))^2}+\frac{r^2\varepsilon_1\langle \mathbf{e}_1,A_{2}\rangle }{\delta^2+r^2+\gamma r^2( \mathbf{e}_1.\eta_t^\varepsilon(x_1))^2}+\frac{r^2(\mathbf{e}_1.\eta_t^\varepsilon(x_1))\langle \mathbf{e}_1,A_{1}^{\text{sing}} \rangle }{\delta^2+r^2+\gamma r^2( \mathbf{e}_1.\eta_t^\varepsilon(x_1))^2}\nonumber\\&= (1)+(2)+(3)+(4)+(5)+(6),\label{le1-secondmainthm}
	\end{align}
	with $r=|x_1-x_2|$. By definition of $\mathbf{T}^l_{\varepsilon_1,i,j}, \mathbf{T}^{l,1}_{\varepsilon_1,i,j}$, we can estimate that
	\begin{enumerate} \abovedisplayskip=0pt \belowdisplayskip=0pt
		\item []
		\begin{multline*}
		|(1)|\leq \sum_{l,i,j}~ \frac{\mathbf{T}^{1,1}_{\varepsilon_1,i,j}(D^ab_j)(x_l)}{\delta}\wedge \mathbf{T}^{1}_{\varepsilon_1,i,j}(D^ab_j)(x_l);~~~~~~~~~~~~~~~~~~
		\end{multline*}
			\begin{multline*}
	|(2)|\leq \sum_{l,i,j}~\frac{\mathbf{T}^{1,1}_{\varepsilon_1,i,j}((\eta_t-\eta_t^{\varepsilon})\xi_{tj}|D^{s}b_{tj}|)(x_l)}{\delta}	\wedge \mathbf{T}^{1}_{\varepsilon_1,i,j}((\eta_t-\eta_t^{\varepsilon})\xi_{tj}|D^{s}b_{tj}|)(x_l);
		\end{multline*}
			\begin{multline*}
		|(3)|\leq \sum_{l,i,j} \mathbf{T}^{1}_{\varepsilon_1,i,j}((\eta_t^{\varepsilon}-\eta_t^{\varepsilon}(x_l))\xi_{tj}|D^{s}b_{tj}|)(x_l)~~~~~~~~~~~~~~~~~~~
		\end{multline*}
			\begin{multline*}
		|(4)|\leq ||\nabla \eta_t^\varepsilon||_{L^\infty(\mathbb{R}^d)}\sum_{l,i,j} \mathbf{T}^{1,1}_{\varepsilon_1,i,j}(\xi_{tj}|D^{s}b_{tj}|)(x_l)~~~~~~~~~~~~~~~~~~~~
		\end{multline*}
		\begin{multline*}
		|(5)|\leq \varepsilon_1\sum_{l,i,j}~\frac{\mathbf{T}^{2,1}_{\varepsilon_1,i,j}(Db_{tj})(x_l)}{\delta}\wedge \mathbf{T}^{2}_{\varepsilon_1,i,j}(Db_{tj})(x_l);~~~~~~~~~~~~~~~~~~
		\end{multline*}
		\begin{multline*}
		|(6)|\leq \gamma^{-1/2}\sum_{l,i,j}~\frac{\mathbf{T}^{1,1}_{\varepsilon_1,i,j}(\xi_{tj}|D^{s}b_{tj}|)(x_l)}{\delta}\wedge \mathbf{T}^{1}_{\varepsilon_1,i,j}(\xi_{tj}|D^{s}b_{tj}|)(x_l).~~~~~
		\end{multline*}
	\end{enumerate}
Set 
\begin{align*}
\mathbf{P}_1(Db)(x,t)&=\sum_{i,j}\mathbf{T}^{1,1}_{\varepsilon_1,i,j}(D^ab_j)(x)+\mathbf{T}^{1,1}_{\varepsilon_1,i,j}((\eta_t-\eta_t^{\varepsilon})\xi_{tj}|D^{s}b_{tj}|)(x)+\mathbf{T}^{1}_{\varepsilon_1,i,j}((\eta_t^{\varepsilon}-\eta_t^{\varepsilon}(x))\xi_{tj}|D^{s}b_{tj}|)(x)\\&+||\nabla \eta_t^\varepsilon||_{L^\infty(\mathbb{R}^d)} \mathbf{T}^{1,1}_{\varepsilon_1,i,j}(\xi_{tj}|D^{s}b_{tj}|)(x)+\mathbf{T}^{2,1}_{\varepsilon_1,i,j}(Db_{tj})(x)+\mathbf{T}^{1,1}_{\varepsilon_1,i,j}(\xi_{tj}|D^{s}b_{tj}|)(x).
\end{align*}
By Lemma \ref{Remar0} and Remark \ref{Remar1}, there exists $q_0>1$ such that $
||\mathbf{P}_1(Db)||_{L^1((0,T),L^{q_0}(B_R(0)))}\leq C(R,\varepsilon_1,\varepsilon)||b||_{L^1((0,T),BV(\mathbb{R}^d))}$ for any $R>0.$
Combining these with \eqref{le1-secondmainthm} yields \eqref{A1-sec-thm}.\\
2.  Again, thanks to \eqref{diff-equaB}  we obtain that  
	\begin{align}\nonumber
&A_2=\frac{\gamma r (\mathbf{e}_1.\eta_t^\varepsilon(x_1)) \langle \eta_t^\varepsilon(x_1),\mathbf{B}_{t}(x_1)-\mathbf{B}_{t}(x_2)\rangle }{\delta^2+r^2+\gamma r^2( \mathbf{e}_1.\eta_t^\varepsilon(x_1))^2}\\&\nonumber=\frac{\gamma r^2(\mathbf{e}_1.\eta_t^\varepsilon(x_1))\langle \eta_t^\varepsilon(x_1),A_{1}^{\text{reg}}\rangle }{\delta^2+r^2+\gamma r^2( \mathbf{e}_1.\eta_t^\varepsilon(x_1))^2}+\frac{\gamma r^2(\mathbf{e}_1.\eta_t^\varepsilon(x_1))\langle \eta_t^\varepsilon(x_1),A_{1}^{\text{appro}}\rangle }{\delta^2+r^2+\gamma r^2( \mathbf{e}_1.\eta_t^\varepsilon(x_1))^2}+\frac{\gamma r^2(\mathbf{e}_1.\eta_t^\varepsilon(x_1)) \langle \eta_t^\varepsilon(x_1),A_{1}^{\text{diff-1}}\rangle }{\delta^2+r^2+\gamma r^2( \mathbf{e}_1.\eta_t^\varepsilon(x_1))^2}\\&+\frac{\gamma r^2(\mathbf{e}_1.\eta_t^\varepsilon(x_1))\langle \eta_t^\varepsilon(x_1),A_{1}^{\text{diff-2}}\rangle }{\delta^2+r^2+\gamma r^2( \mathbf{e}_1.\eta_t^\varepsilon(x_1))^2}+\frac{\gamma r^2\varepsilon_1 (\mathbf{e}_1.\eta_t^\varepsilon(x_1))\langle \eta_t^\varepsilon(x_1),A_{2}\rangle }{\delta^2+r^2+\gamma r^2( \mathbf{e}_1.\eta_t^\varepsilon(x_1))^2}+\frac{\gamma r^2(\mathbf{e}_1.\eta_t^\varepsilon(x_1))^2\langle \eta_t^\varepsilon(x_1),A_{1}^{\text{sing}} \rangle }{\delta^2+r^2+\gamma r^2( \mathbf{e}_1.\eta_t^\varepsilon(x_1))^2}\nonumber\\&= (7)+(8)+(9)+(10)+(11)+\frac{\gamma r^2(\mathbf{e}_1.\eta_t^\varepsilon(x_1))^2\langle \eta_t^\varepsilon(x_1),A_{1}^{\text{sing}} \rangle }{\delta^2+r^2+\gamma r^2( \mathbf{e}_1.\eta_t^\varepsilon(x_1))^2}.\label{le2-secondmainthm}
\end{align}
Plugging \eqref{div-equ} into \eqref{le2-secondmainthm} gives
\begin{align}\nonumber
&\nonumber A_2
=(7)+(8)+(9)+(10)+(11)+\frac{\gamma r^2(\mathbf{e}_1.\eta_t^\varepsilon(x_1))^2E^{\text{reg}} }{\delta^2+r^2+\gamma r^2( \mathbf{e}_1.\eta_t^\varepsilon(x_1))^2}+\frac{\gamma r^2(\mathbf{e}_1.\eta_t^\varepsilon(x_1))^2E^{\text{appro}} }{\delta^2+r^2+\gamma r^2( \mathbf{e}_1.\eta_t^\varepsilon(x_1))^2}\\&+\frac{\gamma r^2(\mathbf{e}_1.\eta_t^\varepsilon(x_1))^2E^{\text{diff-1}} }{\delta^2+r^2+\gamma r^2( \mathbf{e}_1.\eta_t^\varepsilon(x_1))^2}+\frac{\gamma r^2(\mathbf{e}_1.\eta_t^\varepsilon(x_1))^2E^{\text{diff-2}} }{\delta^2+r^2+\gamma r^2( \mathbf{e}_1.\eta_t^\varepsilon(x_1))^2}\nonumber\\&+\frac{\gamma r^2(\mathbf{e}_1.\eta_t^\varepsilon(x_1))^2\left[\tilde{\Theta}^{\varepsilon_1,\mathbf{e}_1}_{1,r}\star\left[\operatorname{div}(\mathbf{B}_t)\right](x_1)+\tilde{\Theta}^{\varepsilon_1,\mathbf{e}_2}_{1,r}\star\left[\operatorname{div}(\mathbf{B}_t)\right](x_2)\right] }{\delta^2+r^2+\gamma r^2( \mathbf{e}_1.\eta_t^\varepsilon(x_1))^2}\nonumber\\&=(7)+(8)+(9)+(10)+(11)+(12)+(13)+(14)+(15)+(16).\label{le3-secondmainthm}
\end{align}
As above, there exists $\mathbf{P}_2(Db)(x,t)\in L^1((0,T),L^{q_0}_{\loc}(\mathbb{R}^d)$ for $q_0>1$ such that $||\mathbf{P}_2(Db)||_{L^1((0,T),L^{q_0}(B_R(0)))}\leq C(R,\varepsilon_1,\varepsilon)||b||_{L^1((0,T),BV(\mathbb{R}^d))}$ for any $R>0$ and  \begin{enumerate} \abovedisplayskip=0pt \belowdisplayskip=0pt
	\item []
	\begin{multline*}
	|(7)|+|(12)|\leq 2\gamma^{1/2}\sum_{l,i,j}~ \frac{\mathbf{P}_2(Db)(x_l,t)}{\delta}\wedge \mathbf{T}^{1}_{\varepsilon_1,i,j}(D^ab_j)(x_l);~~~~~~~~~~~~~~~~~~
	\end{multline*}
		\begin{multline*}
	|(8)| +|(13)\leq 2\gamma^{1/2} \sum_{l,i,j}~\frac{\mathbf{P}_2(Db)(x_l,t)}{\delta}	\wedge \mathbf{T}^{1}_{\varepsilon_1,i,j}((\eta_t-\eta_t^{\varepsilon})\xi_{tj}|D^{s}b_{tj}|)(x_l);
	\end{multline*}
		\begin{multline*}
	|(9)|+|(10)|+|(14)|+|(15)|\leq \gamma^{1/2} \sum_{l}\mathbf{P}_2(Db)(x_l,t);~~~~~~~~~~~~~
	\end{multline*}
		\begin{multline*}
	|(11)|\leq \varepsilon_1\gamma^{1/2}\sum_{l,i,j}~\frac{\mathbf{P}_2(Db)(x_l,t)}{\delta}\wedge \mathbf{T}^{2}_{\varepsilon_1,i,j}(Db_{tj})(x_l).~~~~~~~~~~~~~~~~~~
	\end{multline*}
\end{enumerate}
For $(16)$, thanks to $\Theta^{\varepsilon,e}_1\geq 0$ and $\operatorname{div}^s(B_t)\leq 0$ we can estimate 
\begin{align*}
(16)&\leq \frac{\gamma r^2(\mathbf{e}_1.\eta_t^\varepsilon(x_1))^2\left[\tilde{\Theta}^{\varepsilon_1,\mathbf{e}_1}_{1,r}\star\left[\operatorname(\operatorname{div}^a(\mathbf{B}_t))^+\right](x_1)+\tilde{\Theta}^{\varepsilon_1,\mathbf{e}_2}_{1,r}\star\left[(\operatorname{div}^a(\mathbf{B}_t))^+\right](x_2)\right] }{\delta^2+r^2+\gamma r^2( \mathbf{e}_1.\eta_t^\varepsilon(x_1))^2}\\ &\leq C(\varepsilon_1,\gamma)~ \sum_{l=1}^{2}~\frac{\mathbf{I}_1(\mathbf{1}_{B_{4R}}(\operatorname{div}^a(\mathbf{B}_t))^+)(x_l)}{\delta}\wedge \mathbf{M}(\mathbf{1}_{B_{4R}}(\operatorname{div}^a(\mathbf{B}_t))^+)(x_l).
\end{align*}
Combining above inequalities together yields \eqref{A2-sec-thm}. The proof is complete.
\end{proof}
\begin{remark} Our estimates in Theorem \ref{mainthm2} can be implied a partial result of Bressan's conjecture (in \cite{Bressan2})~for mixing flows. This will
	be pursued in our forthcoming work.
\end{remark}
\begin{remark} In this remark, we would like to discuss another conjecture of Bressan. Let $B_n\in C^1_b((0,\infty)\times\mathbb{R}^d)$ be such that $
	||B_n||_{L^1\cap L^\infty}+||D B_n||_{L^1}\leq C~~\forall~~n.$\\
	\textit{Bressan's compactness conjecture:} If $X_n$ solves $
	\frac{d}{dt}X_n(t,x)=B_n(t,X_n(t,x)), ~~X_n(0)=id~~\mathbb{R}^{d}\times (0,\infty)$
	and satisfies 
	\begin{align}\label{cond1}
	C_1\leq JX_n(t,x)\leq C_2~~\forall~~n.
	\end{align}
	Then, $X_n$ is locally compact in $L^1((0,\infty)\times \mathbb{R}^d)$. This conjecture  was proven in \cite{bianBoni2} via the well posedness of continuity equations in the class of nearly incompressible $BV$ vector fields. 
	Note that \eqref{cond1} implies \begin{align}\label{es-div}
	\sup_{t_1,t_2}|\int_{t_1}^{t_2} \operatorname{div}(B^n)(t,X_n(t,x)) dt|\leq C~~\forall~~n.
	\end{align} 
	We \textit{hope} that thanks to our estimates  in Theorem \ref{mainthm2} and assumption \eqref{es-div}, we can obtain this conjecture. 
\end{remark}
\section{Well posedness  of Regular Lagrangian flows and Transport, \\Continuity equations}
\textbf{5.1} \textit{Well posedness  of Regular Lagrangian flows}: The following results are obtained from Theorem \ref{mainthm2}, Corollary \ref{maincorollary} and  Lemma \ref{decay-lem}. Theirs proof are very similar to proofs in Section 6 and 7 in \cite{BoCrip}.
\begin{proposition}(Uniqueness) \label{uniqueness}Let $\mathbf{B}$ be a vector field as in Corollary \ref{maincorollary}  satisfying assumption $\mathbf{(R_1)}$.  Assume that $\operatorname{div}(\mathbf{B})\in L^1((0,T),\mathcal{M}_{\loc}(\mathbb{R}^d))$, $(\operatorname{div}(\mathbf{B}))^+\in L^1((0,T),L^1_{\loc}(\mathbb{R}^d))$.  If there exist the regular Lagrangian flows $X_1,X_2$ associated to $\mathbf{B}$ starting at time $t$, then we have $X_1\equiv X_2$. 
\end{proposition}
\begin{proposition}(Stability) Let $\mathbf{B}_n$ be a sequence of vector fields satisfying assumption $\mathbf{(R_1)}$ converging in $L^1_{\loc}([0,T]\times\mathbb{R}^d)$ to a vector field $\mathbf{B}$ which satisfies the assumptions of  $\mathbf{B}$ in Proposition \ref{uniqueness}. Assume that there exist $X_n$ and $X$ regular Lagrangian flows starting at time $t$ associated $\mathbf{B}_n$ and $\mathbf{B}$ resp. and denote by $L_n$ and $L$ the compression constants of the flows. Assume that for some decomposition $
	\frac{\mathbf{B}_{n}}{1+|x|}=\tilde{B}_{n,1}+\tilde{B}_{n,2}$
	as in assumption $\mathbf{(R_1)}$, we have $
	L_n+||\tilde{B}_{n,1}||_{L^1((0,T),L^1(\mathbb{R}^d)}+||\tilde{B}_{n,2}||_{L^1((0,T),L^\infty(\mathbb{R}^d)}\leq C$  $\forall~n\in \mathbb{N},$
	for some constant $C>0$.  Then,  for any compact set $K$, 
	\begin{align}
	\lim\limits_{n\to \infty}\sup_{s\in [t,T]}\int_K |X_n(s,x)-X(s,x)|\wedge 1dx=0.
	\end{align}
\end{proposition}
\begin{proposition}(Compactness) Let $\mathbf{B}_n\in C_b^1([0,T]\times\mathbb{R}^d,\mathbb{R}^d)$  converge  in $L^1_{\loc}([0,T]\times\mathbb{R}^d)$ to a vector field $\mathbf{B}$ which satisfies  the assumptions of  $\mathbf{B}$  in Proposition \ref{uniqueness}. 
Let $X_n$  be the flow starting at time $t$ associated $\mathbf{B}_n$ and denote by $L_n$  the compression constants of the flow. Assume that for some decomposition $
	\frac{\mathbf{B}_{n}}{1+|x|}=\tilde{B}_{n,1}+\tilde{B}_{n,2}$
	as in assumption $\mathbf{(R_1)}$, we have $
	L_n+||\tilde{B}_{n,1}||_{L^1((0,T),L^1(\mathbb{R}^d)}+||\tilde{B}_{n,2}||_{L^1((0,T),L^\infty(\mathbb{R}^d)}\leq C$  $\forall~n\in \mathbb{N},$
	for some constant $C>0$.  Then, there exists  a regular Lagrangian flow $X$ starting at time $t$ associated to $B$ such that  for any compact set $K$, 
	\begin{align}
	\lim\limits_{n\to \infty}\sup_{s\in [t,T]}\int_K |X_n(s,x)-X(s,x)|\wedge 1 dx=0 .
	\end{align}
\end{proposition}
\begin{proposition}(Existence)\label{exist} Let $\mathbf{B}$ be as in Proposition \ref{uniqueness}. Assume that  $\operatorname{div}(\mathbf{B})\geq a(t)$ in $(0,T)\times\mathbb{R}^d$ with $a\in L^1((0,T))$. Then, for all $t\in [0,T)$ there exists a regular Lagrangian flow $X:=X(.,t,.)$ associated to $\mathbf{B}$ starting at time $t$. Moreover, the flow $X$ satisfies $X\in C(D_T;L^0_{\loc}(\mathbb{R}^d))\cap \mathcal{B}(D_T;\log L_{\loc}(\mathbb{R}^d))$
	where $D_T=\{(s,t):0\leq t\leq s\leq T\}$ and for every $0\leq t\leq \tau\leq s\leq T$, there holds $X(s,\tau,X(\tau,t,x))=X(s,t,x)$ for $\mathcal{L}^d$-a.e $x\in \mathbb{R}^d$.
\end{proposition}
In the previous proposition we assume the condition $\operatorname{div}(\mathbf{B})\geq a(t)$ in order to be sure to have a smooth approximating sequence with equi-bounded compression constants.
\begin{proposition}(Properties of the Jacobian)\label{jaco-pro} Let $\mathbf{B}$ be as in Proposition \eqref{exist}, $X:X(.,t,.)$ the regular Lagrangian flow associated to $\mathbf{B}$  starting at time $t$. Assume that $\operatorname{div}(\mathbf{B})\in L^1((0,T),L^\infty(\mathbb{R}^d))$. Then, the function $
	JX(s,t,x)=\exp\left(\int_{t}^{s}\operatorname{div}(\mathbf{B})(\tau,X(\tau,t,x))d\tau\right)$
satisfies 
\begin{align}\label{jacobian1}
\int_{\mathbb{R}^d}\phi(x)dx=\int_{\mathbb{R}^d} \phi(X(s,t,x)) JX(s,t,x)dx~~\forall~~\phi \in L^1(\mathbb{R}^d)
\end{align}
and  $
\partial_s JX(s,t,x)=JX(s,t,x) \operatorname{div}(\mathbf{B})(\tau,X(s,t,x))$ for all $s\in (t,T).$ Moreover,  $\exp(-L)\leq JX(s,t,x)\leq\exp(L)$ with $L=||\operatorname{div}(\mathbf{B})||_{L^1((0,T),L^\infty(\mathbb{R}^d))}$ and $JX\in C(D_T;L^\infty(\mathbb{R}^d)-w\star)\cap C(D_T;L^1_{\loc }(\mathbb{R}^d))$ where $D_T=\{(s,t):0\leq t\leq s\leq T\}$. Besides, for any $0\leq t\leq s\leq T$, $X^{-1}(t,s,.)(x)$ exists almost everywhere $x\in \mathbb{R}^d$.  The function $JX$ is called the Jacobian of the flow $X$. 
\end{proposition}
\textbf{5.2} \textit{Well posedness of Transport and continuity equations}:
Next, we will connect the Regular Lagrangian flows to the transport and continuity equations. We first recall definition of  renormalized solution of \eqref{contieq}, it was first introduced in \cite{Diperlions}.
\begin{definition} Let $u_0:\mathbb{R}^d\to \mathbb{R}$ be a measurable function, let $\mathbf{B}\in L^1_{\loc}((0,T)\times \mathbb{R}^d;\mathbb{R}^d)$ be a vector field such that $\operatorname{div}(\mathbf{B})\in L^1_{\loc}((0,T)\times\mathbb{R}^d)$ and let  $G,F\in L^1_{\loc}((0,T)\times\mathbb{R}^d)$. A measure function $u:[0,T]\times\mathbb{R}^d\to \mathbb{R}$ is a renormalized solution of \eqref{contieq} if for every function $\beta:\mathbb{R}\to\mathbb{R}$ satisfying $\beta\in C^1_b(\mathbb{R})$ and $\beta'(z)z\in L^\infty(\mathbb{R})$, $\beta(0)=0$ we have that 
	\begin{align*}
	\partial_t\beta(u)+\operatorname{div}(\mathbf{B}\beta(u))+\operatorname{div}(\mathbf{B})\left(u\beta'(u)-\beta(u)\right)=G u\beta'(u)+F\beta'(u)
	\end{align*}
and $\beta(u)(t=0)=\beta(u_0)$	in the sense of distributions.
\end{definition}
We have the following proposition:
\begin{proposition}  Let $\mathbf{B}$ be as in Proposition \eqref{exist}, $X$ be the regular Lagrangian flow associated to $\mathbf{B}$  starting at time $0$  in Proposition \eqref{exist}. Assume that $\operatorname{div}(\mathbf{B})\in L^1((0,T),L^\infty(\mathbb{R}^d))$. Let  $G,F\in L^1((0,T)\times\mathbb{R}^d)$  and let $u_0:\mathbb{R}^d\to \mathbb{R}$ be a measurable function.  Then,  there exists a unique renormalized solution $u:[0,T]\times\mathbb{R}^d\to \mathbb{R}$ of \eqref{contieq} starting from $u_0$. Furthermore, for any $(t,x)\in [0,T]\times\mathbb{R}^d$ we have
	\begin{align}\nonumber
	u(t,x)&=\frac{u_0(\overline{x})}{JX(t,\overline{x})} \exp\left(\int_{0}^{t}G(s,X(s,\overline{x}))ds\right)\\&+\frac{1}{JX(t,\overline{x})}\int_{0}^{t}f(\tau,X(\tau,\overline{x})) \exp\left(\int_{\tau}^{t}G(s,X(s,\overline{x}))ds\right) JX(\tau,\overline{x})d\tau,
	\end{align}
	with $\overline{x}=X^{-1}(t,.)(x)$, $JX(t,\overline{x}):=JX(t,0,\overline{x})$.
\end{proposition}
Proof of previous proposition is very similar to \cite{comcrip}[Proof of Theorem 2.7]. It is left to the reader.

%
\section{Appendix}
\begin{proof}[Proof of Proposition \ref{exampleDipernaLions}]  1. By \cite{Diperlions},  there exist two different regular Lagrangian flows $X_1,X_2$  associated to the following vector field
	\begin{align*}
	\mathbf{B}(x)=\left(-\operatorname{sign}(x_2)\left[\frac{x_1}{|x_2|^2}\mathbf{1}_{|x_1|\leq |x_2|}+\mathbf{1}_{|x_1|>|x_2|}\right],-\left[\frac{1}{|x_2|}\mathbf{1}_{|x_1|\leq |x_2|}+\mathbf{1}_{|x_1|>|x_2|}\right]\right)
	\end{align*} such that  for any $x\in \mathbb{R}^2$, $X_1,X_2\in W^{1,p}_{\text{loc}}(\mathbb{R}^2)$ for any $1<p<2$, 
	$X_1,X_2\in L_{\loc}^\infty(\mathbb{R}^2;C(\mathbb{R}))\cap C(\mathbb{R}^2;L^q_{\loc}(\mathbb{R}))$ for any $q<\infty$ and 
	$X(t,.)_{\#}\mathcal{L}^d =\mathcal{L}^d$ for any $t\in [0,T]$, $X_j(t+s,.)=X_j(t,X(s,.))$ a.e on $\mathbb{R}^2$, for all $t,s\in \mathbb{R}^d$.  
	Clearly, $\frac{|\mathbf{B}(x)|}{|x|+1}\in L^1(\mathbb{R}^2)+L^\infty(\mathbb{R}^2)$ and $\operatorname{div}(\mathbf{B})=0$. \\
	2. Therefore, it is enough to show that  there exist  functions $\Omega_1,...,\Omega_m\in \left(L^\infty\cap BV\right)(S^1) $ such that $\Omega_l(\theta)=\Omega_l(t\theta)$ for $\theta\in S^1,t>0$, $\int_{S^1}\Omega_j=0$ and  for any $R>1$ we have 
	\begin{align}
	\partial_l \mathbf{B}^i=\sum_{j=1}^{m}\left(\frac{\Omega_j^i(.)}{|.|^2}\right)\star\mu_{jR}^{l}~~\text{in}~~\mathcal{D}'(B_R)
	\end{align}
	for some $\mu_{jR}^{l}\in \mathcal{M}_b(\mathbb{R}^2)$ $l=1,2$ and $j=1,...,m$.\\
	Let $K_1(x_1,x_2)=c\frac{x_2^2-x_1^2}{(x_1^2+x_2^2)^2}$, $K_2(x_1,x_2)=c\frac{x_1^2-x_2^2}{(x_1^2+x_2^2)^2}$ be kernels of operators $\mathcal{R}_1^2,\mathcal{R}_2^2$, where $\mathcal{R}_1,\mathcal{R}_2$ are the Riesz transforms in $\mathbb{R}^2$. Let $\chi\in C^\infty_c([0,\infty))$ be such that $\chi=1$ in $[0,2)$ and $\chi=0$ in $(4,\infty)$ and set $\chi_r(x)=\chi(\frac{|x|}{r})$. Put $b_0(x_1,x_2)= -\mathbf{1}_{|y_1|> |y_2|}\in BV_{\loc}(\mathbb{R}^2)$. Fix $R>100$, we have 
	\begin{align}\nonumber
	\partial_1\mathbf{B}^1(x)&=-\frac{\operatorname{sign}(x_2)}{|x_2|}\mathbf{1}_{|x_1|\leq |x_2|}+\chi_{8R}\frac{\operatorname{sign}(x_2)|x_1|}{|x_2|^2}d\delta_{|x_1|=|x_2|}(x_1)d\mathcal{L}^1(x_2)+\chi_R\partial_1b_0(x_1,x_2)\\& = \mathbf{K}_1\star\delta_{0}(x)+\sum_{j=1,2}\mathcal{R}_j^2(\chi_R\partial_1b_0)(x)+\sum_{j=1,2}\mathcal{R}_j^2(\chi_{8R}\nu)(x)~\text{in}~\mathcal{D}'(B_{R}).\label{deriB11}
	\end{align}
	where $\mathbf{K}_1(y)=\frac{\Omega_1(y)}{|y|^2}$, $\Omega_1(y_1,y_2)=-\frac{|y|}{y_2}\mathbf{1}_{|y_1|\leq |y_2|}\in \left(L^\infty\cap BV\right)(B_2(0)\backslash B_1(0))$ with $\int_{S^1}\Omega_1=0$, $\chi_R\partial_1b_0\in \mathcal{M}_b(\mathbb{R}^2)$ and $\nu(x_1,x_2)=\frac{\operatorname{sign}(x_2)|x_1|}{|x_2|^2}d\delta_{|x_1|=|x_2|}(x_1)d\mathcal{L}^1(x_2)$ in $\mathcal{D}'(B_{2R})$.\\
	Since $\chi_{8R}=(\chi_{8R}-\chi_{|x|})+(\chi_{|x|}-\chi_{|x|/8})+\chi_{|x|/8}$,
	\begin{align}\nonumber
	\sum_{j=1,2}\mathcal{R}_j^2(\chi_{8R}\nu)(x)&=\sum_{j=1,2}\chi_R(x)\mathcal{R}_j^2(\chi_{|x|/8}\nu)(x)+\sum_{j=1,2}\chi_R(x)\mathcal{R}_j^2((\chi_{8R}-\chi_{|x|})\nu)(x)\\&+\sum_{j=1,2}\mathcal{R}_j^2((\chi_{|x|}-\chi_{|x|/2})\nu)(x)~\text{in}~\mathcal{D}'(B_{R}),\label{deriB12}
	\end{align}
	and
	\begin{align}\label{deriB13}
	\chi_R(x)\mathcal{R}_j^2(\chi_{|x|/8}\nu)(x), \chi_R(x)\mathcal{R}_j^2((\chi_{8R}-\chi_{|x|})\nu)(x)\in (L^\infty\cap L^1)(\mathbb{R}^2) ~j=1,2.
	\end{align}
	We now show that there exists $\tilde{\Omega}_j\in \left(L^\infty\cap BV\right)(B_2(0)\backslash B_1(0))$ such that $\tilde{\Omega}_j(\theta)=\tilde{\Omega}_j(r\theta)$ for any $r>0,\theta\in S^1$, $\int_{S^1}\tilde{\Omega}_j(\theta)=0$ and 
	\begin{align}
	\mathcal{R}_j^2((\chi_{|x|}-\chi_{|x|/2})\nu)(x)=\frac{\tilde{\Omega}_j(x)}{|x|^2}=\frac{\tilde{\Omega}_j(.)}{|.|^2}\star\delta_{0}(x)~~\forall~x\in \mathbb{R}^2\label{deriB14}
	\end{align}
	Indeed, we have 
	\begin{align*}
	&\mathcal{R}_j^2((\chi_{|x|}-\chi_{|x|/2})\nu)(x)\\&=\lim_{\varepsilon\to 0}\int_{\mathbb{R}}\left(K_j^\epsilon(x_1-|y_2|,x_2-y_2)+K_j^\epsilon(x_1+|y_2|,x_2-y_2)\right)(\chi_{|x|}-\chi_{|x|/2})(y_2,y_2)\frac{dy_2}{y_2}\\&= \frac{1}{|x|^2}\lim_{\varepsilon\to 0}\int_{\mathbb{R}}\left(K_j^\epsilon(\theta_1-|y_2|,\theta_2-y_2)+K_j^\epsilon(\theta_1+|y_2|,\theta_2-y_2)\right)\left[\chi\left(\sqrt{2}|y_2|\right)-\chi\left(8\sqrt{2}|y_2|\right)\right]\frac{dy_2}{y_2}\\&:=\frac{\tilde{\Omega}_j(\theta)}{|x|^2},
	\end{align*}
	with $K_j^\epsilon(.)=\mathbf{1}_{|.|>\epsilon}K_j^\epsilon(.), \theta=x/|x|$.\\
	Clearly, $\tilde{\Omega}_j(\theta_1,-\theta_2)=-\tilde{\Omega}_j(\theta_1,\theta_2)$ and $\tilde{\Omega}_j(-\theta_1,\theta_2)=\tilde{\Omega}_j(\theta_1,\theta_2)$, so, $\int_{S^1}\tilde{\Omega}_j=0.$\\ To prove $\tilde{\Omega}_j\in \left(L^\infty\cap BV\right)(S^{1})$, we can 
	assume  that $\theta_2,\theta_1\geq 0$ and $\theta_1\not=\theta_2$. So, we can write 
	\begin{align*}
	\tilde{\Omega}_j(\theta)&=a_j^1(\theta)+a_j^2(\theta)
	\end{align*}
	where 
	\begin{align*}
	&a_j^1(\theta)=\lim_{\varepsilon\to 0}\int_{\frac{1}{4\sqrt{2}}}^{2\sqrt{2}} K_j^\epsilon(\theta_1-y_2,\theta_2-y_2)\left[\chi\left(\sqrt{2}y_2\right)-\chi\left(8\sqrt{2}y_2\right)\right]\frac{dy_2}{y_2}\\&
	a_j^2(\theta)=\int_{-\infty}^0K_j(\theta_1-|y_2|,\theta_2-y_2)\left[\chi\left(\sqrt{2}|y_2|\right)-\chi\left(8\sqrt{2}|y_2|\right)\right]\frac{dy_2}{y_2}\\&+\int_{\mathbb{R}}K_j(\theta_1+|y_2|,\theta_2-y_2)\left[\chi\left(\sqrt{2}|y_2|\right)-\chi\left(8\sqrt{2}|y_2|\right)\right]\frac{dy_2}{y_2}.
	\end{align*}
	Clearly, $
	a_j^1(\theta) \in C^\infty(S^{1}\cap\{\theta_1,\theta_2\geq 0\}\backslash\{(1/\sqrt{2},1/\sqrt{2})\})$
	$a_j^2(\theta)\in C_b^\infty(S^{1}\cap\{\theta_1,\theta_2\geq 0\})$.  Using the fact that 
	\begin{align*}
	\left|\int_{\frac{1}{4\sqrt{2}}}^{2\sqrt{2}} K_j^\epsilon(\theta_1-y_2,\theta_2-y_2)dy_2\right|+\left|\int_{\frac{1}{4\sqrt{2}}}^{2\sqrt{2}} K_j^\epsilon(\theta_1-y_2,\theta_2-y_2)(\theta_2-y_2)dy_2\right|\leq C~\forall \varepsilon>0.
	\end{align*}
	and Taylor's Formula : for any $y_2,\theta_2\geq 0$, 
	\begin{align*}
	|\tilde{f}(y_2)-\tilde{f}(\theta_2)-\tilde{f}'(\theta_2)(y_2-\theta_2)|\lesssim |y_2-\theta_2|^2, \tilde{f}(y_2)=y_2^{-1}\left[\chi\left(\sqrt{2}y_2\right)-\chi\left(8\sqrt{2}y_2\right)\right],
	\end{align*}
	we find that 
	\begin{align*}
	|a_j^1(\theta)|\lesssim1+ \int_{\frac{1}{4\sqrt{2}}}^{2\sqrt{2}} |K_j(\theta_1-y_2,\theta_2-y_2)|\theta_2-y_2|^2dy_2\lesssim 1.
	\end{align*}
	Thus, we obtain that $\tilde{\Omega}_j\in \left(L^\infty\cap BV\right)(S^{1})$, and in particular, $\tilde{\Omega}_j\in \left(L^\infty\cap BV\right)(B_2(0)\backslash B_1(0))$.
	Therefore, we derive from \eqref{deriB11},\eqref{deriB12},\eqref{deriB13} and \eqref{deriB14} that there exist functions $\Omega_1,...,\Omega_m\in \left(L^\infty\cap BV\right)(S^1) $ such that $\Omega_j(\theta)=\Omega_j(r\theta)$ for $\theta\in S^1,r>0$, $\int_{S^1}\Omega_j=0$ and  for any $R>1$ we have $
	\partial_1\mathbf{B}^1=\sum_{j=1}^{m}\left(\frac{\Omega_j(.)}{|.|^2}\right)\star\mu_{jR}~~\text{in}~~\mathcal{D}'(B_R)$
	for some $\mu_{1R},...,\mu_{mR}\in \mathcal{M}_b(\mathbb{R}^2)$. Similarly, we can do this for $\partial_2\mathbf{B}_1,\partial_1\mathbf{B}^2,\partial_2\mathbf{B}^2$. The proof is complete.
\end{proof}
To prove Lemma \ref{lem-es-hol}, we need to have the following result:
\begin{lemma}\label{lem-es-onhyper} Let $e\in S^{d-1}$. For any  $z_1,z_1'\in\tilde{H}_{e},  y_2,y_2',z_2,z_2'\in H_{e}$,  $\varepsilon>0,\rho>0$, there holds
	\begin{align}\nonumber
	M&=	\int_{\tilde{H}_{e}} |f(y_2'+y_1 )-f(y_2'+z_1')|\mathbf{1}_{|(z_1+z_2)-(y_1+y_2)|\leq \varepsilon}~1\wedge\left(\frac{\rho}{|(z_1'+z_2')-(y_1+y_2)|}\right)^{d+2}~ d\mathcal{H}^1(y_1)\\&\nonumber\lesssim \frac{\rho^2}{\varepsilon} \int_{\tilde{H}_{e}}  1\wedge\left(\frac{\rho}{|(z_1'+z_2')-(z+y_2)|}\right)^{d-\frac{1}{2}}d|Df^{e}_{y_2'}|(z)\\&+\rho \int_{\tilde{H}_{e}} \mathbf{1}_{|(z_1+z_2)-(z+y_2)|\leq 4\varepsilon} ~1\wedge\left(\frac{\rho}{|(z_1'+z_2')-(z+y_2)|}\right)^{d+\frac{1}{2}}d|Df^{e}_{y_2'}|(z).\label{es-Nov222}
	\end{align}
\end{lemma}
\begin{proof} Since 
	$|f(y_2'+y_1 )-f(y_2'+z_1')|\leq \int_{\tilde{H}_{e}} \mathbf{1}_{|z-z_1'|\leq 2|z_1'-y_1|}d|Df^{e}_{y_2'}|(z),$
	so, $
	M\leq \int_{\tilde{H}_{e}} Vd|Df^{e}_{y_2'}|(z),$
	where $$
	V= \int_{\tilde{H}_{e}} \mathbf{1}_{|z-z_1'|\leq 2|z_1'-y_1|}1_{|(z_1+z_2)-(y_1+y_2)|\leq \varepsilon}\min\left\{1,\left(\frac{\rho}{|(z_1'+z_2')-(y_1+y_2)|}\right)^{d+2}\right\}d\mathcal{H}^1(y_1).$$\\
	Note that if $|z-z_1'|\leq 2|z_1'-y_1|$, then 
	$|(z_1'+z_2')-(z+y_2)|\leq 4|(z_1'+z_2')-(y_1+y_2)|$ and
	$|(z_1+z_2)-(z+y_2)|\leq |(z_1+z_2)-(y_1+y_2)|+3|(z_1'+z_2')-(y_1+y_2)|$. 
	Thus, we can estimate 
	\begin{align*}
	&V= \int_{\tilde{H}_{e}} \mathbf{1}_{|(z_1'+z_2')-(y_1+y_2)|\leq \varepsilon}+\int_{\tilde{H}_{e}} \mathbf{1}_{|(z_1'+z_2')-(y_1+y_2)|> \varepsilon}\\&\lesssim\mathbf{1}_{|(z_1+z_2)-(z+y_2)|\leq 4\varepsilon} ~1\wedge\left(\frac{\rho}{|(z_1'+z_2')-(z+y_2)|}\right)^{d+\frac{1}{2}}\int_{\tilde{H}_{e}} 1\wedge\left(\frac{\rho}{|z_1'-y_1|}\right)^{\frac{3}{2}}d\mathcal{H}^1(y_1)\\&+
	\frac{\rho}{\varepsilon}~ 1\wedge\left(\frac{\rho}{|(z_1'+z_2')-(z+y_2)|}\right)^{d-\frac{1}{2}}\int_{\tilde{H}_{e}} 1\wedge\left(\frac{\rho}{|z_1'-y_1|}\right)^{\frac{3}{2}}d\mathcal{H}^1(y_1)\\&\lesssim
	\rho \mathbf{1}_{|(z_1+z_2)-(z+y_2)|\leq 4\varepsilon} 1\wedge\left(\frac{\rho}{|(z_1'+z_2')-(z+y_2)|}\right)^{d+\frac{1}{2}}+
	\frac{\rho^2}{\varepsilon} 1\wedge\left(\frac{\rho}{|(z_1'+z_2')-(z+y_2)|}\right)^{d-\frac{1}{2}}
	\end{align*}
	which implies \eqref{es-Nov222}. The proof is complete. 
	%
\end{proof}
\begin{proof}[Proof of Lemma \ref{lem-es-hol}] We first observe that 
	\begin{align*}
	M&\leq \sum_{k=0}^{d-2} \int_{\tilde{H}_{e_1}}...\int_{\tilde{H}_{e_d}}~1\wedge\left(\frac{\rho}{|\sum_{i=1}^{d}(x_i-y_i)|}\right)^{d+2} \mathbf{1}_{|\sum_{i=1}^{d}(y_{0i}-y_i)|\leq \varepsilon}\\&~~~\times \left|f(\sum_{i=1}^{d-k-1}y_i+\sum_{i=d-k+1}^{d}x_i+y_{d-k})-f(\sum_{i=1}^{d-k-1}y_i+\sum_{i=d-k+1}^{d}x_i+x_{d-k})\right|d\mathcal{H}^1(y_d)...d\mathcal{H}^1(y_1).
	\end{align*}
	Applying Lemma \ref{lem-es-onhyper} to $e=e_{d-k}$, $y_2'=\sum_{i=1}^{d-k-1}y_i+\sum_{i=d-k+1}^{d}x_i$, $y_2=\sum_{i\not= d-k}y_i$, $z_1=y_{0,d-k},z_2=\sum_{i\not= d-k} y_{0,i}$,$z_1'=x_{d-k},z_2'=\sum_{i\not= d-k} x_i$ yields
	\begin{align}\nonumber
	M&\lesssim \sum_{k=0}^{d-2} \frac{\rho^2}{\varepsilon}\int_{\tilde{H}_{e_1}}...\int_{\tilde{H}_{e_d}}1\wedge\left(\frac{\rho}{|\sum_{i=0}^{d}(x_i-y_i)|}\right)^{d-\frac{1}{2}}\\&\nonumber~~~\times d\mathcal{H}^1(y_d)...d\mathcal{H}^1(y_{d-k+1})d|Df^{e_{d-k}}_{\sum_{i=1}^{d-k-1}y_i+\sum_{i=d-k+1}^{d}x_i}|(y_{d-k})d\mathcal{H}^1(y_{d-k-1})...d\mathcal{H}^1(y_1)\\&+\sum_{k=0}^{d-2} \rho\int_{\tilde{H}_{e_1}}...\int_{\tilde{H}_{e_d}} 1\wedge\left(\frac{\rho}{|\sum_{i=0}^{d}(x_i-y_i)|}\right)^{d+\frac{1}{2}} \mathbf{1}_{|\sum_{i=1}^{d}(y_{0i}-y_i)|\leq \varepsilon}\nonumber\\&~~~\times d\mathcal{H}^1(y_d)...d\mathcal{H}^1(y_{d-k+1})d|Df^{e_{d-k}}_{\sum_{i=1}^{d-k-1}y_i+\sum_{i=d-k+1}^{d}x_i}|(y_{d-k})d\mathcal{H}^1(y_{d-k-1})...d\mathcal{H}^1(y_1).\label{es-nov231}
	\end{align}
	It is clear to see that
	$$\int_{\tilde{H}_{e_{d-k+1}}}...\int_{\tilde{H}_{e_d}}1\wedge \left(\frac{\rho}{|\sum_{i=0}^{d}(x_i-y_i)|}\right)^{d-\frac{1}{2}} d\mathcal{H}^1(y_d)...d\mathcal{H}^1(y_{d-k+1})\lesssim \rho^k~ 1\wedge \left(\frac{\rho}{|\sum_{i=1}^{d-k}(x_i-y_i)|}\right)^{d-k-\frac{3}{4}},$$
	and
	\begin{align*}
	&\int_{\tilde{H}_{e_{d-k+1}}}...\int_{\tilde{H}_{e_d}}1\wedge\left(\frac{\rho}{|\sum_{i=0}^{d}(x_i-y_i)|}\right)^{d+\frac{1}{2}} ~\mathbf{1}_{|\sum_{i=1}^{d}(y_{0i}-y_i)|\leq \varepsilon} d\mathcal{H}^1(y_d)...d\mathcal{H}^1(y_{d-k+1})\\&\lesssim \rho^k\mathbf{1}_{|\sum_{i=1}^{d}(y_{0i}-x_i)|\leq 2\varepsilon}\mathbf{1}_{|\sum_{i=1}^{d-k}(y_{0i}-y_i)|\leq 2\varepsilon} 1\wedge\left(\frac{\rho}{|\sum_{i=1}^{d-k}(x_i-y_i)|}\right)^{d-k+\frac{1}{4}}\\&~~+
	\frac{\rho^{k+1}}{\varepsilon}~1\wedge \left(\frac{\rho}{|\sum_{i=0}^{d-k}(x_i-y_i)|}\right)^{d-k-\frac{3}{4}}.
	\end{align*}
	Combining these with \eqref{es-nov231} we find that 
	\begin{align*} 
	&M\lesssim \sum_{k=0}^{d-2} \frac{\rho^{k+2}}{\varepsilon}\int_{\tilde{H}_{e_1}}...\int_{\tilde{H}_{e_{d-k}}}~1\wedge \left(\frac{\rho}{|\sum_{i=1}^{d-k}(x_i-y_i)|}\right)^{d-k-\frac{3}{4}} d\nu^1_{k,\sum_{i=d-k+1}^{d}x_i}(y_{d-k},...,y_1)\\&~~+\sum_{k=0}^{d-2} \rho^{k+1}\mathbf{1}_{|\sum_{i=1}^{d}(y_{0i}-x_i)|\leq 2\varepsilon}\int_{\tilde{H}_{e_1}}...\int_{\tilde{H}_{e_{d-k}}}1\wedge \left(\frac{\rho}{|\sum_{i=1}^{d-k}(x_i-y_i)|}\right)^{d-k+\frac{1}{4}} d\nu^2_{k,\sum_{i=d-k+1}^{d}x_i}(y_{d-k},..,y_1).
	\end{align*}
	Hence, using the fact that for any $\omega\in \mathcal{M}^+(\bigotimes_{i=1}^{d-k} \tilde{H}_{e_i})$,
	\begin{align*}
	\int_{\tilde{H}_{e_1}}...\int_{\tilde{H}_{e_{d-k}}}1\wedge \left(\frac{\rho}{|\sum_{i=1}^{d-k}(x_i-y_i)|}\right)^{d-k+\frac{1}{4}}d\omega(y_{d-k},...,y_1)\lesssim\rho^{d-k}\mathbf{M}(\omega,\bigotimes_{i=1}^{d-k} \tilde{H}_{e_i})(\sum_{i=1}^{d-k}x_i),
	\end{align*}
	one gets the first inequality of Lemma \ref{lem-es-hol}. Similarly, we also have second one. The proof is complete.
\end{proof}

								 \providecommand{\bysame}{\leavevmode\hbox to3em{\hrulefill}\thinspace}
								 \providecommand{\MR}{\relax\ifhmode\unskip\space\fi MR }
								 \providecommand{\MRhref}[2]{%
								 	\href{http://www.ams.org/mathscinet-getitem?mr=#1}{#2}
								 }
								 \providecommand{\href}[2]{#2}

\end{document}